\documentclass[11pt,a4paper]{article}
\usepackage{geometry,amsmath,amssymb,amsthm,graphicx,epsfig,float,cite}
\usepackage{enumerate}
\usepackage{multirow}
\usepackage{lscape}
\usepackage[graph]{xy}
\usepackage[colorlinks=true,linkcolor=blue,citecolor=blue]{hyperref}
\usepackage{mathrsfs}
\usepackage{caption}
\usepackage{subcaption}
\usepackage{psfrag}
\usepackage{todonotes}

\newcommand{\assign}{:=}

\newcommand{\dR}{\mathbb{R}}  
\newcommand{\ve}{\varepsilon}     

\newcommand{\ds}{\displaystyle}
\theoremstyle{plain}
\newtheorem{theorem}{Theorem}
\newtheorem{fact}{Fact}
\newtheorem{lemma}{Lemma}
\newtheorem{proposition}{Proposition}
\newtheorem{remark}{Remark}

\newtheorem{definition}{Definition}
\theoremstyle{definition}

\usepackage{booktabs}

\title{A Generalized Multivariable Newton Method}
\author{Regina S. Burachik \and Bethany I. Caldwell \and C. Yal\c{c}{\i}n Kaya}

\begin{document}
\maketitle

\begin{abstract} \noindent
It is well known that the Newton method may not converge when the initial guess does not belong to a specific quadratic convergence region. We propose a family of new variants of the Newton method with the potential advantage of having a larger convergence region as well as more desirable properties near a solution. We prove quadratic convergence of the new family, and provide specific bounds for the asymptotic error constant. We illustrate the advantages of the new methods by means of test problems, including two and six variable polynomial systems, as  well as a challenging signal processing example.  We present a numerical experimental methodology which uses a large number of randomized initial guesses for a number of methods from the new family, in turn providing advice as to which of the methods employed is preferable to use in a particular search domain. 
\end{abstract}



\pagestyle{myheadings}
\markboth{}{\sf\scriptsize A Generalized Multivariable Newton Method \ \ by R. S. Burachik, B. I. Caldwell and C. Y. Kaya}

\section{Introduction}

Newton's method and its variants are a fundamental tool for solving nonlinear equations. Namely, given a  ${\cal C}^1$-function $f:\dR^n\to \dR^n$, Newton's method is designed to converge iteratively to a solution of the problem 
\begin{equation}\label{NLE-pro}
    f(x)={\bf 0}\,,\quad x\in \dR^n.
\end{equation}
Problem \eqref{NLE-pro} arises in practically every pure and applied discipline, including mathematical programming, engineering, physics, health sciences and economics. As a result, studies of Newton's method form an extremely active area of research, with new variants being constantly developed and tested. Basic results on Newton’s method and comprehensive lists of references can be found, e.g.,  in the books by Dennis and Schnabel \cite{DenSch1983}, Ostrowski \cite{Ostrowski},  Ortega and Rheinboldt \cite{Ortega} and Deuflhard \cite{Deulfhard2011}. The interested reader will find an excellent survey of Newton's method in \cite{Polyak2007}. 

When started at an initial guess close to a solution, Newton's method is well-defined and converges quadratically to a solution of \eqref{NLE-pro}, unless the Jacobian of $f$ is singular and the second partial derivatives of $f$ are not bounded.  Hence, if the user has an idea of where a solution might be lying, Newton's method is well-known to be the fastest and most effective method for solving \eqref{NLE-pro}.  

To ensure {\em global convergence} (i.e., to ensure convergence to a solution from any initial point),  suitable modifications of the Newton method are needed. An example of a globally convergent variant is the so-called {\em Levenberg-Marquardt method} \cite{Levenberg,Marquardt}.  This involves a modification of Newton's search direction at each step of the method.  Without this modification, however, quadratic convergence can only be ensured when the initial guess belongs to a {\em quadratic convergence region}, namely a region from which every starting point generates a quadratically convergent Newton sequence.  

For a given solution $x^*$ of \eqref{NLE-pro}, a quadratic convergence region is, in general, not known {\em a priori}. Kantorovich \cite{Kantorovich1948} and Smale \cite{Smale1986} establish quadratic convergence to a solution of \eqref{NLE-pro} under suitable assumptions on the initial guess, and they do so without modifying Newton's original search direction.  Our aim, on the other hand, is to devise a suitable modification of the Newton iteration, so that the quadratic convergence region is (ideally) larger than the one resulting from the classical Newton iterations. Our approach can, in some sense, be seen as a {\em preconditioning} of the iterations, so that a larger convergence region is obtained. This preconditioning might be helpful when very little is known about the location of the solutions of \eqref{NLE-pro}.

In~\cite{BuraKaya2012}, the authors presented a generalized version of the classical univariate Newton iteration in which the original problem \eqref{NLE-pro} is replaced by a ``modified" system (for $n=1$)
\begin{equation}\label{NLE-pro1}
    f\circ s^{-1}(x)=0\,,
\end{equation}
where $s:\dR\to \dR$ is a ${\cal C}^1$-invertible function in a neighbourhood of the solution. The classical Newton method then corresponds to the choice $s(x)=x$. By judiciously choosing $s$ in a way that relates to the nature of Problem \eqref{NLE-pro1}, the authors illustrate in \cite{BuraKaya2012} via numerical experiments that the region of quadratic convergence can be enlarged, and hence a wider choice of initial guesses are likely to result in quadratic convergence.

In the present paper, we propose a multivariate version of the generalized Newton method proposed in \cite{BuraKaya2012}. We establish the quadratic convergence under suitable assumptions, and test this new method in our numerical experiments. For suitable choices of $s$, we illustrate via extensive numerical experiments that the region of convergence corresponding to the new method may be larger than the one observed for the classical Newton iteration.  

Recall that, if a sequence $(x^k)$ converges to $x^*$ (with $x^k\neq x^*,\,\forall k$), it is said to converge quadratically to $x^*$   whenever we have that
\[
\lambda:=\lim_{k\to\infty}\dfrac{\|x^{k+1}-x^*\|}{\|x^{k}-x^*\|^2}<\infty\,,
\]
where $\lambda$ denotes the so-called {\em asymptotic error constant} (see Definition \ref{def:convSpeed}). Moreover, the smaller $\lambda$ is, the faster the convergence will be. 

With different choices of $s$, the value of $\lambda$ will also be different, in general. Our generalized Newton methods provide a tool for enlarging the region where $\lambda=\lambda(s)$ is finite. Moreover, a suitable choice of $s$ might produce a smaller value $\lambda(s)$, thus resulting in an improvement of the convergence speed.  We illustrate this phenomenon in Section \ref{ss:AEC}. 

The choice of a suitable function $s$ is, however, not clear in general, and more studies are needed to develop a systematic way of designing such choices. An appropriate choice of $s$ may result in a more robust behaviour of the generalized Newton method over a larger search domain, as can be observed in the numerical experiments we carry out in Section \ref{sec:NE}.

The present paper is organized as follows. In Section \ref{sec:prel}, we state the basic definitions, useful remarks, and properties. In Section \ref{sec:conv}, we establish the quadratic convergence results for the generalized Newton's method.  In Section~\ref{ss:AEC}, we establish bounds on the asymptotic error constant. In Section \ref{sec:NE}, we test and compare the classical and a number of generalized methods for example problems with two and six variables, one of the problems being a challenging signal processing example. In the last section, the conclusion and a discussion are presented.

\section{Preliminaries}
\label{sec:prel}

We present first the main definitions and assumptions. 

\begin{definition}\label{def:GradJacHess} \rm
Let $x\in\mathbb{R}^n$ and let $g:\mathbb{R}^n\rightarrow\mathbb{R}^n$ be twice continuously differentiable. We write 
\[
g(x)=\left[\begin{array}{c}
     g_1(x)  \\
     \vdots\\
     g_n(x)
\end{array}\right]\in \dR^n \hbox{ and, for each }i=1,\ldots,n,\hbox{ we have } \nabla g_i(x) := \left[\begin{array}{c}
    \ds\frac{\partial g_i}{\partial x_1}(x)  \\
     \vdots\\
    \ds\frac{\partial g_i}{\partial x_n}(x) 
\end{array}\right]\in \dR^n,
\]
where the vector $\nabla g_i(x)$ is called the {\em gradient of $g_i$ at $x$} for every $i=1,\ldots,n$. The {\em Jacobian of $g$ at $x$}, denoted by $J_g(x)$, is the $n\times n$ matrix which
has for row $i$ the (transpose of the) gradient of each $g_i$, for $i=1,\ldots,n$. More precisely, $J_g(x)$ is defined as
 \[
 J_g(x):=
 \left[\begin{array}{cccc}
 \ds\frac{\partial g_1}{\partial x_1}(x) & \ds\frac{\partial g_1}{\partial x_2}(x) & \cdots & \ds\frac{\partial g_1}{\partial x_n}(x) \\ \vdots & \vdots &  & \vdots \\ \ds\frac{\partial g_n}{\partial x_1}(x) & \ds\frac{\partial g_n}{\partial x_2}(x) & \cdots & \ds\frac{\partial g_n}{\partial x_n}(x) \end{array}\right]\in\mathbb{R}^{n\times n}.\]
 
For each $i=1,\ldots,n$, the {\em Hessian of $g_i$ at $x$} is denoted by $\nabla^2 g_i(x)$ and defined as
 
 \[\nabla^2 g_i(x):=\begin{bmatrix} \ds\frac{\partial^2 g_i}{\partial x_1^2}(x) & \ds\frac{\partial^2 g_i}{\partial x_1\partial x_2}(x) & \cdots & \ds\frac{\partial^2 g_i}{\partial x_1\partial x_n}(x) \\ \vdots & \vdots &  & \vdots \\ \ds\frac{\partial^2 g_i}{\partial x_n\partial x_1}(x) & \ds\frac{\partial^2 g_i}{\partial x_n\partial x_2}(x) & \cdots & \ds\frac{\partial^2 g_i}{\partial x_n^2}(x) \end{bmatrix}\in\mathbb{R}^{n\times n},\]
 for $i=1,2,\dots,n$.
 \end{definition}
 
 In what follows, we denote by $\|\cdot\|$ the Euclidean norm, i.e., the $\ell_2$-norm in $\dR^n$. Denote by $B(x,r):=\{z\in \dR^n\::\: \|z-x\|< r\}$ the open ball centered at $x$ with radius $r$. Similarly, denote  by $B[x,r]:=\{z\in \dR^n\::\: \|z-x\|\le  r\}$ the closed ball centered at $x$ with radius $r$. Given a matrix $A\in \dR^{n\times n}$, recall that a {\em norm of $A$}, denoted as $\|A\|$, can be given by $\|A\|:= \max\{\|Ax\|\::\: \|x\|\le 1\}$, often referred to as the $\ell_2$-norm of $A$.
 
 The following simple lemma will be used in the proof of Proposition \ref{prop:multivarLambdaProof}.

\begin{lemma}\label{lem:Tg}
Let $A_1,\ldots,A_n\in \dR^{n\times n}$ and fix $u\in \dR^n$.
Consider the vectors
\[
v:=\left[
\begin{array}{c}
   u^T\,A_1\,u    \\
    \vdots \\
     u^T\,A_n\, u  
\end{array}
\right]\in \dR^n,\hbox{ and } w:=\left[
\begin{array}{c}
   \|A_1\|    \\
    \vdots \\
     \|A_n\|  
\end{array}
\right] \in \dR^n,
\]
where $\|A_i\|$ is the $\ell_2$-norm of the matrix $A_i$ for $i=1,\ldots,n$.
Then $\|v\|\le \|u\|^2 \|w\|$. 

\end{lemma}
\begin{proof}
By Cauchy--Schwartz and the definition of the norm,  $u^T\,A_i\,u \le \|u\|^2 \|A_i\|$. Hence,
\[
\begin{array}{rcl}
\|v\|^2&=& \sum_{i=1}^{n} (u^T\,A_i\,u)^2\le  \sum_{i=1}^{n} (\|u\|^2\|A_i\|)^2\\
&&\\
&= & \|u\|^4 \sum_{i=1}^{n}\|A_i\|^2 =\|u\|^4 \|w\|^2,
\end{array}
\]
which yields the statement.
\end{proof}

We recall next some standard definitions and notation we will use in our analysis.

\begin{definition}\label{def:homeo} \rm
Let $A\subset \dR^n$ and $B\subset \dR^m$. Let $h:A\to B$ and fix $D\subset A$.
\begin{itemize}
\item[(a)] We say that $h\in {\cal C}^0(D)$ if $h$ is continuous at every $x\in D$.
      
\item[(b)] We write  $h\in {\cal C}^1(D)$ if $h$ is continuously differentiable at every $x\in D$. Equivalently, $J_h(\cdot):D\to \dR^{m\times n}$ is a continuous function of $x$, for every $x\in D$.

\item[(c)] Assume that $m=n$ and  $h:A\rightarrow B$ is bijective. So there exists $h^{-1}:B\to A$ with $B=h(A)$. If $h\in {\cal C}^0(A)$ and $h^{-1}\in {\cal C}^0(B)$, we say that $h$ is a {\em homeomorphism from $A$ to $B$}. Moreover, if $h\in {\cal C}^1(A)$ and $h^{-1}\in {\cal C}^1(B)$, we say that $h$ is a ${\cal C}^1$-{\em homeomorphism from $A$ to $B$}.

\item[(d)] Fix $\beta>0$. We say that $h\in{\rm Lip}_{\beta}(D)$ if we have
\[
\| h(x)-h(x')\|\le \beta \|x-x'\|,
\]
for every $x,x'\in D$.
\end{itemize}
 \end{definition}
 
\begin{remark}\label{rem:homeo}\rm
 Let $D\subset \dR^n$ be an open set and let $s:\mathbb{R}^n\rightarrow\mathbb{R}^n$ be a ${\cal C}^1$-homeomorphism from $D$ to $s(D)$. Then for every $x\in D$, we have that $J_s(x)\in \mathbb{R}^{n\times n}$ is invertible  and 
 \[
 [J_s(x)]^{-1}=J_{s^{-1}}(s(x)).
 \]
Indeed, given $x\in D$, there is a unique $y\in s(D)$ such that $y=s(x)$. Differentiate both sides of the equality $s\circ s^{-1}(y)=y$ to derive
 \begin{equation}\label{eq:inv}
I=J_s(s^{-1}(y))J_{s^{-1}}(y)= J_s(x)J_{s^{-1}}(s(x)),
\end{equation}
 where we used the Chain Rule. Now \eqref{eq:inv} directly yields the claim.
 \end{remark}

\begin{remark}\label{rem:Jacobian}\rm
Let $D\subset \dR^n$ be an open set and consider two functions $s,f:\mathbb{R}^n\rightarrow\mathbb{R}^n$ such that $s$ is a ${\cal C}^1$-homeomorphism from $D$ to $s(D)$, and that $f\in {\cal C}^1(D)$. Define $F:=f\circ s^{-1}:s(D)\rightarrow f(D)$. By the Chain Rule and Remark~\ref{rem:homeo}, we have, for every $x\in D$,
\begin{equation}\label{eq:inv2}
   J_F(s(x))=J_f(x)J_{s^{-1}}(s(x))=J_f(x)[J_s(x)]^{-1}.
\end{equation}
In particular, if $J_f(x)$ is invertible, we obtain
\[
    [J_F(s(x))]^{-1}=J_s(x)[J_f(x)]^{-1},
\]
for every $x\in D$. 
 \end{remark}

Let $g:\dR^n\to \dR^n$, so $g(x):=(g_1(x),\ldots, g_n(x))^T$. For each $i=1,\ldots,n$, the gradient and Hessian of $g_i$ collect the first- and second-order information, respectively, of $g_i$. The Jacobian, on the other hand, collects in a single operator all the first-order information for all coordinates of $g$. Similarly, we will need an operator that encapsulates all second order information for all coordinates of $g$. We formally introduce these operators next. 
 
\begin{definition}\label{def:complicatedFunctions} \rm
Let $g:\mathbb{R}^n\rightarrow\mathbb{R}^n$ be twice continuously differentiable. With the notation of Definition \ref{def:GradJacHess}, define the function $T_g:(\mathbb{R}^n)^n\rightarrow\mathbb{R}^{n\times(n\times n)}$ as
\[
T_g({z^1},\dots,{z^n}):=\begin{bmatrix} \nabla^2 g_1({z}^1) \\ \vdots \\ \nabla^2 g_n({z}^n) \end{bmatrix},
\]
where $z^j\in \mathbb{R}^n$ for $j=1,\ldots,n$. Given $n$ vectors $z^1,\dots,z^n\in\mathbb{R}^n$, define the map $T_g({z}^1,\dots,{z}^n): (\mathbb{R}^n)^n \to \dR^n$ as
\[
T_g({z}^1,\dots,{z}^n)_{(u^1,\dots,u^n)} :=
\begin{bmatrix} {(u^1)}^T \nabla^2 g_1({z}^1)u^1 \\
\vdots \\
{(u^n)}^T \nabla^2 g_n({z}^n)u^n
\end{bmatrix} \in\mathbb{R}^{n},
\]
where $u^j\in \mathbb{R}^n$ for $j=1,\ldots,n$. Finally, define the following norm-like concept for the map   $T_g({z}^1,\dots,{z}^n)$:
\[
\|T_g({z}^1,\dots,{z}^n)\|:=\sqrt{\sum_{i=1}^n\|\nabla^2 g_i({z}^i)\|^2}\,,
\]
where the norms in the right hand-side are the $\ell_2$-norms of the Hessians of the $g_i$'s. When $z^i=z^j=z$ for every $i,j\in \{1,\ldots,n\}$, we use the short-hand notation
\[
T_g({z},\dots,{z})=:T_g({z}).
\]
\end{definition}

\begin{definition}\label{def:spectral} \rm
Given a symmetric matrix $A\in \dR^{n\times n}$, denote its set of eigenvalues by $\Sigma(A)$. Recall that, when the matrix norm is induced by the $\ell_2$-norm we have that $\|A\|=\max\{|\lambda|: \lambda\in \Sigma(A)\}=:{\rm SR}(A)$, the {\em spectral radius of $A$}. 
Denote by $\lambda_{\rm min}(A)$ the minimum eigenvalue of $A$, and by $\lambda_{\rm max}(A)$ the maximum eigenvalue of $A$.
\end{definition}

\begin{remark}\label{rem:Tg-lambda}\rm
Definition \ref{def:spectral} and the fact that ${\rm SR}(\nabla^2 g_i({z}^i))=\|\nabla^2 g_i({z}^i)\|$ for $i=1,\ldots,n$ directly yields
\[
\|T_g({z}^1,\dots,{z}^n)\|= \sqrt{\sum_{i=1}^n({\rm SR}(\nabla^2 g_i({z}^i)))^2}.
\]
\end{remark}

For future use, we now give an elementary fact. 

\begin{fact}\label{fact:elem}
Assume that $a\le b$ and $q\in [a,b]$. Denote by $c:=\max\{|a|, |b|\}$. The following hold. \begin{itemize}
\item[(i)] If $a\ge 0$ then $a^2\le q^2 \le b^2$
\item[(ii)] If $b\le 0$ then $b^2\le q^2 \le a^2$
\item[(iii)] If $a< 0< b$ then $0\le q^2 \le c^2$.
\end{itemize}
\end{fact}

The concepts of rate of convergence and the asymptotic error constant will have an important role in our analysis, so we recall their definitions next.

\begin{definition} \label{def:convSpeed} \rm
Consider a method that generates a sequence $(x^k)\subseteq \mathbb{R}^n$ such that the sequence converges to ${x}^*$, where ${x}^k\neq {x}^*,\,\forall k$. If $\alpha>0$ and $\lambda>0$ with
\[
\lim_{k\rightarrow\infty}\frac{\|{x}^{k+1}-{x}^*\|}{\|{x}^k-{x}^*\|^\alpha} = \lambda\,,
\]
then $(x^k)$ is said to {\em converge to $x^*$ with order $\alpha$} and {\em asymptotic error constant $\lambda$}. When $\alpha=2$, we say that the method converges {\em quadratically}. 
\end{definition}

\subsection{Main Assumptions}
The following are our main assumptions for establishing quadratic convergence. We follow the analysis from \cite{DenSch1983}.
\begin{itemize}
\item[($H_0$)] Problem \eqref{NLE-pro} has a solution, denoted by $x^*$.  There exists $r>0$ such that\linebreak $f\in{\cal C}^1(B(x^*,r))$. Denote $D_1\assign B(x^*,r)$ throughout.

\item[($H_1$)] $J_f(x)$ is nonsingular for all $x\in D_1$ and there exists $\gamma_1>0$ such that
\begin{equation}\label{H1b}
\| J_f(x)^{-1}\|\le \gamma_1,\hbox{ for all }x\in D_1\,.
\end{equation}

\item[($H_2$)] The function $s$ is a ${\cal C}^1$-homeomorphism from $D_1$ to $s(D_1)$ and there exists $\gamma_2>0$ such that
\begin{equation}\label{H1c}
\| J_s(x)\|\le \gamma_2,\hbox{ for all }x\in D_1\,.
\end{equation}

\item[($H_3$)] Denote $D_2:= s(D_1)$. There exist $\beta_1,\beta_2>0$ such that $J_f\in\textrm{Lip}_{\beta_1}(D_1)$ and $J_{s^{-1}}\in\textrm{Lip}_{\beta_2}(D_2)$.
\end{itemize}

Given a set $A\subset \dR^n$, we denote by ${\rm cl}(A)$ the closure of the set $A$.

\begin{remark}\label{rem:M1M2}\rm
Assumption $(H_3)$ implies the existence of $M_1,M_2>0$ such that
\[
\|J_f(x)\|\leq M_1, \,\,\hbox{ and }\,\, \|J_{s^{-1}}(s(x))\|=\|[J_s(x)]^{-1}\|\leq M_2\,.
\]
Indeed, this follows from the fact that the mappings are continuous over the compact sets ${\rm cl}(D_1)$ and ${\rm cl}(D_2)$, respectively. 
\end{remark}

\begin{remark}\label{rem:H3}\rm
Assumption $(H_3)$ allows us to apply Lemma 4.1.16 in \cite{DenSch1983} to deduce that there exists $\epsilon>0, L_0,L_1>0$ such that
\[
    L_1\|z-z'\|\leq\left\|s^{-1}(z)-s^{-1}(z')\right\|\leq L_0\|z-z'\|
\]
for all $z,z'\in B(s(x^*),\epsilon)\cap s(D_1)$. Setting $z=s(x),z'=s(x')$ this implies
\[
    L_1\|s(x)-s(x')\|\leq\|x-x'\|\leq L_0\|s(x)-s(x')\|\,.
\]
\end{remark}

\noindent The authors of \cite{BuraKaya2012} proposed a generalized Newton method for solving~\eqref{NLE-pro} for $n=1$. This method can be described by the following iterative formula:
\begin{equation}\label{NM-1}
     g(x^k) = x^{k+1} = s^{-1}\left(s(x^k)-s'(x^k)\frac{f(x^k)}{f'(x^k)}\right),
\end{equation}
where $s:\dR\to \dR$ is a ${\cal C}^1$-invertible function in a neighbourhood of the solution. As mentioned in the Introduction, this modification can be seen as a preconditioning of the problem~\eqref{NLE-pro}, in which the choice of a suitable function $s$ can improve/enlarge the region of convergence of the method. Next we extend the above approach in a natural way to higher dimensions. Namely, for a ${\cal C}^1$-function $f:\dR^n\to \dR^n$, consider the problem of solving
\begin{equation}\label{NLE-pro-n}
f(x)=0,\quad x\in \dR^n.
\end{equation}
In order to solve problem \eqref{NLE-pro-n}, we replace the derivatives in \eqref{NM-1} by the corresponding Jacobians. More precisely, consider the function $g:\dR^n\to \dR^n$ defined by
\begin{equation}\label{eqn:multiGnewtonIter}
  g(x):=s^{-1}\left(s(x)-J_s(x)[J_f(x)]^{-1}f(x)\right),
\end{equation}
where $s:\mathbb{R}^n\rightarrow\mathbb{R}^n$ has an inverse $s^{-1}$, and $J_s({x})$, $J_f({x})$ are the (assumed nonsingular) Jacobians of $s$ and $f$ at $x$, respectively. It can be directly checked that a fixed point $x^*$ of $g$ is a solution of \eqref{NLE-pro-n}, as long as both Jacobians are invertible at $x^*$. The Generalized Newton iteration is obtained by the rule $g(x^k)=x^{k+1}$, where $g$ is the function defined in \eqref{eqn:multiGnewtonIter}.

\begin{definition}  \label{def:itGN} \rm
Assume that $(H_0)-(H_2)$ hold. Given $x^k\in D_1$, define
\begin{equation}\label{it:GN}
x^{k+1} := s^{-1}\left( s(x^k)-J_s(x^k)[J_f(x^k)]^{-1}f(x^k)\right).
\end{equation}
\end{definition}

In the next section we extend  the standard quadratic convergence results to the sequence defined by \eqref{it:GN}.

\section{Convergence of the Multivariate Generalized Newton Method}
\label{sec:conv}
\setcounter{lemma}{0}
\setcounter{theorem}{0}

The following is Lemma 4.1 from \cite{BuraKaya2012} rewritten for the multivariate case. This lemma states that the iteration \eqref{it:GN} coincides with the classical Newton iteration for the composite function $F:=f\circ s^{-1}$. 

\begin{lemma}\label{lem:transformation}
With the notation and hypothesis of Definition \ref{def:itGN}, let $y^k=s(x^k)$ and $F := f\circ s^{-1}$. The iteration \eqref{it:GN}  can be written as
\begin{equation} \label{lem:composite}
y^{k+1}=y^k - [J_F(y^k)]^{-1}F(y^k)\,.
\end{equation}
\end{lemma}
\begin{proof}

Using the definitions of $y^k$ and $y^{k+1}$, and Remark \ref{rem:Jacobian}, we can write the iteration in~\eqref{lem:composite} as
\[
    s(x^{k+1}) = s(x^k)-J_s(x^k)[J_f(x^k)]^{-1}f(x^k)\,.
\]
Now applying $s^{-1}$ to this equality yields the iteration \eqref{it:GN}.
\end{proof}

We will use \cite[Theorem 5.2.1]{DenSch1983} which we quote next.

\begin{theorem}\label{thm:quadconv}
Let $F:\dR^n\rightarrow\dR^n$ be such that $F\in {\cal C}^1(D)$, for an open convex set $D\subseteq\dR^n$. Assume that there exists $\tilde{x}\in\dR^n$ such that $F(\tilde{x})=0$,
and $\delta,\beta,\gamma>0$ such that the following hold.
\begin{itemize}
    \item[(i)] $B(\tilde{x},\delta)\subseteq D$, with $ J_F(\tilde{x})$ invertible and $\|[J_F(\tilde{x})]^{-1}\|\leq\beta$.
    \item[(ii)]  $J_F\in\textrm{Lip}_\gamma(B(\tilde{x},\delta))$ (i.e., for all $x,y\in B(\tilde{x},\delta), \|J_F(x)-J_F(y)\|\leq\gamma\|x-y\|$).
\end{itemize}
Then, there exists $\epsilon>0$ such that for all $\tilde{x}_0\in B(\tilde{x},\epsilon)$, the sequence $(\tilde{x}^k)$ generated by the rule
\begin{equation}\label{eq_D}
    \tilde{x}^{k+1} = \tilde{x}^k-[J_F(\tilde{x}^k)]^{-1}F(\tilde{x}^k),
\end{equation}
 for all $k\geq 0$, has the following properties:
\begin{itemize}
    \item[(a)] $(\tilde{x}^k)$ is well defined (i.e. $[J_F(\tilde{x}^k)]^{-1}$ exists for all $k\geq 0$).
    \item[(b)] The convergence to $\tilde{x}$ is quadratic, namely,
\[
    \|\tilde{x}^{k+1}-\tilde{x}\|\leq\beta\gamma\|\tilde{x}^k-\tilde{x}\|^2\,,
\]
for all $k\geq 0$.
\end{itemize}

\end{theorem}
 
Next we show that we can apply this theorem to our setting for a suitable choices of $F,\,\tilde{x}$ and $D$.

\begin{lemma}\label{lem:QC}
Assumptions ($H_0$)--($H_3$) imply that Conditions (i)--(ii) in Theorem \ref{thm:quadconv} hold for $F:=f\circ s^{-1}$, $\tilde{x}:=s(x^*)$ and $D:=D_1$. Consequently, there exists $\ve>0$ such that, for $y^0\in B(s(x^*),\ve)$, the sequence 
\[
y^{k+1}:=y^k-[J_{F}({y}^k)]^{-1}F({y}^k),
\]
for all $k\geq 0$, has the following properties:
\begin{itemize}
\item[(a)] $(y^k)$ is well defined (i.e. $[J_F(y^k)]^{-1}$ exists for all $k\geq 0$).
\item[(b)] The convergence to $s(x^*)$ is quadratic, namely,
\[
\|{y}^{k+1}-s(x^*)\|\leq \eta\|y^k-s(x^*)\|^2,\quad\forall{k\ge 0}\,,
\]
with $\eta:=\gamma_1\gamma_2 L_0(\dfrac{M_1\beta_2}{L_1}+M_2\beta_1)$.
\end{itemize}
\end{lemma}

\begin{proof}
Note that the statements (a) and (b) involving the sequence $(y^k)$ will follow directly from Theorem \ref{thm:quadconv} once we establish Conditions (i)--(ii) for suitable constants.  Therefore, we proceed to prove (i) and (ii). By $(H_0)$ and the definitions of $F$ and $\tilde{x}$, we can write
\[
    F(\tilde{x})=f\circ s^{-1}(s(x^*))=f(x^*)=0.
\]
By $(H_2)$ and $(H_3)$ we have that $D_2=s(D_1)$ is an open neighbourhood of $\tilde{x}=s(x^*)$. Hence we can take $\delta>0$ such that $B(\tilde x, \delta)\subset D_2= s(B(x^*,r))$. Using Remark \ref{rem:Jacobian} we can write
\begin{equation}\label{eqn:J_F}
    J_F(\tilde{x})=J_{f\circ s^{-1}}(s(x^*))=J_f(x^*)[J_s(x^*)]^{-1}.
\end{equation}
By ($H_1$)--($H_2$) and Remark \ref{rem:homeo}, we deduce that the matrix on the right hand side of \eqref{eqn:J_F} is nonsingular and hence $J_F(\tilde{x})$ is nonsingular. Using Remark \ref{rem:homeo} again gives
\begin{equation}\label{eqn:bound}
    \left\|[J_F(\tilde{x})]^{-1}\right\|=\left\|J_s(x^*)[J_f(x^*)]^{-1}\right\|\leq\|J_s(x^*)\|\left\|[J_f(x^*)]^{-1}\right\|\leq\gamma_1\gamma_2,
\end{equation}
where we used ($H_1$)--($H_2$) in the last inequality. The expressions  \eqref{eqn:J_F}-\eqref{eqn:bound} imply that condition (i) in Theorem \ref{thm:quadconv} holds for $\tilde{x}\in B(\tilde x, \delta)$ with $\beta:=\gamma_1\gamma_2$. 

Next, we check Condition~(ii) in Theorem \ref{thm:quadconv}  for $J_F=J_{f\circ s^{-1}}$. Namely, we show now that there exists $\gamma>0$ such that $J_{f\circ s^{-1}}\in\textrm{Lip}_\gamma(B(\tilde{x},\delta))$. By $(H_2)$, given $z,z'\in D_2=s(D_1)$ there exist unique $x,x'\in D_1=B(x^*,r)$ such that $z=s(x), z'=s(x')$. Adding and subtracting a suitable term and using Remark \ref{rem:Jacobian} we obtain
\begin{equation}\label{eq_A}
\begin{array}{l}
\|J_{f\circ s^{-1}}(z)-J_{f\circ s^{-1}}(z')\|\\
\\
\qquad = \| J_{f\circ s^{-1}}(z)-J_f(x)[J_s(x')]^{-1} + J_f(x)[J_s(x')]^{-1} -J_{f\circ s^{-1}}(z')\|
\\
\\
  \qquad  \le \left\|J_f(x)\left([J_s(x)]^{-1}-[J_s(x')]^{-1}\right)\right\|+\left\|(J_f(x)-J_f(x'))[J_s(x')]^{-1}\right\| \\
    \\
\qquad \leq\|J_f(x)\|\left\|J_{s^{-1}}(s(x))-J_{s^{-1}}(s(x'))\right\|+\left\|[J_s(x')]^{-1}\right\|\|J_f(x)-J_f(x')\|.
\end{array}
\end{equation}
By $(H_3)$ we know that $J_{s^{-1}}\in\textrm{Lip}_{\beta_2}(D_2)$ and hence for every $x,x'\in D_1$ we have
\begin{equation}\label{eq_B}
    \|J_{s^{-1}}(s(x))-J_{s^{-1}}(s(x'))\|\leq\beta_2\|s(x)-s(x')\| \le \dfrac{\beta_2}{L_1}\|x-x'\|,
\end{equation}
where we have used Remark \ref{rem:H3} in the last inequality. By $(H_3)$ we also have that $J_f\in\textrm{Lip}_{\beta_1}(D_1)$, so 
\begin{equation}\label{eq_C}
   \|J_f(x)-J_f(x')\|\leq\beta_1\|x-x'\|,
\end{equation}
for every $x,x'\in D_1$. Using \eqref{eq_B}-\eqref{eq_C} in \eqref{eq_A}, together with Remark \ref{rem:M1M2} gives
\begin{align*}
 \|J_{f\circ s^{-1}}(z)-J_{f\circ s^{-1}}(z')\|   & \leq \left(\dfrac{M_1\beta_2}{L_1}+M_2\beta_1\right)\|x-x'\| =\Bar{M}\left\|s^{-1}(z)-s^{-1}(z')\right\| \\
 &\\
    \qquad& \leq\Bar{M}L_0\|z-z'\|\,;
\end{align*}
so $J_F\in\textrm{Lip}_{\Bar{M}L_0}(D_2)$, where $\Bar{M}=\dfrac{M_1\beta_2}{L_1}+M_2\beta_1$.  Hence Condition~(ii) holds for $\gamma:= \Bar{M}L_0$. This completes the proof of conditions (i) and (ii). Using now Theorem \ref{thm:quadconv} and the definitions of $\beta$ and $\gamma$, we obtain (a) and (b) for the sequence $(y^k)$ with the stated value of $\eta$.  
\end{proof}

\begin{theorem} \label{thm:gen_convergence}
With the notation of Definition \ref{def:itGN}, assume that $(H_0)$--$(H_3)$ hold. The sequence $(x^k)$ given by the rule \eqref{it:GN} is well defined and converges quadratically to $x^*$.
\end{theorem}

\begin{proof}
By Lemma \ref{lem:QC}, the sequence $(y^{k})$ with $y^k:=s(x^k)$ is well defined and converges quadratically to $s(x^*)$. By Lemma \ref{lem:transformation}, the iteration on $(y^k)$ can be equivalently written as
\[
s(x^{k+1}) = s(x^k)-J_s(x^k)[J_f(x^k)]^{-1}f(x^k).
\]
We will show now that $(x^k)$ converges quadratically to $x^*$. Indeed, by part (b) of Lemma \ref{lem:QC}, we have that
 \begin{equation}\label{yk}
    \|{y}^{k+1}-{s(x^*)}\|\leq\eta\|{y}^k-{s(x^*)}\|^2,
 \end{equation}
 for $\eta$ as in Lemma \ref{lem:QC}(b). We can write
\begin{align*}
    \|x^{k+1}-x^*\| & = \left\|s^{-1}\left(s(x^k)-J_s(x^k)[J_f(x^k)]^{-1}f(x^k)\right)-x^*\right\| \\
    \\
    & = \left\|s^{-1}\left(s(x^k)-J_s(x^k)[J_f(x^k)]^{-1}f(x^k)\right)-s^{-1}(s(x^*))\right\|
\end{align*}
Applying Remark \ref{rem:H3} to bound the right-most expression, we obtain
\begin{align*}
   \|x^{k+1}-x^*\|   & \leq L_0\left\|s(x^k)-J_s(x^k)[J_f(x^k)]^{-1} f(x^k) - s(x^*)\right\| \\
  & \\
    & = L_0 \, \|y^{k+1}-s(x^*)\|,
\end{align*}
where we have used the definition of $y^{k+1}$ in the last inequality. We can now use \eqref{yk} in the above expression to derive
\begin{align*}
  \|x^{k+1}-x^*\|      & \leq L_0\eta\|y^k-s(x^*)\|^2=L_0\eta\|s(x^k)-s(x^*)\|^2.
\end{align*}
Applying again Remark \ref{rem:H3} to bound the right-most expression, we obtain
\begin{align*}
  \|x^{k+1}-x^*\|     & \leq \frac{L_0\eta}{L_1}\|x^k-x^*\|^2.
\end{align*}
So $(x^k)$ converges quadratically to $x^*$, as desired.
\end{proof}

\section{Bounds on the Asymptotic Error Constant \boldmath{$\lambda$}}
\label{ss:AEC}
\setcounter{definition}{0}
\setcounter{remark}{0}
\setcounter{lemma}{0}
\setcounter{theorem}{0}

While the results of the previous section hold for the specific iteration \eqref{it:GN}, the following results are true for any fixed-point iteration of the form $g(x^k)=x^{k+1}$. In this section we will always assume that $g$ is twice continuously differentiable.  

\begin{proposition} \label{prop:multivarLambdaProof}
Assume that $g({x}^*)={x}^*$ and that $J_g({x}^*)=0$. Consider the sequence $(x^k)$ defined by the fixed point iteration ${x}^{k+1}=g({x}^k)$. Then there exist 
  sequences $({\xi}_1^k),\dots,({\xi}_n^k)$ converging to $x^*$ such that
  the asymptotic error constant $\lambda$ given by Definition \ref{def:convSpeed} verifies 
  \[
  \lambda = \frac{1}{2}\,\lim_{k\rightarrow\infty}\dfrac{\|T_g({\xi}_1^k,\dots,{\xi}_n^k)_{({x}^k-{x}^*,\dots,{x}^k-{x}^*)}\|}{\|{x}^k-{x}^*\|^2},
  \]
where $T_g$ is as in Definition \ref{def:complicatedFunctions}.
\end{proposition}
\begin{proof}
Since $J_g({x}^*)=0$, it is well known that $(x^k)$ converges quadratically to $x^*$. 
We begin by writing $g(x)$ into a Taylor polynomial around $x^*$, coordinate by coordinate,
\[
g(x)=g({x}^*)+J_g({x}^*)(x-{x}^*)+\frac{1}{2}
\begin{bmatrix}
(x-{x}^*)^T\, \nabla^2 g_1({\xi}_1)\,(x-{x}^*) \\
\vdots \\
(x-{x}^*)^T\, \nabla^2 g_n({\xi}_n)\,(x-{x}^*)
\end{bmatrix},
\]
where ${\xi}_j$, $j=1,\ldots,n$, are between ${x}$ and ${x}^*$. Using now Definition \ref{def:complicatedFunctions} as well as the equalities $g({x}^*)={x}^*$ and $J_g({x}^*)=0$, we obtain
\[
g(x)={x}^* +\frac{1}{2}T_g({\xi}_1,\dots,{\xi}_n)_{({x}-{x}^*,\dots,{x}-{x}^*)}\,.
\]
By taking $x:=x^k$ and using the definition of the fixed-point iteration, we obtain
\[
g({x}^k)=x^{k+1}={x}^*+\frac{1}{2}T_g({\xi}_1^k,\dots,{\xi}_n^k)_{({x}^k-{x}^*,\dots,{x}^k-{x}^*)}\,,
\]
where ${\xi}_j^k$, $j=1,\ldots,n$, are between ${x}^k$ and ${x}^*$. Re-arranging, taking norms, and then dividing by $\|{x}^k-{x}^*\|^2$ yields
\[
\dfrac{ \|x^{k+1}-{x}^*\|}{\|{x}^k-{x}^*\|^2}=\frac{\|T_g({\xi}_1^k,\dots,{\xi}_n^k)_{({x}^k-{x}^*,\dots,{x}^k-{x}^*)}\|}{2 \|{x}^k-{x}^*\|^2}\,.
\]
We know that the sequences $({x}^k)$ and $({\xi}_j^k)$ converge to ${x}^*$ as $k\rightarrow\infty$ for $j=1,\dots,n$. Using these facts and taking limits yield
\begin{equation}\label{star}
\begin{array}{rcl}
\lambda &=& \ds\lim_{k\rightarrow\infty}\dfrac{\|x^{k+1}-{x}^*\|}{\|{x}^k-{x}^*\|^2} = \ds\frac{1}{2}\,\lim_{k\rightarrow \infty}\dfrac{\|T_g({\xi}_1^k,\dots,{\xi}_n^k)_{({x}^k-{x}^*,\dots,{x}^k-{x}^*)}\|}{ \|{x}^k-{x}^*\|^2},
\end{array}
\end{equation}
where we have also invoked Definition \ref{def:convSpeed}. This proves the proposition. 
\end{proof}

Our aim in this section is to use Proposition \ref{prop:multivarLambdaProof} to establish upper and lower bounds for the asymptotic error constant. For this we will need the following definition. For $j\in \{1,\ldots,n\}$, we denote by $[v]_j$ the $j$th coordinate of the vector $v\in \dR^n$.

\begin{definition} \label{def:spectral vectors} \rm
Given $n$ vectors $x^1,\ldots,x^n$, and the matrices  $\nabla^2 g_1(x^1),\ldots,\nabla^2 g_n(x^n)$, define the vector $\mu(x^1,\ldots,x^n)\in \dR^n_+$ as
\begin{equation}\label{eq:spectral vector2}
[\mu(x^1,\ldots,x^n)]_j :=
\left\{  \begin{array}{cl}
0\,, & \hbox{ if\ \ } \lambda_{\rm min}(\nabla^2 g_j(x^j))<0<\lambda_{\rm max}(\nabla^2 g_j(x^j))\,,\\[4mm]
\lambda_{\rm min}(\nabla^2 g_j(x^j))\,,  & \hbox{ if\ \ } \lambda_{\rm min}(\nabla^2 g_j(x^j))\ge 0\,, \\[4mm]
\left|\lambda_{\rm max}(\nabla^2 g_j(x^j))  \right|,   & \hbox{ if\ \ } \lambda_{\rm max}(\nabla^2 g_j(x^j))\le 0\,, \\
\end{array}  \right.  
\end{equation} 
for $j=1,\ldots,n$.  Define also the vector $\rho(x^1,\ldots,x^n)\in \dR^n_+$ as
\[
[\rho(x^1,\ldots,x^n)]_j:={\rm SR}(\nabla^2 g_j(x^j))=\|\nabla^2 g_j(x^j)\|,
\]
for $j=1,\ldots,n$.
For simplicity in appearance, we will write
$\mu(x)$ and $\rho(x)$ when $x^j=x$ for all $j=1,\ldots,n$. Namely,
\[
\mu(x):=\mu(x,\ldots,x)\quad \hbox{ and }\quad
\rho(x):=\rho(x,\ldots,x)\,.
\]
\end{definition}

\begin{remark}\label{rem:rho-cont}\rm
With the notation of Definition \ref{def:spectral}, it is well-known that the spectral radius ${\rm SR}(A)$ is a continuous function of the matrix $A$. Hence, if $g$ is twice continuously differentiable in a neighbourhood around $x$, the function $\rho$ will be a continuous function of $x$. Therefore, whenever $({\xi}_1^k),\dots,({\xi}_{n-1}^k)$ and $({\xi}_n^k)$ are sequences converging to the same point $z$, we will have $\rho({\xi}_1^k,\dots,{\xi}_n^k)$ converging to $\rho(z,\ldots,z)=\rho(z)$. A similar fact can be established for the function $\mu$. 
\end{remark}

\begin{remark}\label{rem:mu-def}\rm 
Clearly, the function $[\mu(\cdot)]_j$ can be equivalently defined as 
\[
\begin{array}{l}
[\mu(x^1,\ldots,x^n)]_j =\\[3mm]
\left\{  
\begin{array}{l}
0,\,  \hbox{\ \ if\ \ } \lambda_{\rm min}(\nabla^2 g_j(x^j))<0<\lambda_{\rm max}(\nabla^2 g_j(x^j))\,,\\[2mm]
\min\{|\lambda_{\rm min}(\nabla^2 g_j(x^j))|,|\lambda_{\rm max}(\nabla^2 g_j(x^j))|\}\,,
\hbox{\ \ if\ } \left\{
\begin{array}{l}
  \lambda_{\rm min}(\nabla^2 g_j(x^j))\ge 0\hbox{\ \ or } \\[1mm]
  \lambda_{\rm max}(\nabla^2 g_j(x^j))\le 0\,,  
\end{array}\right.
\end{array}  
\right.
\end{array}
\]
for $j=1,\ldots,n$. So the definition of $[\mu(\cdot)]_j$ is given over two complementary sets, one of them open and the other closed. Since in each of these sets $[\mu(\cdot)]_j$ is given by a continuous function, and the values of these two functions coincide at the boundary of the two complementary sets, we deduce that $\mu$ is a continuous function.
\end{remark}

The next technical result will be used in Theorem~\ref{thm:lambdaBounds}. Recall from Remark~\ref{rem:Tg-lambda} and Definition~\ref{def:complicatedFunctions} that 
\begin{equation}\label{eq:ro-Tg}
    \|T_g(x^1,\ldots,x^n)\|=\|\rho(x^1,\ldots,x^n)\|\,.
\end{equation}


\begin{proposition} \label{rem:Rayleigh}\rm
With the notation of Definition \ref{def:spectral vectors},  we have that
\[
([\mu(x^1,\ldots,x^n)]_j)^2\le \dfrac{(u^T\nabla^2 g_j(x^j) u)^2}{\|u\|^4}\le 
([\rho(x^1,\ldots,x^n)]_j)^2,
\]
for $j=1,\ldots,n$ and all nonzero $u\in \dR^n$. 
\end{proposition}
\begin{proof}
Recall that by Rayleigh quotient properties,
\[
\lambda_{\rm min}(\nabla^2 g_j(x^j))\le \dfrac{u^T\nabla^2 g_j(x^j) u}{\|u\|^2}\le  \lambda_{\rm max}(\nabla^2 g_j(x^j)),
\]
for all $j=1,\ldots,n$. So for each fixed $j$ we can apply Fact \ref{fact:elem} with \[
\begin{array}{ll}
 a:=\lambda_{\rm min}(\nabla^2 g_j(x^j)),    & b:=\lambda_{\rm max}(\nabla^2 g_j(x^j)), \\
 &\\
 q=\dfrac{u^T\nabla^2 g_j(x^j) u}{\|u\|^2},    & c:=\max\{ |\lambda_{\rm min}(\nabla^2 g_j(x^j))|, |\lambda_{\rm max}(\nabla^2 g_j(x^j))| \}.
\end{array}
\]
Assume that $a=\lambda_{\rm min}(\nabla^2 g_j(x^j))\ge 0$. In this situation, by definition we have\linebreak $[\mu(x^1,\ldots,x^n)]_j=a= \lambda_{\rm min}(\nabla^2 g_j(x^j))$ and $[\rho(x^1,\ldots,x^n)]_j=b= \lambda_{\rm max}(\nabla^2 g_j(x^j))$. Now part (i) of Fact \ref{fact:elem} directly yields
\[
([\mu(x^1,\ldots,x^n)]_j)^2=a^2= (\lambda_{\rm min}(\nabla^2 g_j(x^j)))^2\le q^2= \dfrac{(u^T\nabla^2 g_j(x^j) u)^2}{\|u\|^4} \le b^2= (\lambda_{\rm max}(\nabla^2 g_j(x^j)))^2.
\]
If $b=\lambda_{\rm max}(\nabla^2 g_j(x^j))\le 0$ then $[\rho(x^1,\ldots,x^n)]_j=|\lambda_{\rm min}(\nabla^2 g_j(x^j))|$ and by part (ii) of Fact \ref{fact:elem} we have
\[
\begin{array}{rcl}
 ([\mu(x^1,\ldots,x^n)]_j)^2 =b^2&=&(|\lambda_{\rm max}(\nabla^2 g_j(x^j))|)^2 \\
 &&\\
 &&\le \dfrac{(u^T\nabla^2 g_j(x^j) u)^2}{\|u\|^4}=q^2 \le (|\lambda_{\rm min}(\nabla^2 g_j(x^j))|)^2\\
 &&\\
 &&=a^2=([\rho(x^1,\ldots,x^n)]_j)^2.
  \end{array}
\]
Finally, if $\lambda_{\rm min}(\nabla^2 g_j(x^j))<0<\lambda_{\rm max}(\nabla^2 g_j(x^j))$ then by definition $[\mu(x^1,\ldots,x^n)]_j=0$ and $[\rho(x^1,\ldots,x^n)]_j=\max\{ |\lambda_{\rm min}(\nabla^2 g_j(x^j))|, |\lambda_{\rm max}(\nabla^2 g_j(x^j))| \}=c$. Using part (iii) of Fact \ref{fact:elem} yields
\[
\begin{array}{rcl}
0= ([\mu(x^1,\ldots,x^n)]_j)^2 &\le & \dfrac{(u^T\nabla^2 g_j(x^j) u)^2}{\|u\|^4}=q^2\le c^2\\
&&\\
&&=
(\max\{ |\lambda_{\rm min}(\nabla^2 g_j(x^j))|, |\lambda_{\rm max}(\nabla^2 g_j(x^j))| \})^2\\
&&\\
&&= ([\rho(x^1,\ldots,x^n)]_j)^2.
  \end{array}
\]
This completes the proof.
\end{proof}

\begin{theorem}\label{thm:lambdaBounds}
Suppose $g(x^*)=x^*$ and $J_g(x^*)=0$, with $g$ twice continuously differentiable in a neighbourhood of $x^*$. Consider the asymptotic error constant $\lambda\ge 0$ for the multivariate iterative method $x^{k+1}=g(x^k)$. Then, it holds that 
\begin{equation*}
     \frac{1}{2}\|\mu(x^*)\| \leq \lambda \leq \frac{1}{2}\|T_g(x^*)\|,
\end{equation*}
where $\mu(x^*)$ is as in Definition~\ref{def:spectral vectors}, and $T_g$ as in Definition \ref{def:complicatedFunctions}.
\end{theorem}
\begin{proof}
Recall that, by \eqref{eq:ro-Tg}, $\|T_g(x^*)\|=\|\rho(x^*)\|$. Hence, it is enough to establish the upper bound with $\|\rho(x^*)\|$ instead of $\|T_g(x^*)\|$. By Proposition \ref{prop:multivarLambdaProof} we have that 
\[
  \lambda  =\lim_{k\rightarrow\infty}\dfrac{\|T_g({\xi}_1^k,\dots,{\xi}_n^k)_{({x}^k-{x}^*,\dots,{x}^k-{x}^*)}\|}{2 \|{x}^k-{x}^*\|^2},
\]
where $({\xi}_1^k),\ldots,({\xi}_n^k)$ are $n$ sequences converging to $x^*$. Using now Proposition~\ref{rem:Rayleigh} for $x^j=\xi_j^k$ and $u=x^k -x^*\neq 0$ we deduce that
\begin{equation}\label{eqn:rayleigh}
   ([\mu({\xi}_1^k,\dots,{\xi}_n^k)]_j)^2\le \dfrac{((x^k -x^*)^T\nabla^2 g_i(\xi_j^k) (x^k -x^*))^2}{\|(x^k -x^*)\|^4}\le 
([\rho({\xi}_1^k,\dots,{\xi}_n^k)]_j)^2.
\end{equation}
By definition of $T_g$ and Lemma \ref{lem:Tg} we have
\[
\|T_g({\xi}_1^k,\dots,{\xi}_n^k)_{({x}^k-{x}^*,\dots,{x}^k-{x}^*)}\|\le \|x^k-x^*\|^2 \|T_g({\xi}_1^k,\dots,{\xi}_n^k)\|.
\]
Combine this fact with \eqref{eqn:rayleigh} and Definition \ref{def:complicatedFunctions} to derive
\begin{equation}\label{eq:lambda}
\begin{array}{rcl}
    \frac{1}{2}\lim_{k\rightarrow\infty}\|\mu({\xi}_1^k,\dots,{\xi}_n^k)\| &\leq&
    \lim_{k\rightarrow\infty}\dfrac{\|T_g({\xi}_1^k,\dots,{\xi}_n^k)_{({x}^k-{x}^*,\dots,{x}^k-{x}^*)}\|}{2 \|{x}^k-{x}^*\|^2}\\
    &&\\
  &&\leq \frac{1}{2}\lim_{k\rightarrow\infty}\|T_g({\xi}_1^k,\dots,{\xi}_n^k)\|.
  \end{array}
\end{equation}

\noindent By Remarks \ref{rem:rho-cont}, \ref{rem:mu-def} and the fact that $g$ is twice continuously differentiable, the functions $T_g$ and $\mu$ are continuous, so $$\lim_{k\to\infty} \mu({\xi}_1^k,\dots,{\xi}_n^k)= \mu(x^*) \hbox{ and }
\lim_{k\to\infty} \|T_g({\xi}_1^k,\dots,{\xi}_n^k)\|= \|T_g(x^*)\|=\|\rho(x^*)\|,
$$
where we also used \eqref{eq:ro-Tg} in rightmost equality. These facts, \eqref{eq:lambda} and the definitions of $\lambda$, $\rho$ and $\mu$ yield
\begin{equation*}
    \frac{1}{2}\|\mu(x^*)\| \leq \lambda \leq \frac{1}{2}\|\rho(x^*)\|= \frac{1}{2}\|T_g(x^*)\|,
\end{equation*}
as required.

\end{proof}

\newpage
\section{Numerical Experiments}
\label{sec:NE}

In this section we compare the classical and generalized Newton methods on five example systems of equations of the form $f(x) = {\bf 0}$ as in \eqref{NLE-pro}, with two variables (where visualization is possible) as well as six variables.  The equations in these test problems involve cubic, quartic and exponential functions.  Using various choices of the generalizing function $s$, we look at both local and global behaviour of the classical and generalized methods, and we do this in the following sense.
\begin{itemize}
\item {\em Local behaviour}\,: For each test problem we check that, with either method, convergence to a solution is quadratic, verifying Theorem~\ref{thm:gen_convergence}.  We do this by obtaining numerical estimates of the asymptotic error constant $\lambda$.  Comparisons of the estimates of $\lambda$ for each method give us an idea as to which method is faster locally.  Recall that for any method that we consider the convergence rate is quadratic. So by comparing $\lambda$'s, we compare the local ``speeds'' of quadratically convergent methods.  In other words, the ratio of the $\lambda$'s of two methods will tell us how many times a method is ``faster'' or ``slower'' than the other, locally.  We also provide theoretical bounds on $\lambda$, i.e., intervals in which $\lambda$ lies, for each example by using Theorem~\ref{thm:lambdaBounds}.
    
\item {\em Global behaviour}\,:  It is well known that the main drawback of the classical Newton method is its dependence on the quality of the initial guess (or the starting point) used in the Newton iterations.  For the systems with two equations in two unknowns, we visualize graphically the colour-coded number of iterations that either method requires to converge, if at all, over domains of various sizes, with the set tolerance of $10^{-8}$. These graphs serve to demonstrate the value of the generalized method: The domain in which the generalized method converges in a reasonable number of iterations can be made larger, by choosing the generalizing function $s$ carefully.  We also present statistical information about the global convergence properties by means of a large number of randomly generated starting points for each method, in order to verify the information conveyed by the graphs.  This information ultimately leads to a decision as to which of the methods considered is best to use, on the average, in a given search domain.
\end{itemize}

For computations, we use {\sc Matlab} Release 2019b, update 5.  In getting the estimates of $\lambda$, variable precision arithmetic ({\tt vpa} of {\sc Matlab}) making use of a large number of digits is utilized to be able to obtain these estimates with relatively reliable number of significant figures.

The CPU times are reported by running Matlab on a 13-inch 2018 model MacBook Pro, with the operating system macOS Mojave (version 10.14.6), the processor 2.7 GHz Intel Core i7 and the memory 16 GB 2133 MHz LPDDR3.

The {\sc Matlab} code we have written to generate the colour-coded portraits of number of iterations is based on Cleve Moler's code for viewing fractals generated by univariate Newton iterations~\cite{Moler} in complex plane.   

The methodology used in obtaining the statistical information, such as the average number of iterations, the rate of success of a method, and the CPU times for each iteration and each successful run, on the average, are explained in detail only once, in the first example in Section~\ref{sec:quartic}.  For brevity, we avoid repetitions of this information in the subsequent four examples.

\subsection{Quartic equations}
\label{sec:quartic}

To compare the classical and generalized methods we will first consider the following example system involving simple quartic functions.
\begin{equation} \label{sys:quartics}
    f(x) = \begin{bmatrix}
      {x_2}x_1^3 - 1 \\[1mm]
      {x_1}x_2^3 - 1
    \end{bmatrix} = {\bf 0}\,,
\end{equation}
where $f:\dR^2\to\dR^2$ and ${\bf 0}\in\dR^2$. Clearly, the system \eqref{sys:quartics} has two real solutions; namely $x^* = (1,1)$ and $x^* = (-1,-1)$. The expression for the fixed-point map $g$ associated with this system can be derived by using \eqref{eqn:multiGnewtonIter} with a chosen $s$.  In Table~\ref{tbl:g_ex1}, we do this first with $s(x) = (x_1, x_2)$ for the system in \eqref{sys:quartics} and get $g$ for the classical Newton method.  The appearance of $x_1^3$ and $x_2^3$ in the first and second equations, respectively, prompts us to choose $s(x) =  (x_1^3, x_2^3)$ for the generalized Newton method. Then we use this $s$ to get the $g$ for the generalized method as displayed in Table~\ref{tbl:g_ex1}.  We refer to the method obtained in this way the {\em cube-generalized Newton method}.

\renewcommand{\arraystretch}{1.5}
\begin{table}[h]
\centering
\begin{tabular}{ccc}
$f(x)$ & $s(x)$ & $g(x)$ \\ \hline \\[-6mm]
$\begin{bmatrix}
{x_2}x_1^3-1 \\
{x_1}x_2^3-1
\end{bmatrix}$ & 
$\begin{bmatrix}
{x_1} \\
{x_2}
\end{bmatrix}$ & 
$\begin{bmatrix}
{x_1}-(2x_1^3x_2^3 - 3x_2^2 + x_1^2)/(8x_1^2x_2^3) \\
{x_2}-(2x_1^3x_2^3 - 3x_1^2 + x_2^2)/(8x_1^3x_2^2)
\end{bmatrix}$ \\[7mm]
  & 
$\begin{bmatrix}
x_1^3 \\
x_2^3
\end{bmatrix}$ & 
$\begin{bmatrix}
\sqrt[3]{x_1^3 - 3(2x_1^3x_2^3 - 3x_2^2 + x_1^2)/(8x_2^3}) \\
\sqrt[3]{x_2^3 - 3(2x_1^3x_2^3 - 3x_2^2 + x_1^2)/(8x_1^3})
\end{bmatrix}$ \\[6mm] \hline
\end{tabular}
\caption{\small\sf System~\eqref{sys:quartics}: Fixed-point map $g$ with different choices of $s$.}
\label{tbl:g_ex1}
\end{table}
\renewcommand{\arraystretch}{1}

By Theorem~\ref{thm:gen_convergence}, the fixed-point methods (classical and generalized Newton) using the choices of $g$ listed in Table~\ref{tbl:g_ex1} are quadratically convergent and this can numerically be verified.  What can one say about the asymptotic error constant?  Table~\ref{tbl:ex1_results} encapsulates the ensuing answer.  The values listed for $\lambda$ are the numerical estimates of the asymptotic error constant\ \ $\lim_{k\rightarrow\infty}{\|x^{k+1}-{x}^*\|}/{\|{x}^k-{x}^*\|^2}$\ \ from Definition~\ref{def:convSpeed}.  These estimates are the same for both of the solutions $x^* = (1,1)$ and $x^* = (-1,-1)$ (referred to as Solutions~1 \& 2) because of the symmetry of the equations in System~\eqref{sys:quartics}.  We note that the estimates of $\lambda$ consistently fall into the intervals defined by the theoretical bounds established for $\lambda$ in Theorem~\ref{thm:lambdaBounds}, which are also shown in the table.  The (estimated) ratio $\lambda_N/\lambda_{GN}$ of the asymptotic error constants of the classical and generalized Newton methods, respectively, implies that the generalized method is about three times faster near a solution, for this example.

\begin{table}[h]
\centering
{\small
\begin{tabular}{ccccc}
Soln & $s_i(x)$ & $[\|\mu(x^*)\|/2,\ \|\rho(x^*)\|/2]$ & $\lambda$ & $\lambda_N/\lambda_{GN}$ \\[1mm] \hline \\[-4mm]
1 \& 2 & $x_i$ & $[0,1.7]$ & $1.06$ &  \\[1mm]
 & $x_i^3$ & $[0,0.8]$ & $0.35$ & $3.0$  \\[1mm] \hline
\end{tabular}
\caption{\small\sf System~\eqref{sys:quartics}: Asymptotic error constants of the classical ($s_i(x) = x_i$, $i = 1,2$) and generalized ($s_i(x) = x_i^3$, $i = 1,2$) Newton methods.}
\label{tbl:ex1_results}}
\end{table}

We pointed out in the Introduction that, for a given solution, a quadratic convergence region is in general not known {\em a priori}. Next we graphically illustrate in Figure~\ref{fig:quartics} that the (quadratic) convergence regions about the solutions for the generalized Newton method we have devised for this example are larger than those resulting from the classical Newton method. We also look at the regions of convergence over larger domains than just the neighbourhoods of the solutions.

In the graphs in Figure~\ref{fig:quartics}, the number of iterations needed to converge with tolerance $10^{-8}$ to a solution from a given point (i.e., an initial guess) is colour-coded as indicated by the colour bar next to each graph: while 2--4 iteration runs are represented by dark blue, 14 or more iteration runs, which are regarded as ``{\em unsuccessful}\,,'' are represented by yellow.  The initial guesses are generated over a $1000\times1000$ grid in the {\em search domains} $[-3,3]^2$, $[-10,10]^2$ and $[-100,100]^2$.  The following immediate observations point to some desirable properties of the cube-generalized method for this example:
\begin{itemize}
    \item Overall, the graphs associated with the cube-generalized method have far smaller yellow regions.
    \item We note by looking at the $[-3,3]^2$-domain that the regions in which convergence is achieved in 2--6 iterations are much larger than that for the classical method.
    \item In particular, if the search domain is chosen to be much larger, for example $[-100,100]^2$, then the classical method is unlikely to converge, while the cube-generalized method has a much better chance to converge.
\end{itemize}

\begin{figure}[t!]
\centering
\begin{subfigure}{.5\textwidth}
\centering
\includegraphics[width=\textwidth]{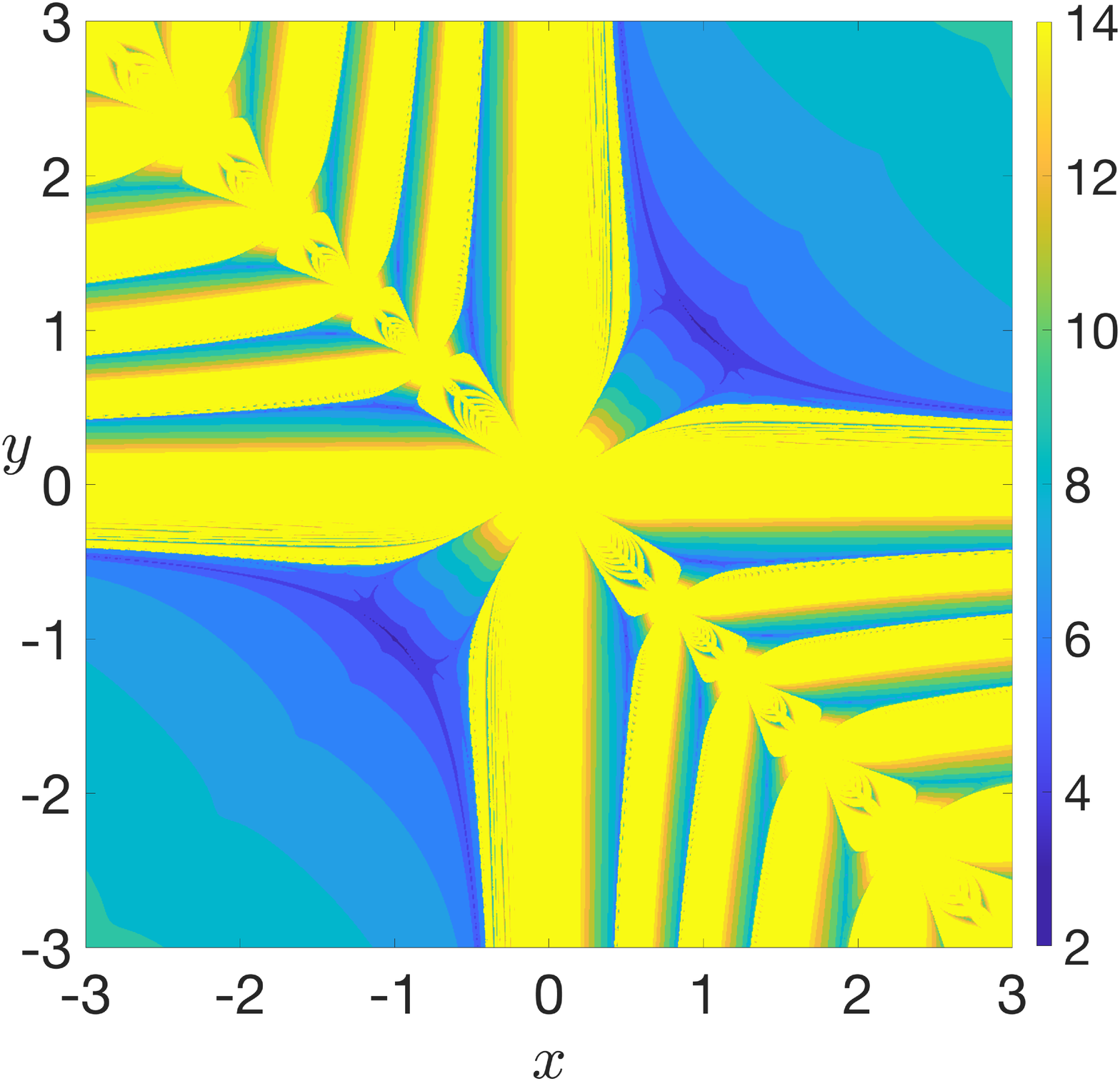}
\end{subfigure}
\hfill
\begin{subfigure}{.49\textwidth}
\centering
\includegraphics[width=\textwidth]{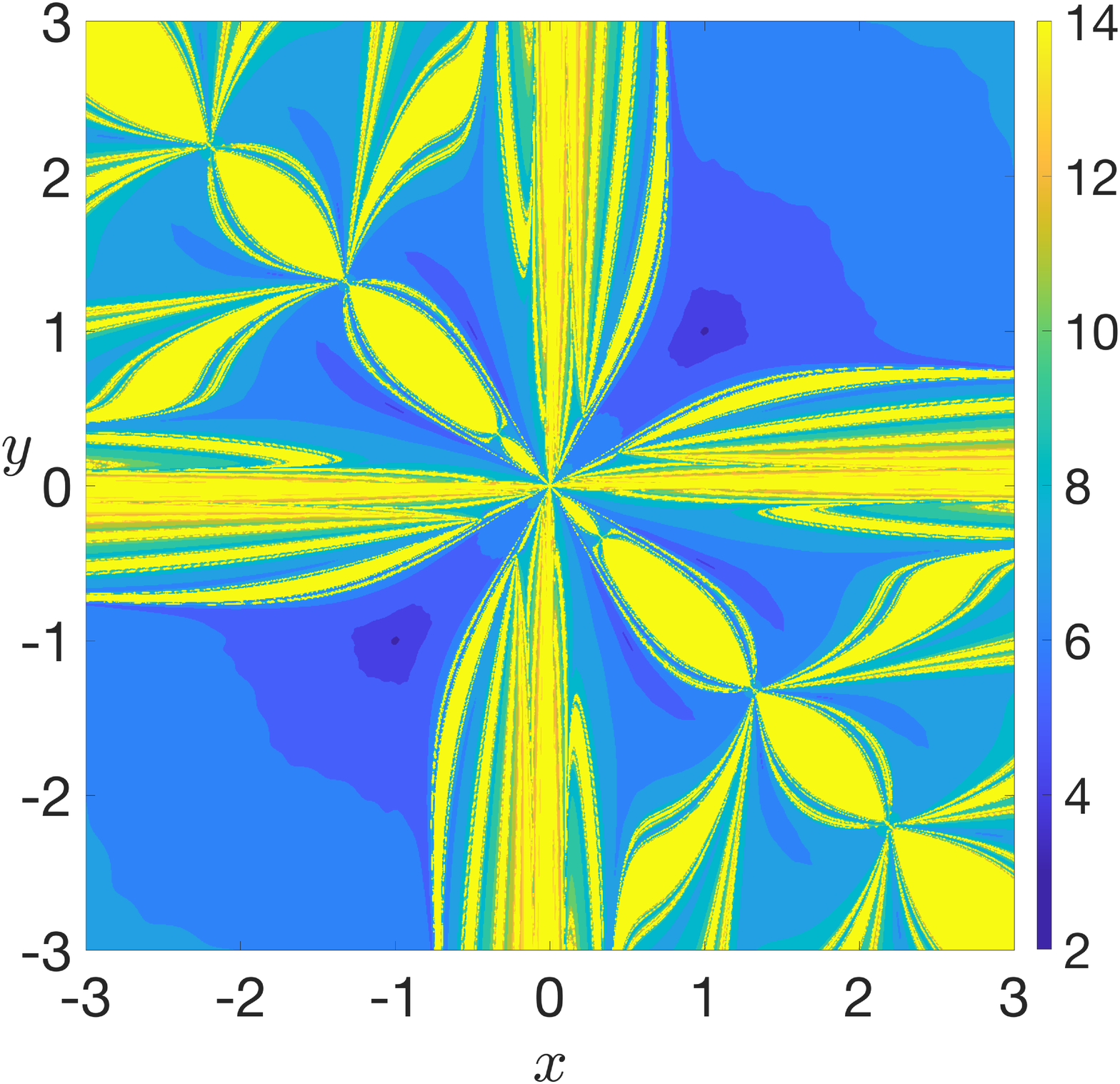}
\end{subfigure}
\begin{subfigure}{.5\textwidth}
\centering
\includegraphics[width=\textwidth]{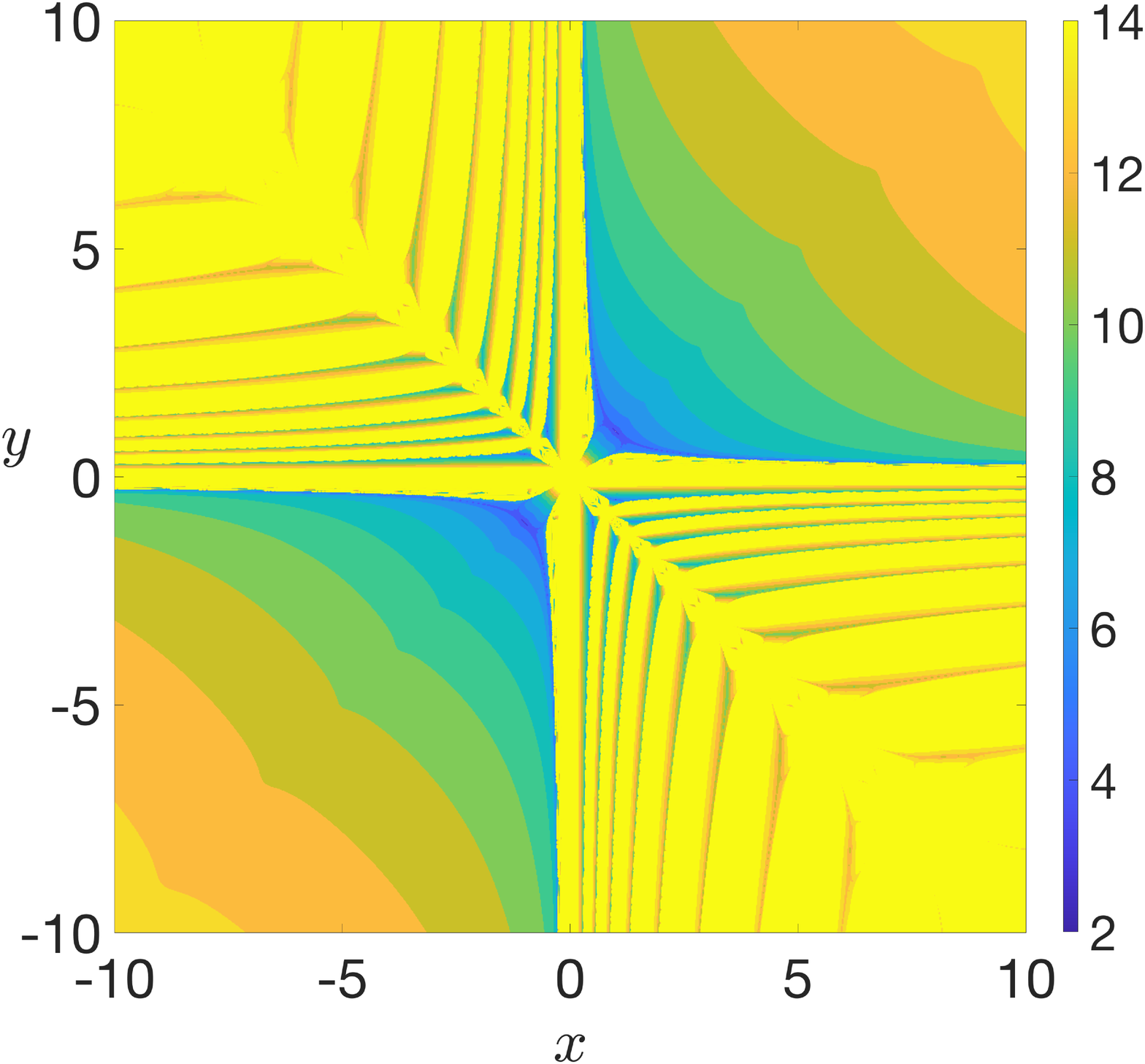}
\end{subfigure}
\hfill
\begin{subfigure}{.49\textwidth}
\centering
\includegraphics[width=\textwidth]{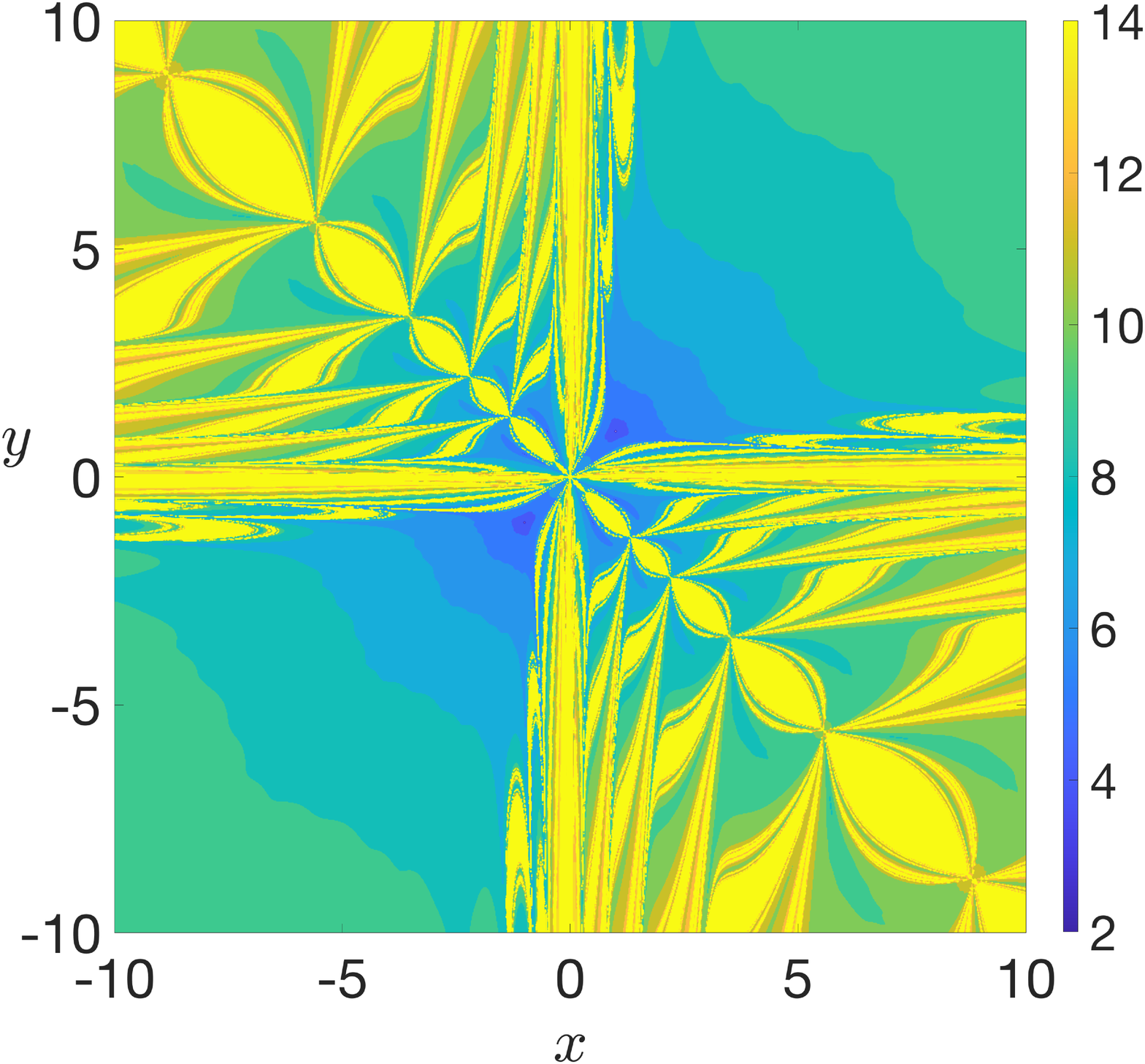}
\end{subfigure}
\begin{subfigure}{.49\textwidth}
\centering
\includegraphics[width=\textwidth]{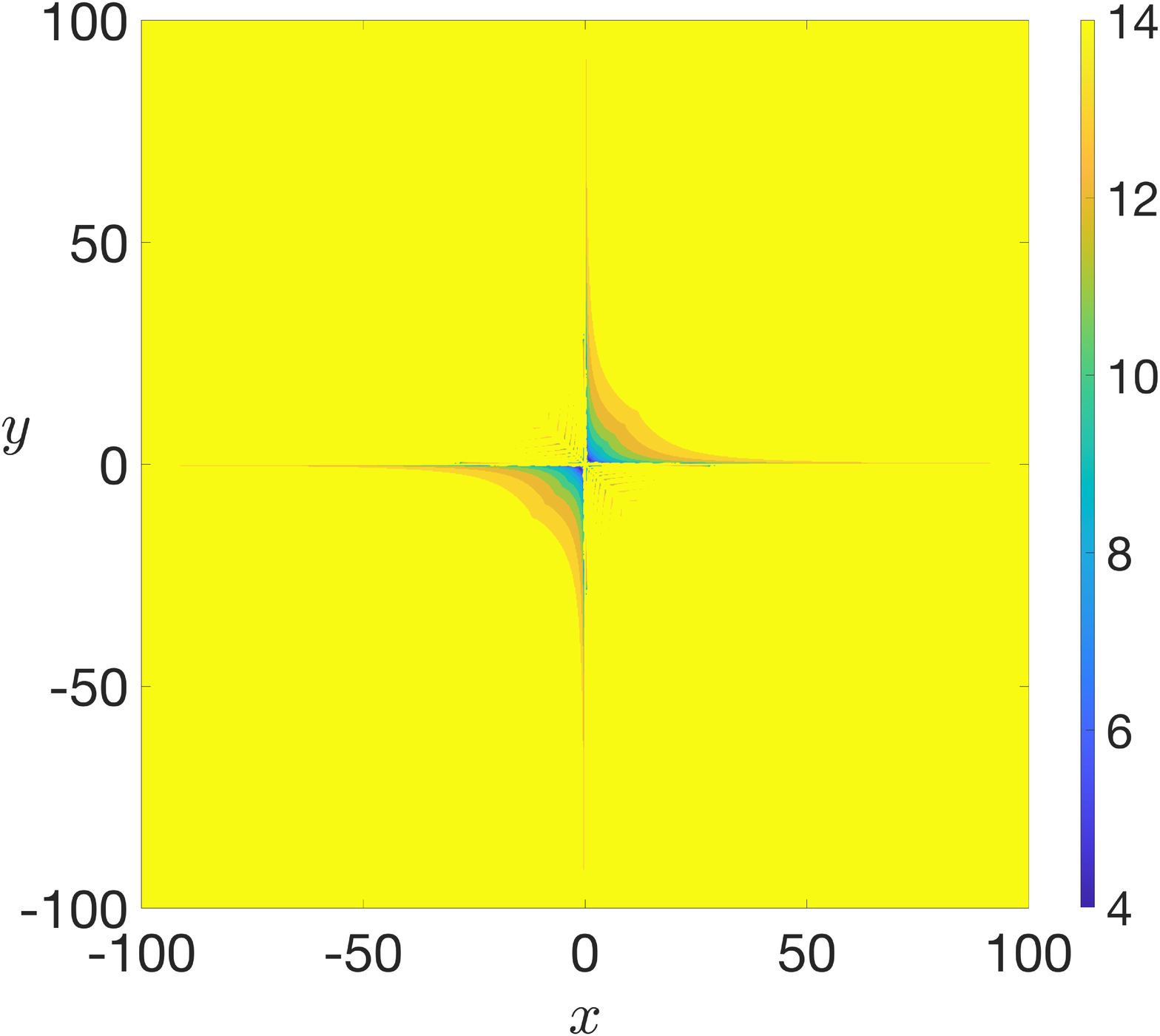}
\caption{$s(x) = (x_1, x_2)$}
\label{fig:cla3}
\end{subfigure}
\hfill
\begin{subfigure}{.49\textwidth}
\centering
\includegraphics[width=\textwidth]{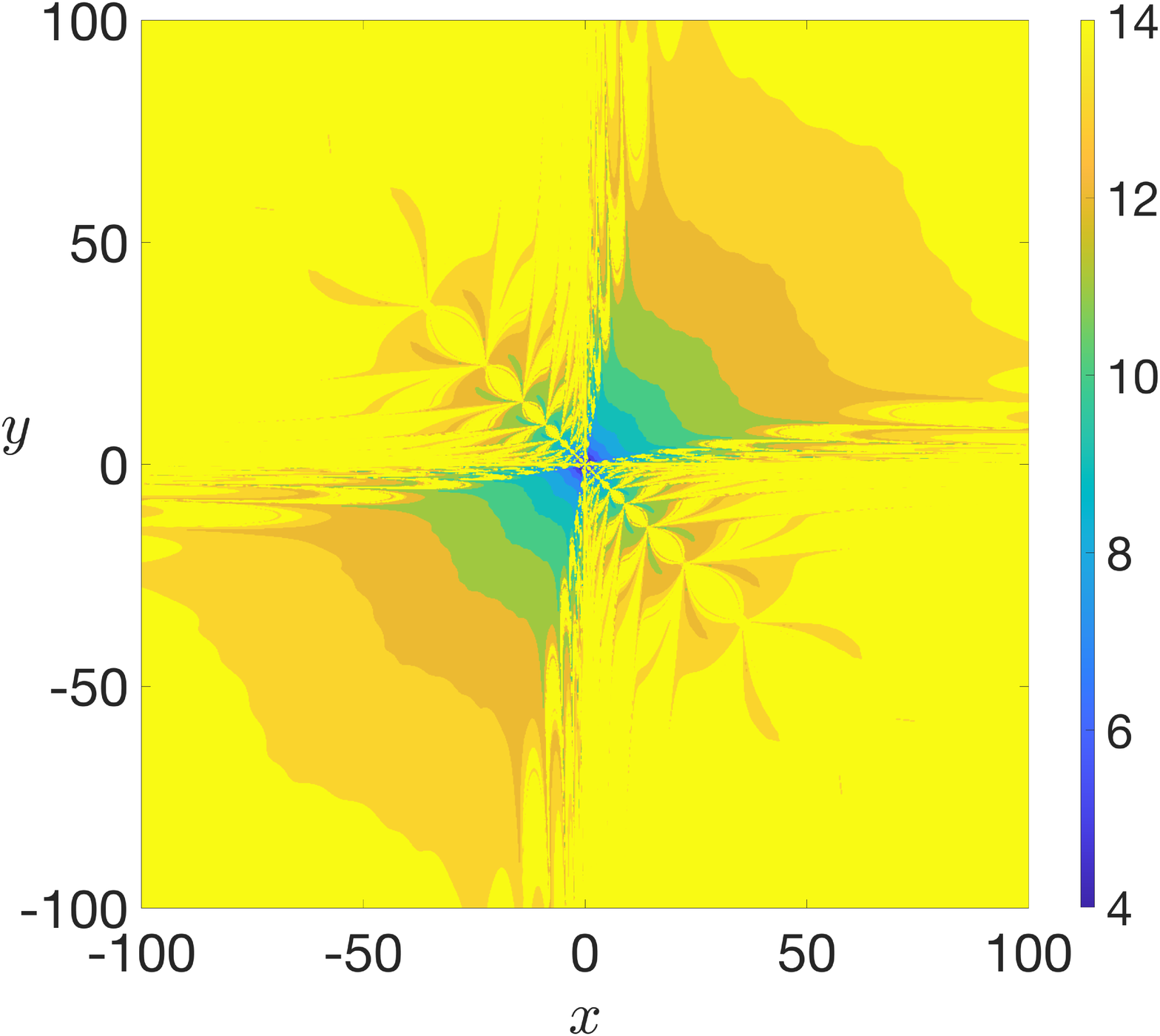}
\caption{$s(x) = (x_1^3, x_2^3)$}
\label{fig:gen3}
\end{subfigure}
\caption{\small\sf System~\eqref{sys:quartics}: Portraits of colour-coded number of iterations required for convergence.}
\label{fig:quartics}
\end{figure}

Next we carry out further numerical experiments to support some of our visual observations in Figure~\ref{fig:quartics}.  In each of the domains $[-3,3]^2$, $[-10,10]^2$ and $[-100,100]^2$, we randomly generate one million starting points and record the number of iterations needed to converge from each point.  We re-iterate that if the number of iterations is 14 or greater, then we deem that particular run {\em unsuccessful}.  In Table~\ref{tbl:quartics}, we list, for several typical choices of $s$,  the average number of iterations over each of the search domains for the successful runs.  We also list the percentage of the runs that were successful, namely the success rate.  

The CPU time taken by a single iteration of the successful runs on the average cannot be found reliably by simply measuring and recording each successful run time and then averaging them, since the very short CPU time of a single run (to the order of $10^{-6}$) cannot be measured reliably.  Therefore, with 100 random starting points, we repeat each successful run $10^{5}$ times and take the average.  This provides an accurate averaged measure of the CPU time per iteration, which is listed for each method in the last column of Table~\ref{tbl:quartics}.

\begin{table}[H]
    \centering
{\small
\begin{tabular}{c|rr|rr|rr|c}
 & \multicolumn{2}{|c|}{$[-3,3]^2$} & \multicolumn{2}{c|}{$[-10,10]^2$} & \multicolumn{2}{c|}{$[-100,100]^2$} & CPU time/ \\[0.5mm]
\cline{2-3} \cline{4-5} \cline{6-7}
 & Ave & Success & Ave & Success & Ave & Success & successful \\
$s_i(x)$ & iter & rate [\%] & iter & rate [\%] & iter & rate [\%] & iter [sec] \\[1mm] \hline\\[-4mm]
$x_i$ & 8.0 & 56.4 & 10.5 & 56.9 & 11.8 & 2.0 & $2.7\times10^{-6}$ \\
$x_i^3$ & 7.1 & 77.0 & 8.9 & 78.6 & 12.3 & 36.2 & $4.7\times10^{-6}$ \\ 
$\sinh(x_i)$ & 7.9 & 67.7 & 9.0 & 25.7 & 9.0 & 0.3 & $2.9\times10^{-6}$ \\
$e^{x_i}$ & 9.0 & 76.0 & 10.7 & 27.6 & 10.6 & 0.3 & $5.1\times10^{-6}$ \\
$\tan{x_i}$ & 5.9 & 10.9 & 6.5 & 14.8 & 7.1 & 0.3 & $3.0\times10^{-6}$ \\[1mm] \hline
\end{tabular}
\caption{\small\sf System~\eqref{sys:quartics}: Performance of the classical and generalized Newton methods with one million randomly generated starting points in domains of various sizes.}
\label{tbl:quartics}}
\end{table}

Table~\ref{tbl:quartics} tells us that over the search domain $[-3,3]^2$ the cube-generalized method is successful 77\% of the time it is run, while the success rate of the classical method is 56\%.  When we generate initial points randomly over a much larger domain, i.e., over $[-100,100]^2$ (this might as well be the situation when we have no knowledge of the location of a solution), the difference in the success rates of the two methods is striking: while the cube-generalized method is successful 36\% of the time, the classical method is successful a mere 2\% of the time it is run.  Although the latter case tells clearly what method to use in the domain $[-100,100]^2$, in the other cases, the success rates alone are not sufficient to tell which method will be (globally) ``better'' to use.

To be able to have a clear idea about which method is more desirable than the others, we need to find the time a method needs before it obtains a solution.  Suppose that, for a given method, the CPU time for a successful run is $3.1\times10^{-5}$ sec and the success rate is 50\%. Then, statistically speaking, on average one will need to run that method twice to get a single solution and the time required for this effort will be $6.2\times10^{-5}$ sec.  So we can find the time required to obtain a solution by a given method as: the CPU time per successful iteration, times the average number of iterations, divided by the success rate written as a decimal. The CPU times obtained in this way for each method are tabulated in Table~\ref{tbl2:quartics}.

\begin{table}[H]
    \centering
{\small
\begin{tabular}{c|ccc}
 & \multicolumn{3}{c}{Time needed to get a single soln [sec]} \\[1mm]
\cline{2-4} \\[-4mm]
$s_i(x)$ & $[-3,3]^2$ & $[-10,10]^2$ & $[-100,100]^2$ \\[1mm] \hline\\[-4mm]
$x_i$ & $3.8\times10^{-5}$ & \framebox{$5.0\times10^{-5}$} & $1.6\times10^{-3}$ \\
$x_i^3$ & $4.3\times10^{-5}$ & $5.3\times10^{-5}$ & \framebox{$1.6\times10^{-4}$} \\ 
$\sinh(x_i)$ & \framebox{$3.4\times10^{-5}$} & $1.0\times10^{-4}$ & $8.7\times10^{-3}$ \\
$e^{x_i}$ & $6.0\times10^{-5}$ & $2.0\times10^{-4}$ & $1.8\times10^{-2}$ \\
$\tan{x_i}$ & $1.6\times10^{-4}$ & $1.3\times10^{-4}$ & $7.1\times10^{-3}$ \\[1mm] \hline
\end{tabular}
\caption{\small\sf System~\eqref{sys:quartics}: CPU time needed on average by the classical and generalized Newton methods to obtain a solution in less than 14 iterations, based on the data in Table~\ref{tbl:quartics}.}
\label{tbl2:quartics}}
\end{table}

From the global convergence point of view, the method with the smallest CPU time over a domain in Table~\ref{tbl2:quartics} should be selected, which are framed for each of the three domains of concern.  For the domain $[-3,3]^2$, the time required by the classical Newton method is about 12\% worse than the generalized method with $s_i(x) = \sinh(x_i)$, $i = 1,2$, which we refer to as the {\em sinh-generalized Newton method}. For the domain $[-10,10]^2$, the classical method seems to be the best to use, although its closest contender, the cube-generalized method, takes only 6\% longer time to find a solution.  Over the domain $[-100,100]^2$, the cube-generalized method is clearly the best method to use, as the classical method needs about 10 times more time in obtaining a solution. To rephrase the latter statement: the cube-generalized method is expected to obtain 10 solutions by the time the classical method finds one.

\subsection{Equations involving exponentials}

The following system is a special instance of the Jennrich and Sampson test problem presented in~\cite{MorGarHil1981, JenSam1968}.
\begin{equation}\label{sys:Jennrich}
    f(x) = \begin{bmatrix}
      e^{x_1} + e^{x_2} - 3 \\[1mm]
      e^{2x_1} + e^{2x_2} - 6
    \end{bmatrix} = {\bf 0}\,.
\end{equation}

System~\eqref{sys:Jennrich} has two solutions, namely $x^* = (a, b)$ and $x^* = (b,a)$, referred to here as {\em Solutions} $1$ and $2$, respectively, where $a = \ln((3+\sqrt{3})/2)\approx 0.861211502516490$ and $b = \ln((3-\sqrt{3})/2)\approx -0.455746394408326$, with the approximations correct to 15 dp.  The appearance of the exponential functions in the equations prompts us to choose $s(x) =  (e^{x_1}, e^{x_2})$ for the generalized Newton method's fixed-point map in~\eqref{eqn:multiGnewtonIter}.  We refer to this method as {\em exp-generalized Newton method}.  As before, $s(x) =  (x_1, x_2)$ is used for the classical Newton method.

Table~\ref{tbl1:Jennrich} lists the numerical estimates and the theoretical intervals for the asymptotic error constant $\lambda$, giving some idea about the local behaviour around a solution.  However, we note that the ratio $\lambda_N/\lambda_{GN}$ is not so accurate in this case as the values obtained in the later iterations for $\lambda_N$ seem to fluctuate between 0.4 and 1.4, which we have averaged as 0.9. The approximate value listed for $\lambda_N/\lambda_{GN}$ implies that, close enough to a solution, the exp-generalized method is more than twice faster.

\begin{table}[H]
\centering
{\small
\begin{tabular}{ccccc}
Soln & $s_i(x)$ & $[\|\mu(x^*)\|/2,\ \|\rho(x^*)\|/2]$ & $\lambda$ & $\lambda_N/\lambda_{GN}$ \\[1mm] \hline \\[-4mm]
1 \& 2 & $x_i$ & $[0.05,2.81]$ & $0.9$\ \ \ &  \\[1mm]
 & $e^{x_i}$ & $[0.19,2.64]$ & $0.35$ & $2.6$ \\[1mm] \hline
\end{tabular}
\caption{\small\sf System~\eqref{sys:Jennrich}: Asymptotic error constants of the classical and generalized Newton methods.}
\label{tbl1:Jennrich}}
\end{table}

As in the example in Section~\ref{sec:quartic}, we depict, in Figure~\ref{fig:Jennrich} the colour-coded number of iterations needed to converge to any one of the two solutions. The success of the exp-generalized method is even more striking in this case: (i) the graphs for the exp-generalized method have far smaller yellow regions, (ii) local convergence regions (that are achieved in 4--6 iterations, shown in darker shades of blue) for the exp-generalized method are much larger and (iii) over the larger domain $[-10,10]^2$, the exp-generalized method has a far better chance of converging in less than 14 iterations.

\begin{figure}[t!]
\centering
\begin{subfigure}{.5\textwidth}
\centering
\includegraphics[width=\textwidth]{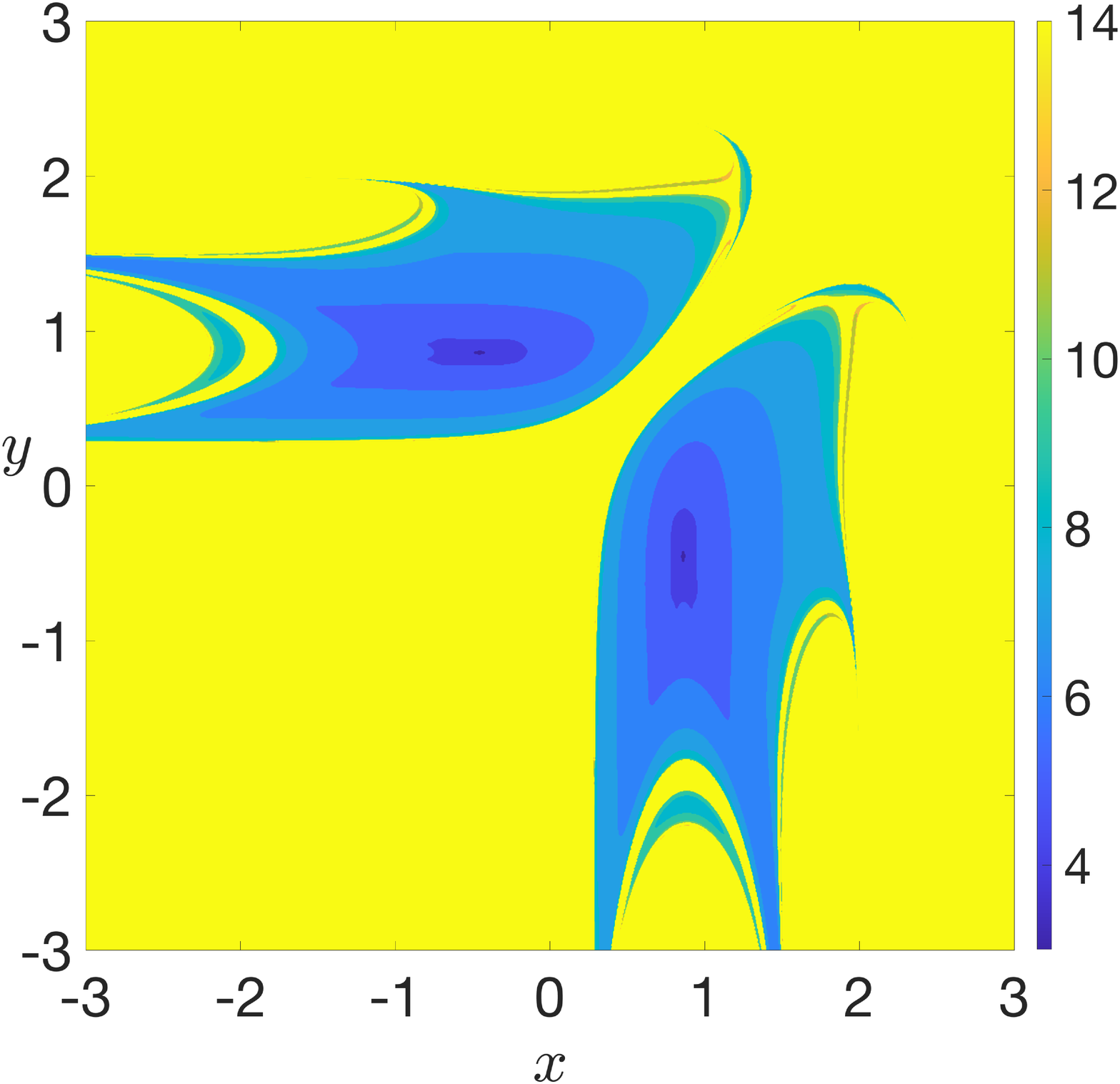}
\end{subfigure}
\hfill
\begin{subfigure}{.49\textwidth}
\centering
\includegraphics[width=\textwidth]{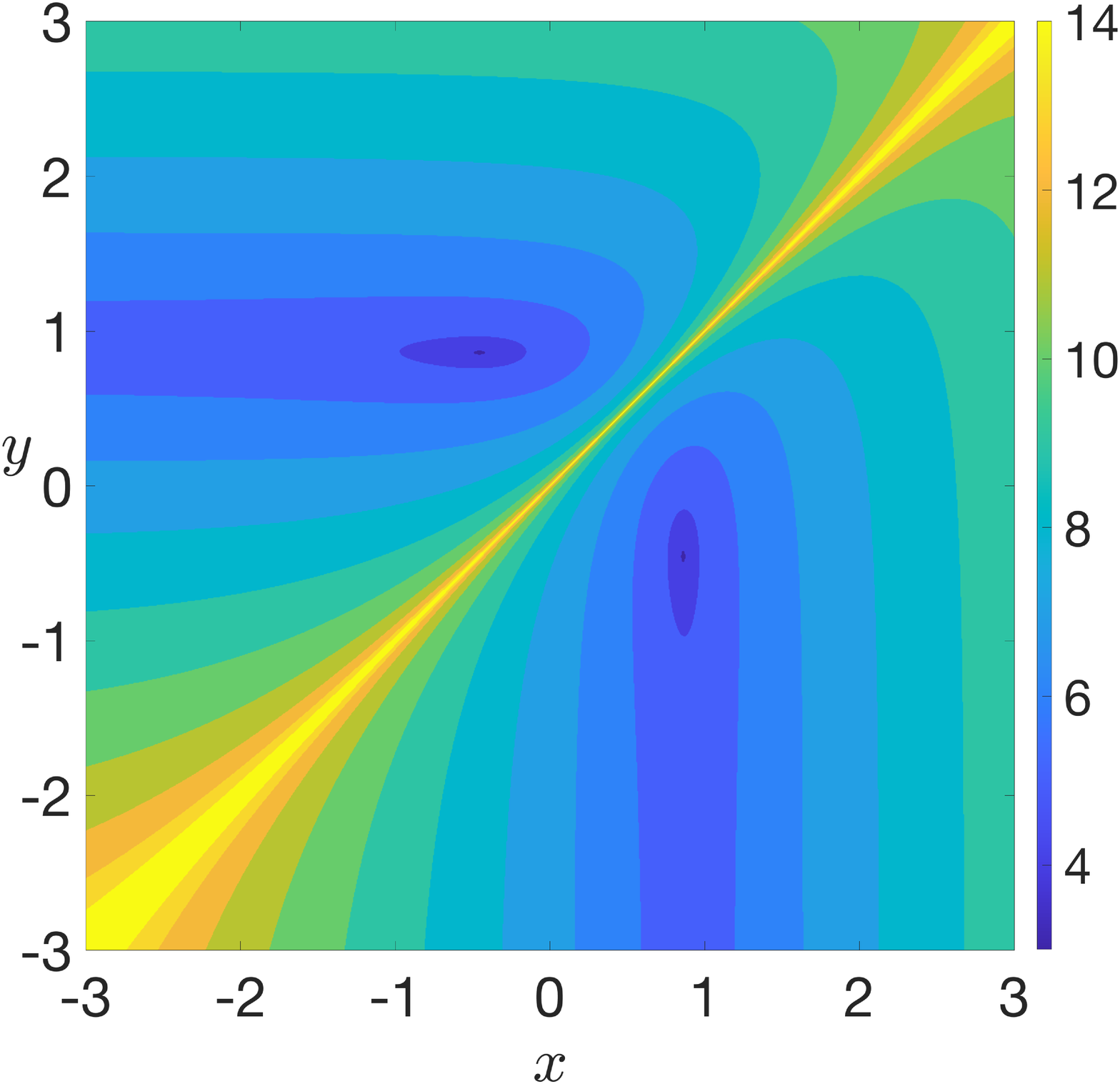}
\end{subfigure}
\begin{subfigure}{.5\textwidth}
\centering
\includegraphics[width=\textwidth]{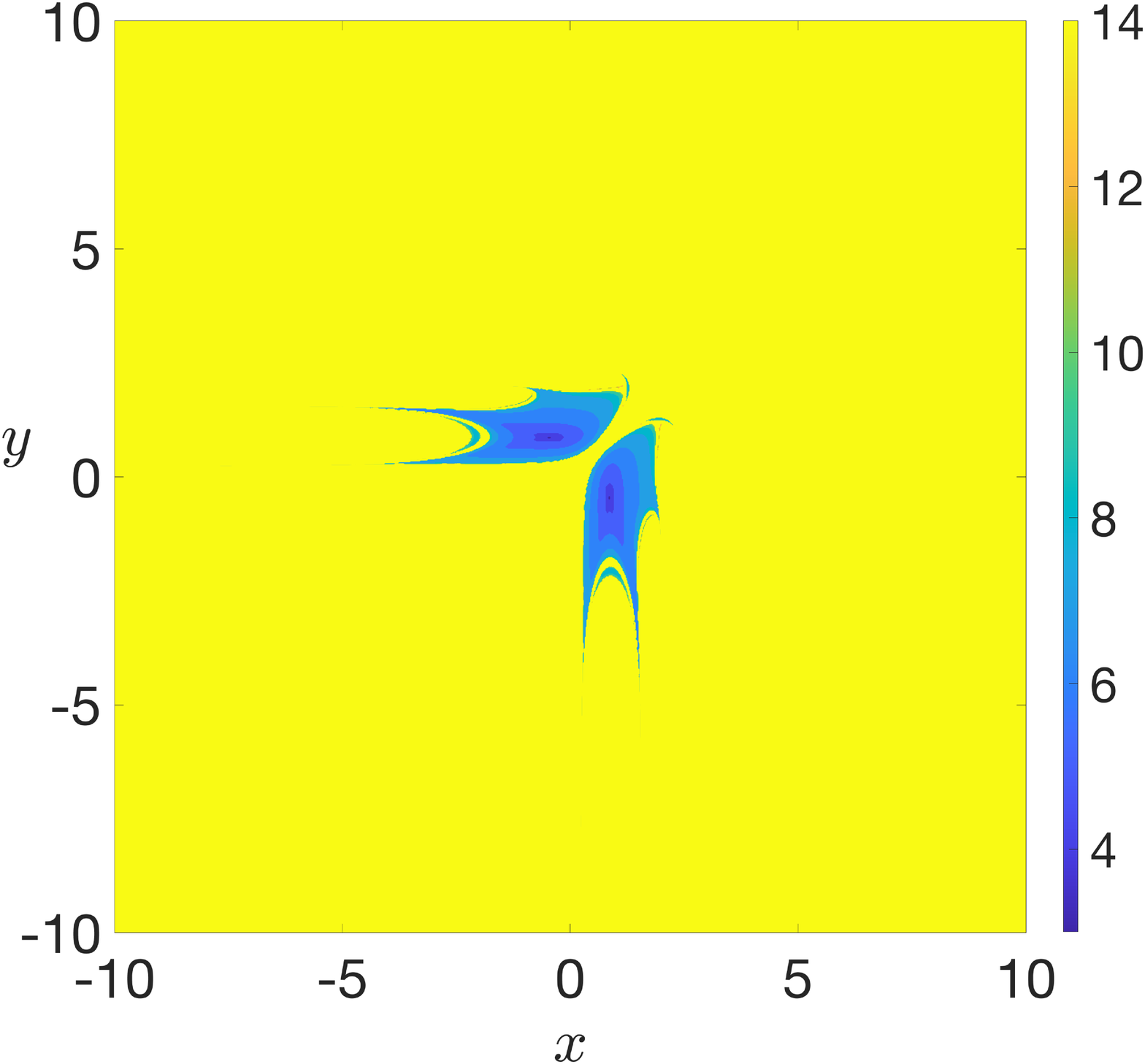}
\label{fig:cla3}
\caption{$s(x) = (x_1, x_2)$}
\end{subfigure}
\hfill
\begin{subfigure}{.49\textwidth}
\centering
\includegraphics[width=\textwidth]{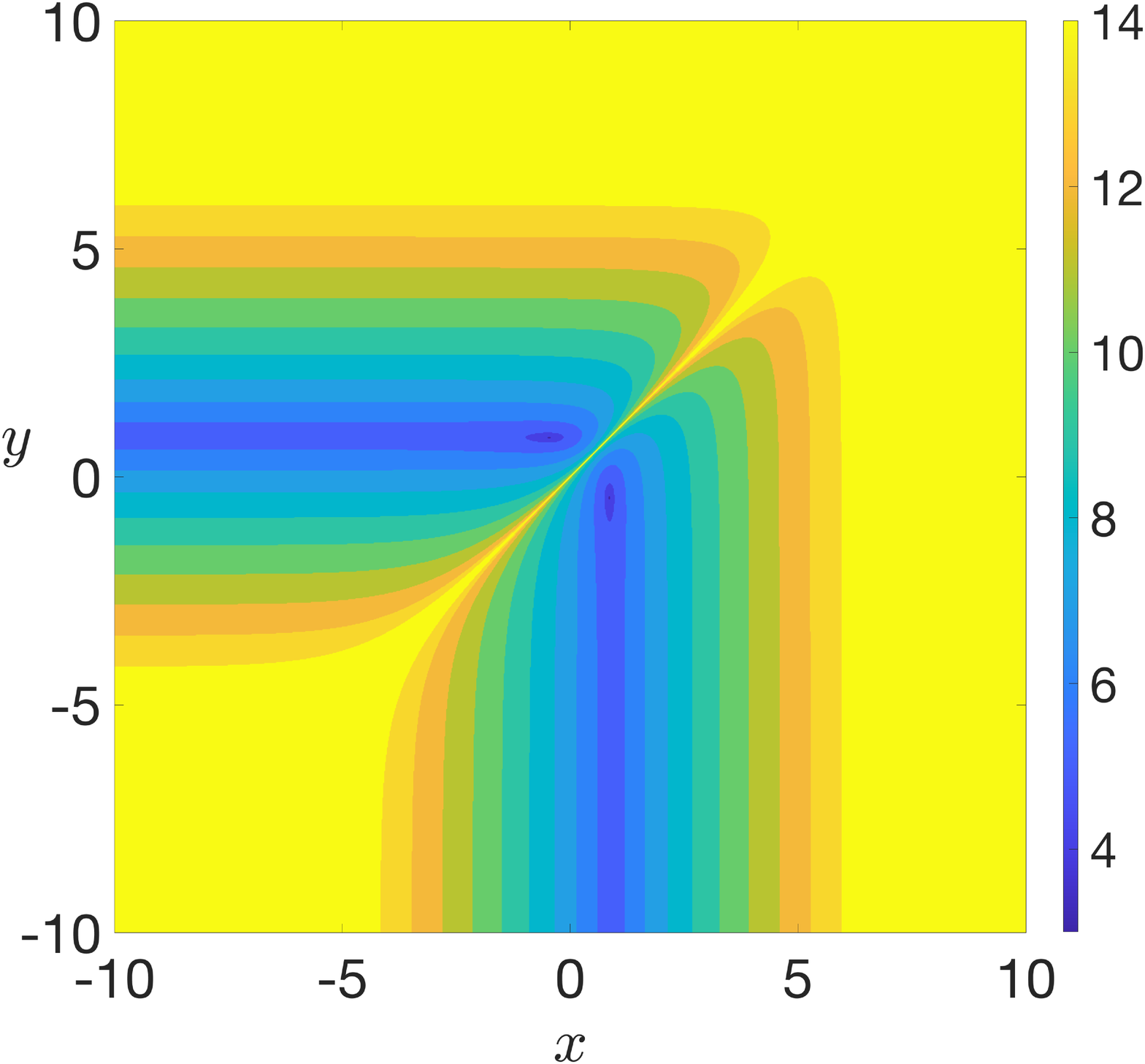}
\label{fig:gen3}
\caption{$s(x) = (e^{x_1}, e^{x_2})$}
\end{subfigure}
\caption{\small\sf System~\eqref{sys:Jennrich}: Portraits of colour-coded number of iterations required for convergence.}
\label{fig:Jennrich}
\end{figure}

Table~\ref{tbl2:Jennrich} provides some statistical data as in the case of Table~\ref{tbl:quartics} for System~\eqref{sys:quartics} in the previous subsection.  It should be noted that the percentage success rates in the table are in agreement with the percentage of the regions which are not yellow in Figure~\ref{fig:Jennrich}, for the cases of $s(x) = (x_1, x_2)$ and $s(x) = (e^{x_1}, e^{x_2})$.  Table~\ref{tbl2:Jennrich} also includes other choices of $s$ for a wider comparison.

\begin{table}[H]
    \centering
{\small
\begin{tabular}{c|rr|rr|c}
 & \multicolumn{2}{|c|}{$[-3,3]^2$} & \multicolumn{2}{c|}{$[-10,10]^2$} & CPU time/ \\[0.5mm]
\cline{2-3} \cline{4-5}
 & Ave & Success & Ave & Success & successful \\
$s_i(x)$ & iter & rate [\%] & iter & rate [\%] & iter [sec] \\[1mm] \hline\\[-4mm]
$x_i$ & 6.6 & 25.0 & 6.7 & 2.4 & $2.7\times10^{-6}$ \\
$x_i^3$ & 7.3 & 12.3 & 7.3 & 1.1 &  $4.6\times10^{-6}$ \\ 
$\sinh(x_i)$ & 6.2 & 17.4 & 6.2 & 1.6 & $2.9\times10^{-6}$ \\
$e^{x_i}$ & 7.8 & 98.3 & 9.6 & 53.3 & $6.7\times10^{-6}$ \\
$\tan{x_i}$ & 6.1 & 9.4 & 6.4 & 10.0 & $2.9\times10^{-6}$ \\[1mm] \hline
\end{tabular}
\caption{\small\sf System~\eqref{sys:Jennrich}: Performance of the classical and generalized Newton methods with one million randomly generated starting points in domains of various sizes.}
\label{tbl2:Jennrich}}
\end{table}

When successful the CPU time one iteration of the classical Newton method spends on the average (over the domain $[-3,3]^2$) is $2.7\times10^{-6}$ sec.  The same CPU time for the exp-generalized Newton method is $6.7\times10^{-6}$ sec, which is about 2.5 times longer. On the other hand, over the domain $[-3,3]^2$, the chance of finding a solution for the exp-generalized method in less than 14 iterations is nearly 4 times higher than using the classical method.  Moreover, over the domain $[-10,10]^2$, the exp-generalized method is 23 times more likely to find a solution in the same manner.  These likelihoods of success which are greatly in favour of the exp-generalized method seems to offset the higher computational times per iteration.

To make sure of the conclusion we have just drawn above as to which method is preferred, we can again prepare a table listing the average CPU time needed for a successful run by each method, as it was previously done in Table~\ref{tbl2:quartics} for System~\eqref{sys:quartics}.  Table~\ref{tbl3:Jennrich} lists these times, which immediately reconfirms that the exp-generalized method should indeed be the preferred method over $[-3,3]^2$, as it would take the classical method 37\% more time to find a solution.  Over the larger domain $[-10,10]^2$, by the time the classical method finds a solution the exp-generalized method will have already found about seven solutions---see the framed CPU times.

\begin{table}[H]
    \centering
{\small
\begin{tabular}{c|cc}
 & \multicolumn{2}{c}{Time needed to get a single soln [sec]} \\[1mm]
\cline{2-3} \\[-4mm]
$s_i(x)$ & $[-3,3]^2$ & $[-10,10]^2$ \\[1mm] \hline\\[-4mm]
$x_i$ & $7.1\times10^{-5}$ & $7.5\times10^{-4}$ \\
$x_i^3$ & $2.7\times10^{-4}$ & $3.1\times10^{-3}$ \\ 
$\sinh(x_i)$ & $1.0\times10^{-4}$ & $1.1\times10^{-3}$ \\
$e^{x_i}$ & \framebox{$5.2\times10^{-5}$} & \framebox{$1.1\times10^{-4}$} \\
$\tan{x_i}$ & $1.9\times10^{-4}$ & $1.9\times10^{-4}$ \\[1mm] \hline
\end{tabular}
\caption{\small\sf System~\eqref{sys:Jennrich}: CPU time needed on the average by the classical and generalized Newton methods to obtain a solution in less than 14 iterations, based on the data in Table~\ref{tbl2:Jennrich}.}
\label{tbl3:Jennrich}}
\end{table}

This is yet another example which clearly illustrates how the structure of the problem can be exploited to solve a system of equations by means of a generalized Newton method.

\subsection{Cubic equations in two variables}
\label{ss:2var}

The example we deal with in this section emanates from an unconstrained global optimization problem solved in~\cite{BuraKaya2019}, which asks to minimize the function $\varphi:\dR^2\to\dR$ given as
\begin{equation}\label{eqn:twoVar}
    \varphi(x) = (x_1^2 - 1)^2 + (x_2^2 - 2)^2 - 0.7\,x_1\,x_2 + 0.2\,x_1 + 0.3\,x_2\,.
\end{equation}
Although the numerical optimization method proposed in~\cite{BuraKaya2019} can find the global minimizer of $\varphi$, common numerical optimization approaches often only find a stationary point of the function $\varphi$, by finding a zero of the gradient of $\varphi$, namely, effectively, they find a solution to the system of equations
\begin{equation} \label{sys:cubic}
f(x) := \nabla\varphi(x) = \left[\begin{array}{c}
4\,x_1^3 - 4\,x_1 - 0.7\,x_2 + 0.2 \\[1mm]
4\,x_2^3 - 8\,x_2 - 0.7\,x_1 + 0.3
\end{array} \right] = {\bf 0}\,.
\end{equation}
In~\cite{BuraKaya2019}, five extremal solutions of the function in~\eqref{eqn:twoVar} are listed as in Table~\ref{tbl:twoVarSol}. Solutions 1--4 are all local minima while Solution~5 is a local maximum. 
\begin{table}[H]
    \centering
    {\small
    \begin{tabular}{c|r|r}
        Soln & $x_1$\hspace{16mm} & $x_2$\hspace{16mm} \\ \hline
        $1$ & $-1.128494496205920$ & $-1.477960288994776$ \\
        $2$ &$ 1.088972069871674$ & $1.442265902284124$ \\
        $3$ & $0.79262879889394$ &  $-1.398008585571904$ \\
        $4$ & $-0.888779137505495$ & $1.352613115553849$ \\
        $5$ & $0.044197271093630$ & $0.033651793151170$
    \end{tabular}
    \caption{\small\sf Extremal solutions of $\varphi(x)$ in \eqref{eqn:twoVar}.}
    \label{tbl:twoVarSol}}
\end{table}

The appearance of $x_1^3$ and $x_2^3$ in~\eqref{sys:cubic} prompts us to consider $s(x) = (x_1^3, x_2^3)$ as the first generalized method in Table~\ref{tbl1:cubic}.  Via experiments we observe that the choice $s(x) = (\sinh x_1, \sinh x_2)$ yields another worthwhile generalization of the Newton method.  Table~\ref{tbl1:cubic} reveals that the estimates of $\lambda$ for both of the generalized methods are (by two to four times) smaller at Solutions~1--4, and larger only at Solution~5.

\begin{table}[H]
\centering
{\small
\begin{tabular}{ccccc}
Soln & $s_i(x)$ & $[\|\mu(x^*)\|/2,\ \|\rho(x^*)\|/2]$ & $\lambda$ & $\lambda_N/\lambda_{GN}$ \\[1mm] \hline \\[-4mm]
1 & $x_i$ & $[0.08, 1.5]$ & $1.2$ &  \\
 & $x_i^3$ & $[0.08, 0.44]$ & $0.3$ & $4.0$  \\
 & $\sinh(x_i)$ & $[0.08, 0.96]$ & $0.8$ & $1.5$  \\[1mm] \hline \\[-4mm]
 
2 & $x_i$ & $[0.09, 1.6]$ & $1.0$ &  \\
 & $x_i^3$ & $[0.09, 0.5 ]$ & $0.3$ & $3.3$  \\
 & $\sinh(x_i)$ & $[0.09, 1.1]$ & $0.6$ & $1.7$  \\[1mm] \hline \\[-4mm]
 
3 & $x_i$ & $[0, 2.9]$ & $2.7$ &  \\
 & $x_i^3$ & $[0, 1.5]$ & $1.4$ & $1.9$  \\
 & $\sinh(x_i)$ & $[0, 2.5]$ & $2.4$ & $1.1$  \\[1mm] \hline \\[-4mm]
 
4 & $x_i$ & $[0, 2.3]$ & $1.2$ &  \\
 & $x_i^3$ & $[0, 0.94]$ & $0.4$ & $3.0$  \\
 & $\sinh(x_i)$ & $[0, 1.8]$ & $0.7$ & $1.7$  \\[1mm] \hline \\[-4mm]
 
5 & $x_i$ & $[0, 0.14]$ & $0.05$ &  \\
 & $x_i^3$ & $[0,37]$ & $29.8$ & $0.002$  \\
 & $\sinh(x_i)$ & $[0, 0.17]$ & $0.16$ & $0.3$  \\[1mm] \hline \\[-4mm]
\end{tabular}
\caption{\small\sf System~\eqref{sys:cubic}: Asymptotic error constants of the classical and generalized Newton methods.}
\label{tbl1:cubic}}
\end{table}

To illustrate the overall behaviour, Figure~\ref{fig:cubics} provides a visualization of the success of three methods in terms of the number of iterations.  In addition to the classical method, we consider the cube- and sinh-generalized methods.

\begin{figure}[t!]
\centering
\begin{subfigure}{.32\textwidth}
\centering
\includegraphics[width=\textwidth]{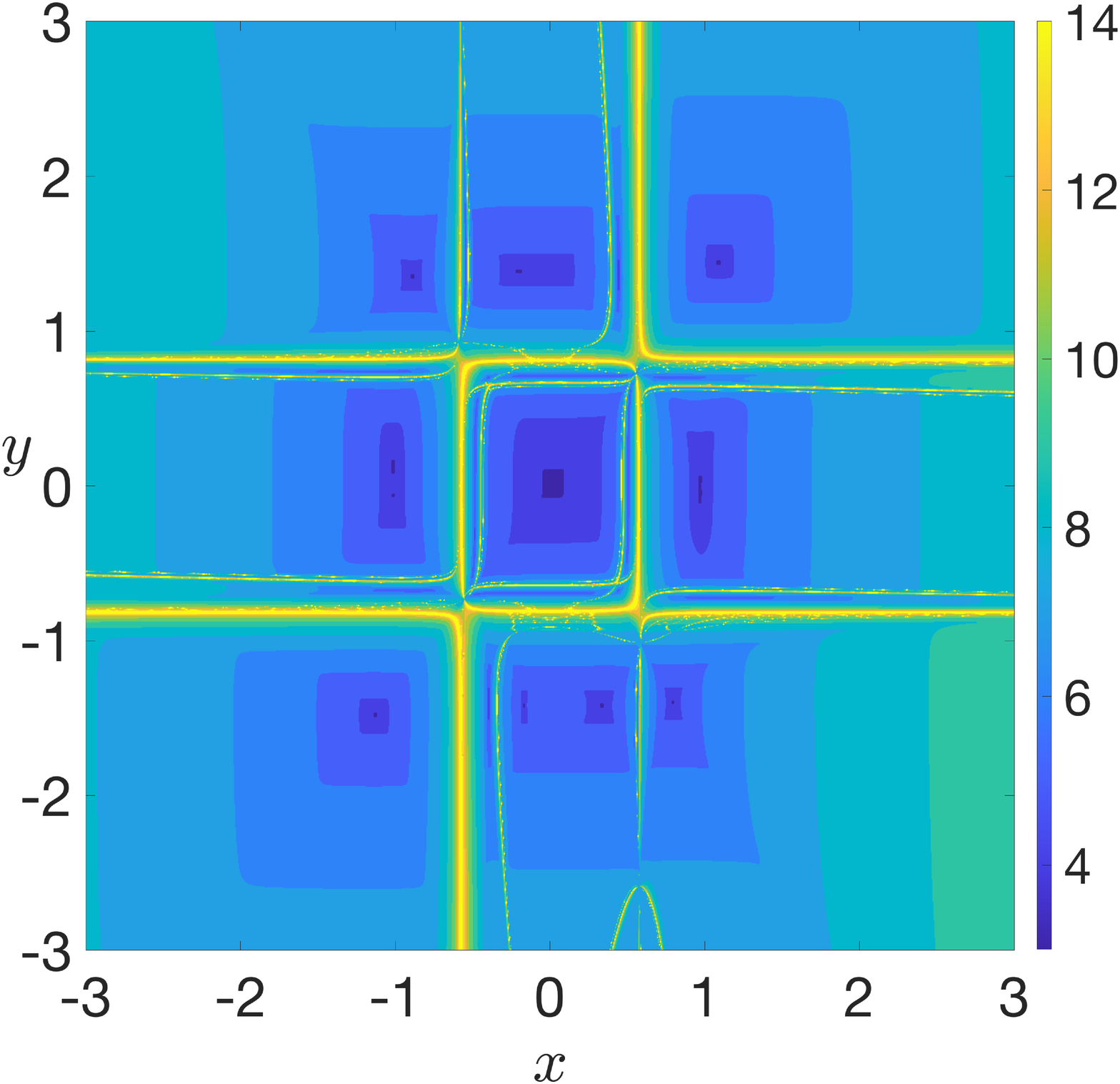}
\end{subfigure}
\begin{subfigure}{.32\textwidth}
\centering
\includegraphics[width=\textwidth]{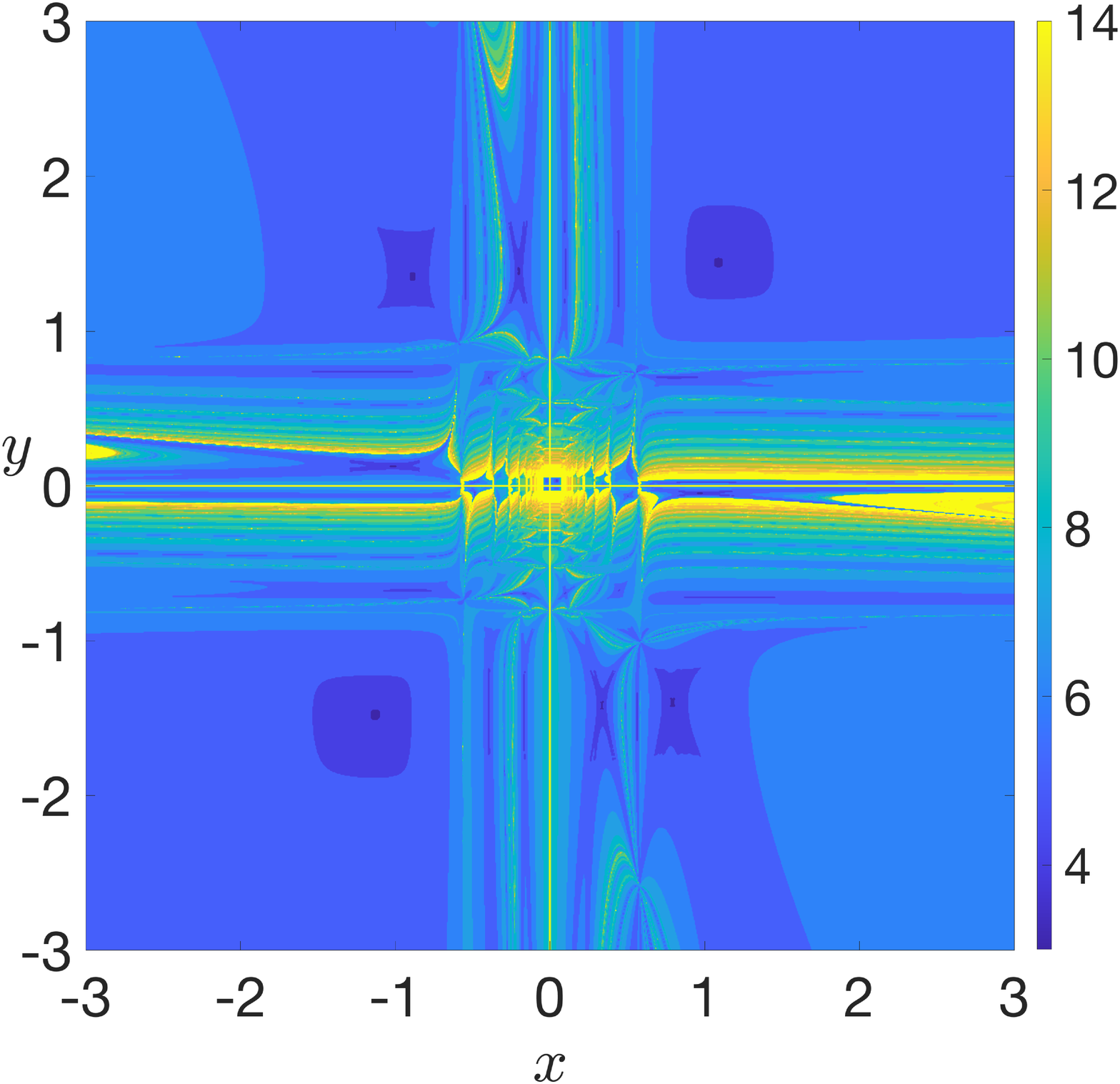}
\end{subfigure}
\begin{subfigure}{.32\textwidth}
\centering
\includegraphics[width=\textwidth]{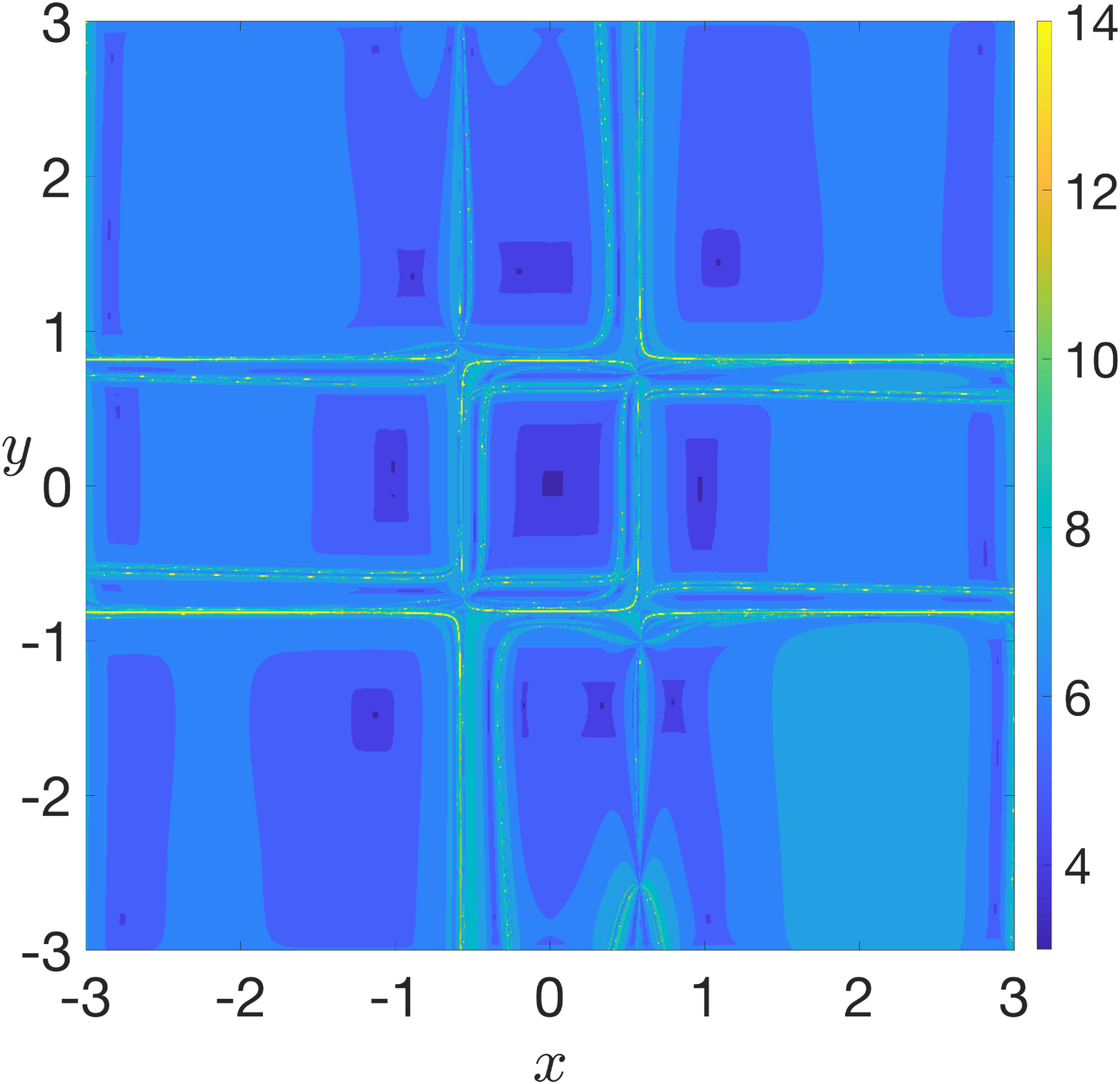}
\end{subfigure}
\begin{subfigure}{.32\textwidth}
\centering
\includegraphics[width=\textwidth]{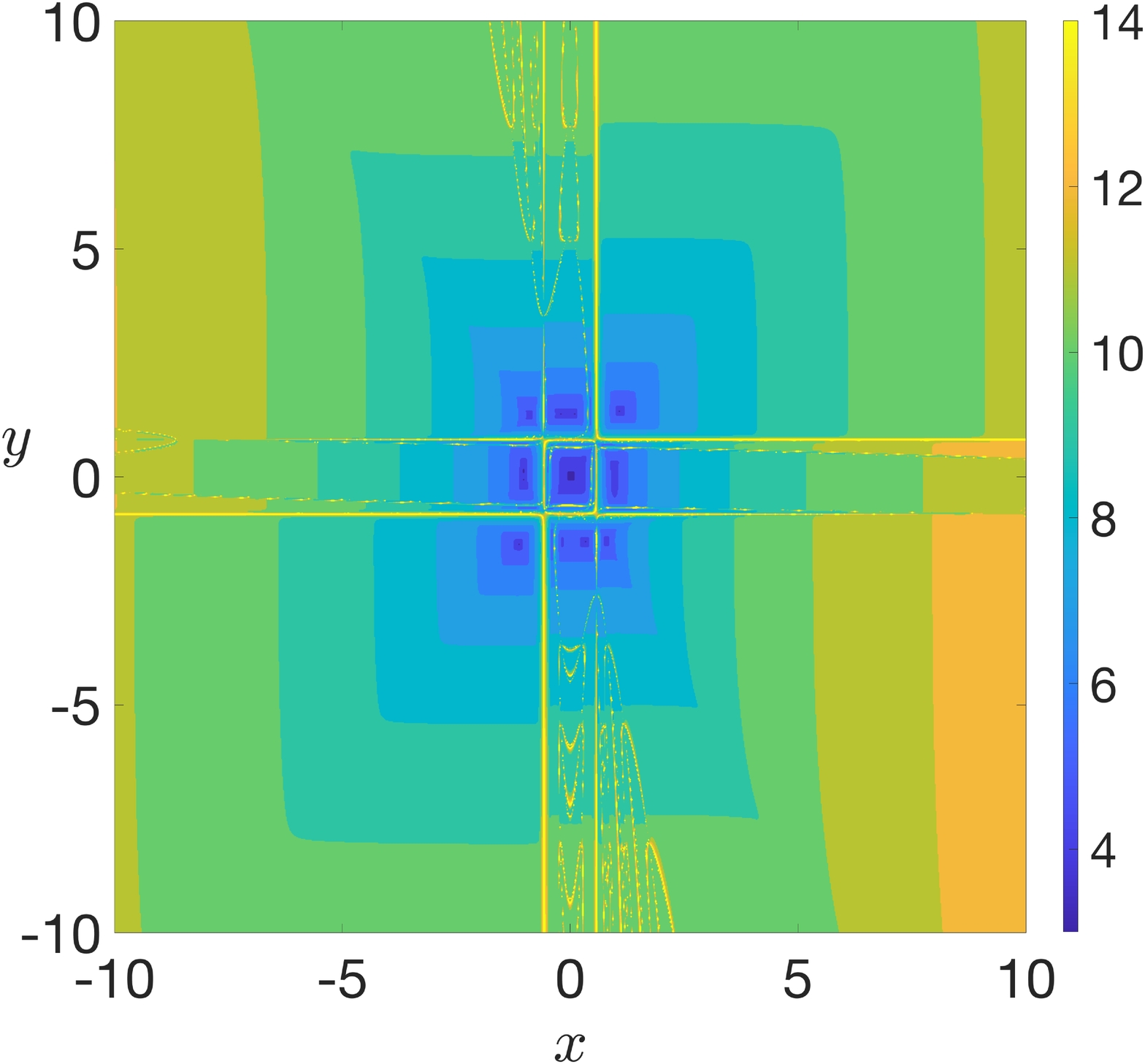}
\end{subfigure}
\begin{subfigure}{.32\textwidth}
\centering
\includegraphics[width=\textwidth]{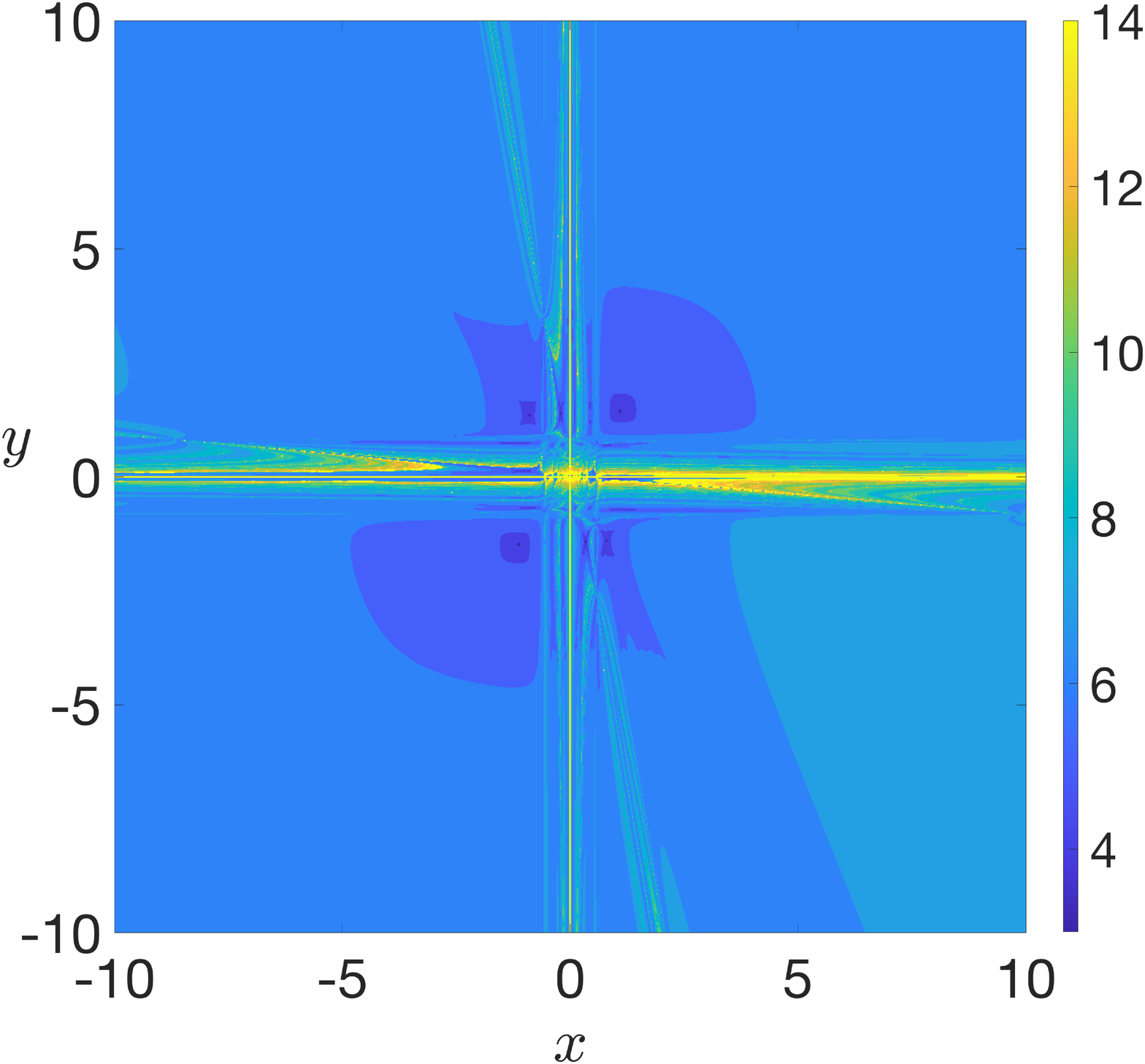}
\end{subfigure}
\begin{subfigure}{.32\textwidth}
\centering
\includegraphics[width=\textwidth]{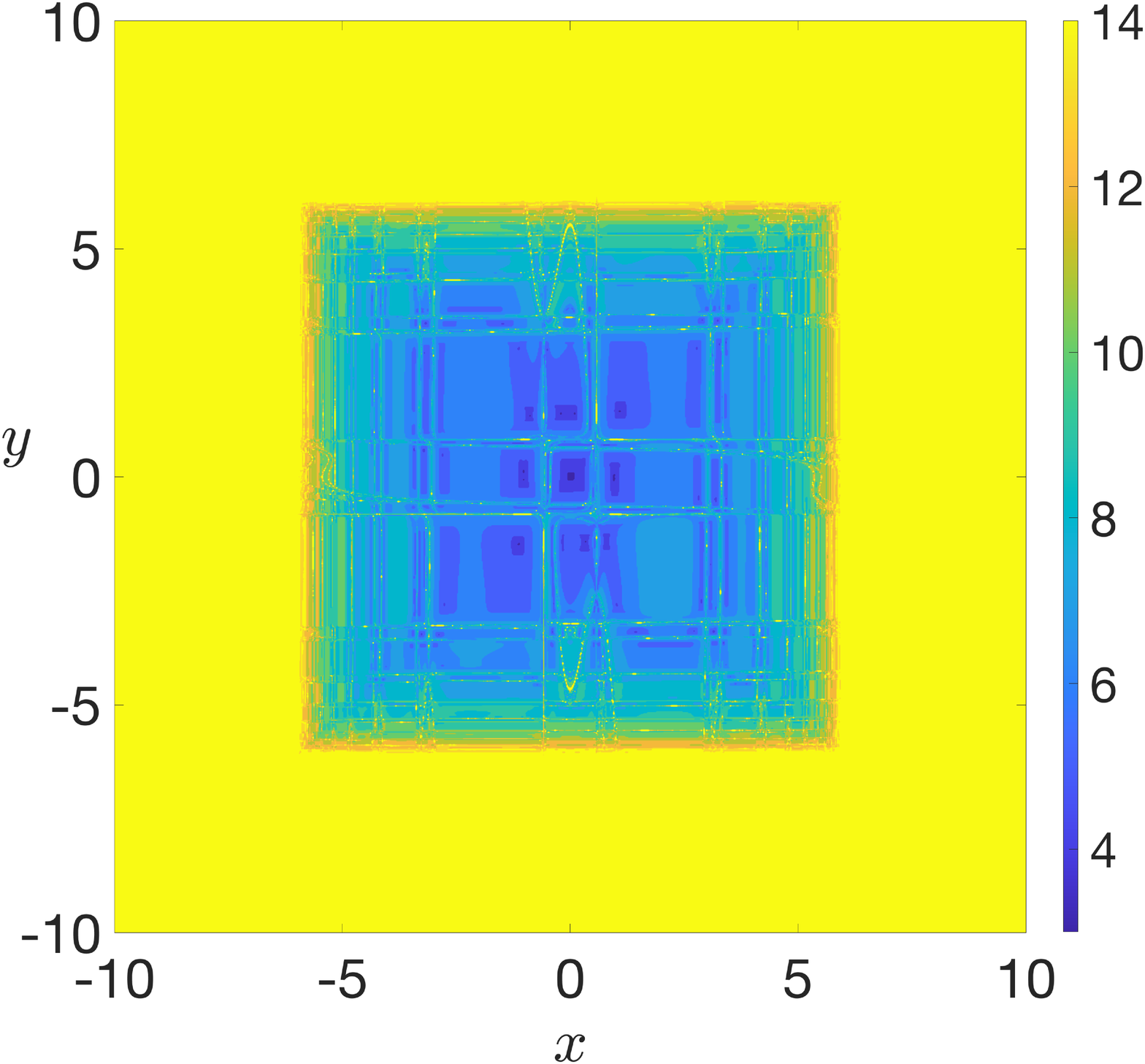}
\end{subfigure}
\begin{subfigure}{.32\textwidth}
\centering
\includegraphics[width=\textwidth]{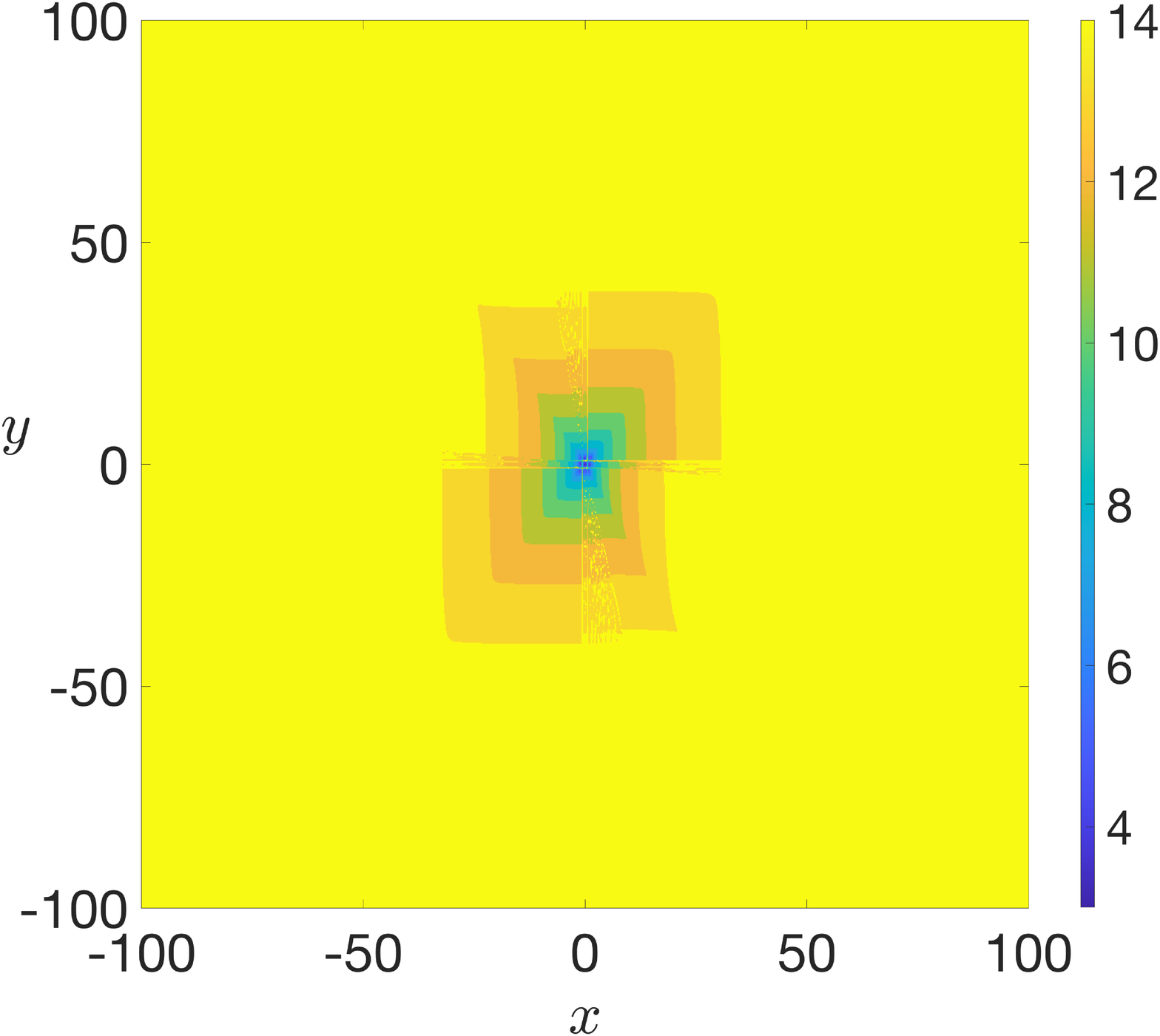}
\caption{$s(x) = (x_1, x_2)$}
\label{fig:cubic_classic}
\end{subfigure}
\begin{subfigure}{.32\textwidth}
\centering
\includegraphics[width=\textwidth]{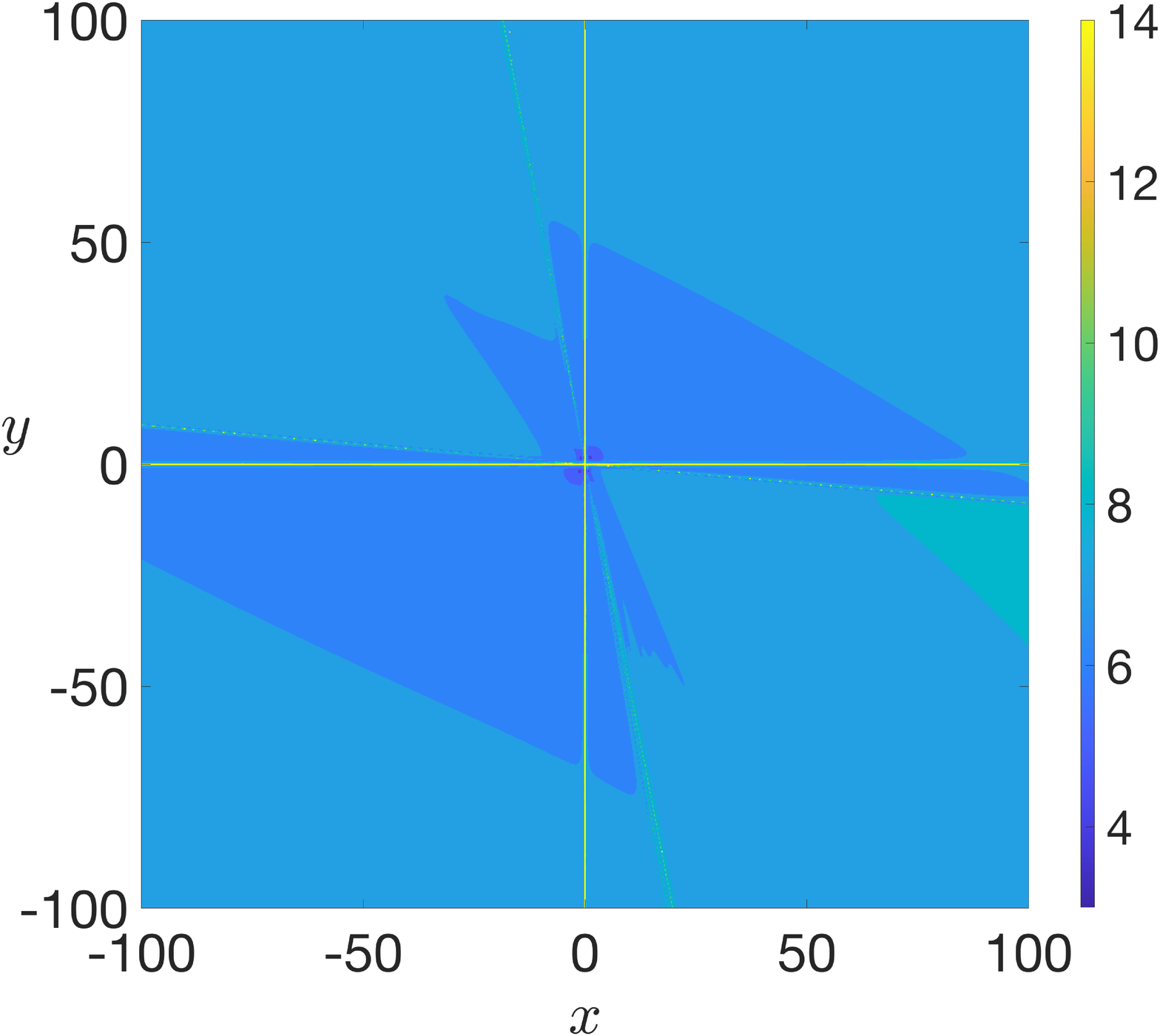}
\caption{$s(x) = (x_1^3, x_2^3)$}
\label{fig:cubic_cube}
\end{subfigure}
\begin{subfigure}{.32\textwidth}
\centering
\includegraphics[width=\textwidth]{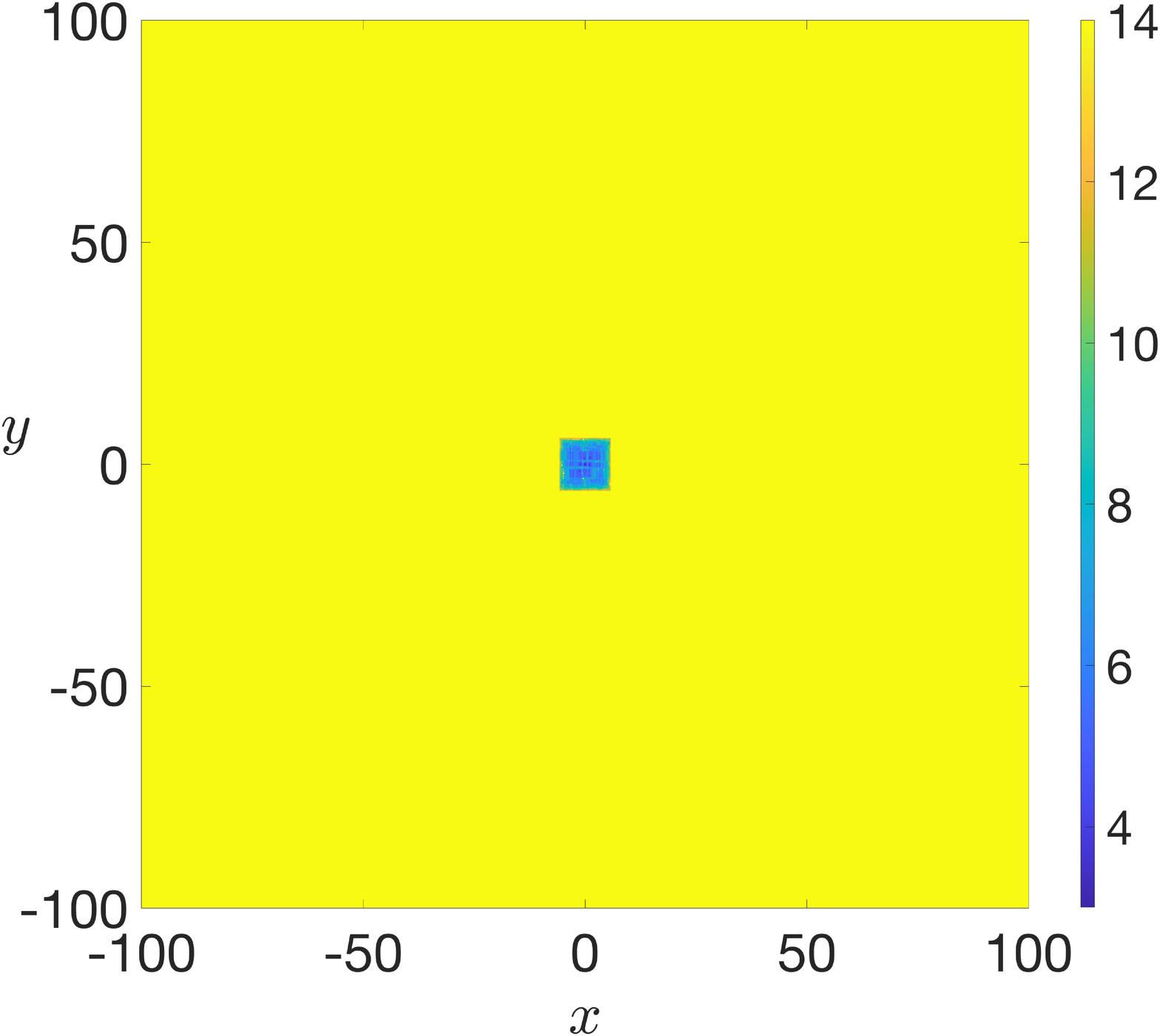}
\caption{$s(x) = (\sinh x_1, \sinh x_2)$}
\label{fig:cubic_sinh}
\end{subfigure}
\caption{\small\sf System~\eqref{sys:cubic}: Portraits of colour-coded number of iterations required for convergence.}
\label{fig:cubics}
\end{figure}

Glancing at the $3\times3$ matrix of graphs of Figure~\ref{fig:cubics}, while the graphs in the entries (2,2) and (3,2) have more of the shades of blue than those in the same rows, the graph in (1,3) appears to have more of the shades of blue and almost no yellow.  The success rates by judging from the non-yellow regions in these graphs are corroborated by the success rates presented in Table~\ref{tbl2:cubic}.  What seem to be the best-performing methods by looking at these graphs are also in agreement with the ones corresponding to the framed entries in Table~\ref{tbl3:cubic}.  

\begin{table}[H]
    \centering
{\small
\begin{tabular}{c|rr|rr|rr|c}
 & \multicolumn{2}{|c|}{$[-3,3]^2$} & \multicolumn{2}{c|}{$[-10,10]^2$} & \multicolumn{2}{c|}{$[-100,100]^2$} & CPU time/ \\[0.5mm]
\cline{2-3} \cline{4-5} \cline{6-7}
 & Ave & Success & Ave & Success & Ave & Success & successful \\
$s_i(x)$ & iter & rate [\%] & iter & rate [\%] & iter & rate [\%] & iter [sec] \\[1mm] \hline\\[-4mm]
$x_i$ & 7.0 & 98.6 & 9.7 & 99.3 & 12.2 & 9.8 & $3.7\times10^{-6}$ \\
$x_i^3$ & 6.1 & 98.6 & 6.3 & 99.7 & 6.8 & 100.0 & $6.1\times10^{-6}$ \\
$\sinh(x_i)$ & 5.9 & 99.8 & 7.9 & 34.8 & 7.8 & 0.3 & $4.1\times10^{-6}$ \\
$e^{x_i}$ & 7.1 & 98.7 & 10.4 & 42.4 & 10.4 & 0.4 & $4.7\times10^{-6}$ \\
$\tan{x_i}$ & 6.7 & 70.7 & 7.3 & 57.5 & 7.8 & 3.3 & $4.0\times10^{-6}$ \\[1mm] \hline
\end{tabular}
\caption{\small\sf System~\eqref{sys:cubic}: Performance of the classical and generalized Newton methods with one million randomly generated starting points in domains of various sizes.}
\label{tbl2:cubic}}
\end{table}

By looking at Table~\ref{tbl3:cubic}, we deduce easily that the $\sinh$-generalized method is the best, although the classical method is only slightly worse, in terms of the time they take for a successful run in the domain $[-3,3]^2$. The classical Newton is the best for $[-10,10]^2$, with this time the cubic-generalized method being slightly worse, taking 8\% longer time in finding a solution. In the largest domain $[-100,100]^2$,  the cube-generalized method is by far the best, as its nearest contender, the classical Newton method, takes more than 11 times longer to obtain a single solution.

\begin{table}[H]
    \centering
{\small
\begin{tabular}{c|ccc}
 & \multicolumn{3}{c}{Time needed to get a single soln [sec]} \\[1mm]
\cline{2-4} \\[-4mm]
$s_i(x)$ & $[-3,3]^2$ & $[-10,10]^2$ & $[-100,100]^2$ \\[1mm] \hline\\[-4mm]
$x_i$ & $2.6\times10^{-5}$ & \framebox{$3.6\times10^{-5}$} & $4.6\times10^{-4}$ \\
$x_i^3$ & $3.8\times10^{-5}$ & $3.9\times10^{-5}$ & \framebox{$4.1\times10^{-5}$} \\ 
$\sinh(x_i)$ & \framebox{$2.4\times10^{-5}$} & $9.3\times10^{-5}$ & $1.1\times10^{-2}$ \\
$e^{x_i}$ & $3.4\times10^{-5}$ & $1.2\times10^{-4}$ & $1.2\times10^{-2}$ \\
$\tan{x_i}$ & $3.8\times10^{-5}$ & $5.1\times10^{-5}$ & $9.5\times10^{-4}$ \\[1mm] \hline
\end{tabular}
\caption{\small\sf System~\eqref{sys:cubic}: CPU time needed on the average by the classical and generalized Newton methods to obtain a solution in less than 14 iterations, based on the data in Table~\ref{tbl2:cubic}.}
\label{tbl3:cubic}}
\end{table}

Going back to Figure~\ref{fig:cubics}, we deduce from the first row of graphs that the regions of convergence in 4--6 iterations of the classical method are considerably enlarged by both of the generalized methods.  This is in agreement with the estimated values of $\lambda_N/\lambda_{GN}$ in Table~\ref{tbl1:cubic}.

\subsection{Cubic equations in six variables}
\label{ss:6var}

We consider another system of cubic equations, but this time the number of equations and unknowns is six.  The system originates from the problem of (globally) minimizing the function $\varphi:\dR^6\to\dR$, which was studied in \cite{BuraKaya2019, QiWanYang2004}, given by
\begin{equation}\label{eqn:sixVar}
    \varphi(x) = \sum_{i=1}^6a_ix_i^4+x^TB\,x+d^Tx\,,
\end{equation}
where
\[
a =\begin{bmatrix} 9 \\ 2 \\ 6 \\ 4 \\ 8 \\ 7 \end{bmatrix},\quad 
B = \begin{bmatrix} 4 & 4 & 9 & 3 & 4 & 1 \\ 4 & 3 & 7 & 9 & 9 & 2 \\ 9 & 7 & 4 & 7 & 6 & 6 \\ 3 & 9 & 7 & 4 & 2 & 6 \\ 4 & 9 & 6 & 2 & 8 & 3 \\ 1 & 2 & 6 & 6 & 3 & 5 \end{bmatrix},\quad 
d = \begin{bmatrix} 2 \\ 6 \\ 5 \\ 0 \\ 0 \\ 2 \end{bmatrix}.
\]
We consider the problem of finding the zeroes of the gradient $\nabla\varphi(x)$ of $\varphi(x)$, in other words, the zeroes of
\begin{equation} \label{sys:6var}
f(x) := \nabla\varphi(x) = 4\begin{bmatrix} a_1\,x_1^3 \\ a_2\,x_2^3 \\ \vdots \\ a_6\,x_6^3 \end{bmatrix} + 2\,B\,x + d = {\bf 0}\,.
\end{equation}
Solutions of \eqref{sys:6var} are stationary points of $\varphi$, in other words, they are candidates for (locally) optimal solutions of $\varphi$, three of which are listed in Table~\ref{tbl:sixVarSol}.  The first solution listed in Table~\ref{tbl:sixVarSol} is a global minimizer of $\varphi$, as reported in~\cite{BuraKaya2019}.  Our aim here is to look at the behaviour of the classical and generalized methods in finding a zero of $f$, which is only a stationary point of $\varphi$.

\begin{table}[H]
    \centering
    {\small
    \begin{tabular}{r|r|r|r}
              & Soln 1\hspace*{12mm} & Soln 2\hspace*{12mm} & Soln 3\hspace*{12mm} \\ \hline
        $x_1$ & $0.545218813388361$ & $-0.599208065573669$ & $0.590580847289543$ \\
        $x_2$ & $-1.464410189791729$ & $-1.571013884485518$ & $1.338889774602320$ \\
        $x_3$ & $-0.720606654276266$ & $0.678323332400517$ & $-0.853265510869097$ \\
        $x_4$ & $1.178144265591973$ & $1.076080413893220$ & $-0.955745102979906$ \\
        $x_5$ & $0.794065108243717$ & $0.745744375791400$ & $-0.646924271685709$ \\
        $x_6$ & $-0.465794119447879$ & $-0.762615830412707$ & $0.708688334528434$
    \end{tabular}
    \caption{Some of the stationary points of $\varphi$ in~\eqref{eqn:sixVar}.}
    \label{tbl:sixVarSol}}
\end{table}

First we look at the (local) behaviour near the solutions listed in Table~\ref{tbl:sixVarSol}.  Table~\ref{tbl1:6var} tabulates the theoretical intervals where $\lambda$ lies, found using Theorem~\ref{thm:lambdaBounds}, as well as the $\lambda$ estimated numerically, for each method.  We observe that the numerical estimates fall into the theoretical intervals.  We also observe that the cube-generalized method has $\lambda$ consistently 2 to 3.5 times smaller than that of the classical method, and therefore locally faster by the same factors. The sinh-generalized method, on the other hand, is observed to be not so fast.  The reason we have included the sinh-generalized method here is that as we will see in Tables~\ref{tbl2:6var} and \ref{tbl3:6var} it can have a desirable performance on a larger scale, in search domains of moderate size.

\begin{table}[H]
\centering
{\small
\begin{tabular}{ccccc}
Soln & $s_i(x)$ & $[\|\mu(x^*)\|/2,\ \|\rho(x^*)\|/2]$ & $\lambda$ & $\lambda_N/\lambda_{GN}$ \\[1mm] \hline \\[-4mm]
1 & $x_i$ & $[0,3.4]$ & $0.7$ &  \\
 & $x_i^3$ & $[0,1.2]$ & $0.2$ & $3.5$ \\
 & $\sinh x_i$ & $[0,2.7]$ & $1.3$ & $0.5$ \\[1mm] \hline \\[-4mm]
 
2 & $x_i$ & $[0,3.1]$ & $0.8$ &  \\
 & $x_i^3$ & $[0,1.0]$ & $0.4$ & $2$ \\
 & $\sinh x_i$ & $[0,2.4]$ & $0.5$ & $1.6$ \\[1mm] \hline \\[-4mm]
 
3 & $x_i$ & $[0,3.1]$ & $0.9$ &  \\
 & $x_i^3$ & $[0,0.9]$ & $0.4$ & $2.3$ \\
 & $\sinh x_i$ & $[0,2.3]$ & $0.8$ & $1.1$ \\[1mm] \hline \\[-4mm]
\end{tabular}
\caption{\small\sf System~\eqref{sys:6var}: Asymptotic error constants of the classical and generalized Newton methods.}
\label{tbl1:6var}}
\end{table}

Since System~\eqref{sys:6var} has six variables, we cannot have the kind of visualization of performance as we had in the previous (two-variable) examples.  However, we can still carry out runs with randomized (one million) initial points and make some statistical observations as we did for the previous example systems. Table~\ref{tbl2:6var}, and subsequently Table~\ref{tbl3:6var}, provide advice as to which method can be chosen for efficiency.

\begin{table}[H]
    \centering
{\small
\begin{tabular}{c|rr|rr|rr|c}
 & \multicolumn{2}{|c|}{$[-3,3]^2$} & \multicolumn{2}{c|}{$[-10,10]^2$} & \multicolumn{2}{c|}{$[-100,100]^2$} & CPU time/ \\[0.5mm]
\cline{2-3} \cline{4-5} \cline{6-7}
 & Ave & Success & Ave & Success & Ave & Success & successful \\
$s_i(x)$ & iter & rate [\%] & iter & rate [\%] & iter & rate [\%] & iter [sec] \\[1mm] \hline\\[-4mm]
$x_i$ & 10.5 & 58.8 & 11.9 & 41.2 & $-$ & 0.0 & $6.3\times10^{-6}$ \\  
$x_i^3$ & 8.0 & 76.7 & 8.5 & 48.9 & 8.8 & 17.7 & $9.5\times10^{-6}$ \\ 
$\sinh(x_i)$ & 8.9 & 74.9 & 11.1 & 17.4 & $-$ & 0.0 & $6.9\times10^{-6}$ \\
$e^{x_i}$ & 10.8 & 62.4 & 12.3 & 2.2 & $-$ & 0.0 & $1.0\times10^{-5}$ \\
$\tan{x_i}$ & 9.2 & 3.2 & 9.8 & 0.6 & $-$ & 0.0 & $7.3\times10^{-6}$ \\[1mm] \hline
\end{tabular}
\caption{\small\sf System~\eqref{sys:6var}: Performance of the classical and generalized Newton methods with one million randomly generated starting points in domains of various sizes.}
\label{tbl2:6var}}
\end{table}

The framed average CPU times required to get a single solution in Table~\ref{tbl3:6var} indicate that the sinh-generalized method should be chosen in the domain $[-3,3]^2$, while the cube-generalized method should be preferred in the larger domains. In $[-3,3]^2$, compared to the sinh-generalized method, the cube-generalized method takes about 21\% more time to find a solution, while the classical method requires 34\% more time.  In the search domain $[-100,100]^2$, the cube-generalized method is unrivalled as none of the other methods is viable to use.  We note that in the largest search domain since the other methods has a success rate less than $0.04\%$, their success rates have been entered as $0.0\%$ into Table~\ref{tbl2:6var}, with no average number of iterations reported.

\begin{table}[H]
    \centering
{\small
\begin{tabular}{c|ccc}
 & \multicolumn{3}{c}{Time needed to get a single soln [sec]} \\[1mm]
\cline{2-4} \\[-4mm]
$s_i(x)$ & $[-3,3]^2$ & $[-10,10]^2$ & $[-100,100]^2$ \\[1mm] \hline\\[-4mm]
$x_i$ & $1.1\times10^{-4}$ & $1.8\times10^{-4}$ & $-$ \\
$x_i^3$ & $9.9\times10^{-5}$ & \framebox{$1.7\times10^{-4}$} & \framebox{$4.7\times10^{-4}$} \\ 
$\sinh(x_i)$ & \framebox{$8.2\times10^{-5}$} & $4.4\times10^{-4}$ & $-$ \\
$e^{x_i}$ & $1.7\times10^{-4}$ & $5.6\times10^{-3}$ & $-$ \\
$\tan{x_i}$ & $2.1\times10^{-3}$ & $1.2\times10^{-2}$ & $-$ \\[1mm] \hline
\end{tabular}
\caption{\small\sf System~\eqref{sys:6var}: CPU time needed on the average by the classical and generalized Newton methods to obtain a solution in less than 14 iterations, based on the data in Table~\ref{tbl2:6var}.}
\label{tbl3:6var}}
\end{table}

\subsection{A signal processing problem}
\label{ss:SPE}

In optimum broad-band antenna processing the minimization of the mean output power subject to linear constraints is a common problem \cite{TCL1993,TCL1996}. In \cite{TCL1996}, a 70-tuple example from \cite{TCL1993} about this signal processing problem has been transformed into the global minimization of the quartic polynomial function $\varphi:\dR^2\to\dR$ given in equation \eqref{eqn:otherTwoVar} below. The details of this transformation can be found in \cite[Appendix~C]{TCL1996}.

\begin{eqnarray}
\varphi(x) &=& a_1 - a_2\,x_1^2 + a_3\,x_1^4 - a_4\,x_1\,x_2 + a_5\,x_1^3\,x_2 - a_6\,x_2^2 + a_7x_1^2\,x_2^2 \nonumber \\[1mm]
&&\ \ \ \, +\ a_8\,x_1\,x_2^3 + a_9\,x_2^4\,, \label{eqn:otherTwoVar}
\end{eqnarray}
where
{\small
\[
\begin{array}{ll}
 a_1 = 0.337280011659804177\,,  &\qquad a_2 = 0.122071359035091510\,, \\[1mm]
 a_3 = 0.077257128600040819\,,  &\qquad a_4 = 0.217646697603541049\,, \\[1mm]
 a_5 = 0.233083387816363887\,,  &\qquad a_6 = 0.129244611969892874\,, \\[1mm]
 a_7 = 0.286227131697582205\,,  &\qquad a_8 = 0.1755719525003619673\,, \\[1mm]
 a_9 = 0.0567691913792773433\,.  & 
\end{array}
\]}
Here we are interested in the problem of finding a stationary point of $\varphi(x)$, namely a zero of
\begin{equation} \label{sys:sigprog}
f(x) := \nabla\varphi(x) = \left[\begin{array}{c}
- 2\,a_2\,x_1 + 4\,a_3\,x_1^3 - a_4\,x_2 + 3\,a_5\,x_1^2\,x_2 + 2\,a_7x_1\,x_2^2 + a_8\,x_2^3 \\[1mm]
a_4\,x_1 + a_5\,x_1^3 - 2\,a_6\,x_2 + 2\,a_7\,x_1^2\,x_2 + 3\,a_8\,x_1\,x_2^2 + 4\,a_9\,x_2^3
\end{array} \right] = {\bf 0}\,.
\end{equation}

The extremal points of $\varphi$, which are zeroes $f$, and the corresponding functional values are given in Table \ref{tbl:otherTwoVarSol} (also see \cite{TCL1996}).  Solutions 3 and 4 are the global minimizers of $\varphi$.

\begin{table}[H]
    \centering
    {\small
    \begin{tabular}{c|r|r|r}
        Soln & $x_1$\hspace*{15mm} & $x_2$\hspace*{15mm} & $\varphi(x)$\hspace*{12mm} \\ \hline
        1 & $-1.037925846421872$ & $1.188144940421522$ & $0.314501964946967$ \\
        2 & $1.037925846421872$ & $-1.188144940421522$ & $0.314501964946967$ \\
        3 & $-0.150370553810688$ &  $-0.948134491036906$ & $0.262292001977528$ \\
        4 & $0.150370553810688$ & $0.948134491036906$ & $0.262292001977528$\\
        5 & $0$ & $0$ & $0.337280011659804$
    \end{tabular}
    \caption{\small\sf Extremal points for the function \eqref{eqn:otherTwoVar}.}
    \label{tbl:otherTwoVarSol}}
\end{table}

Table~\ref{tbl1:sigproc} reconfirms that, for Solutions~1--4, convergence is quadratic and the numerically estimated values of $\lambda$ lie in the theoretical intervals found by using Theorem~\ref{thm:lambdaBounds}.  We observe that the sinh-generalized method has $\lambda$ consistently 1.2 to 2.7 times smaller than that of the classical method, and therefore locally faster by the same factors.

\begin{table}[H]
\centering
{\small
\begin{tabular}{ccccc}
Soln & $s_i(x)$ & $[\|\mu(x^*)\|/2,\ \|\rho(x^*)\|/2]$ & $\lambda$ & $\lambda_N/\lambda_{GN}$ \\[1mm] \hline \\[-4mm]
1 & $x_i$ & $[0.29,2.6]$ & $1.6$ &  \\
 & $x_i^3$ & $[0.14,2.2]$ & $1.7$ & $0.9$  \\
 & $\sinh x_i$ & $[0.24,2.4]$ & $0.6$ & $2.7$  \\
 & $e^{x_i}$ & $[0.23,2.3]$ & $0.4$ & $4$  \\[1mm] \hline \\[-4mm]
 
2 & $x_i$ & $[0.29,2.6]$ & $1.0$ &  \\
 & $x_i^3$ & $[0.14,2.2]$ & $1.5$ & $0.7$  \\
 & $\sinh x_i$ & $[0.24,2.4]$ & $0.6$ & $1.7$  \\
 & $e^{x_i}$ & $[0.34,3.0]$ & $1.5$ & $0.7$  \\[1mm] \hline \\[-4mm]
 
3 & $x_i$ & $[0,3.0]$ & $1.3$ &  \\
 & $x_i^3$ & $[0,4.9]$ & $4.5$ & $0.3$   \\
 & $\sinh x_i$ & $[0,2.8]$ & $1.1$ & $1.2$  \\
 & $e^{x_i}$ & $[0,3.7]$ & $2.0$ & $0.7$  \\[1mm] \hline \\[-4mm]
 
4 & $x_i$ & $[0,3.0]$ & $1.3$ &  \\
 & $x_i^3$ & $[0,4.9]$ & $0.4$ & $3.3$  \\
 & $\sinh x_i$ & $[0,2.8]$ & $1.1$ & $1.2$  \\
 & $e^{x_i}$ & $[0,2.5]$ & $1.0$ & $1.3$  \\[1mm] \hline \\[-4mm]
 
5 & $x_i$ & $-$ & $0.0$ &  \\
 & $x_i^3$ & $-$ & $-$ & $-$  \\
 & $\sinh x_i$ & $-$ & $0.0$ & $-$  \\
 & $e^{x_i}$ & $[0,0.71]$ & $0.5$ & $0.0$  \\[1mm] \hline \\[-4mm]
 
\end{tabular}
\caption{\small\sf System~\eqref{sys:sigprog}: Asymptotic error constants of the classical and generalized Newton methods.}
\label{tbl1:sigproc}}
\end{table}

The missing entries for Solution~5 in Table~\ref{tbl1:sigproc} warrants an explanation. Numerical experiments imply that the classical and sinh-generalized Newton methods' rate of convergence at $x = (0,0)$ is higher than quadratic, since the asymptotic error constants $\lambda$ of each method for quadratic convergence is estimated to be zero. In fact, interestingly, the rate of convergence for either method (numerically) turns out to be cubic, with the associated asymptotic error constants estimated as 1.0 and 1.6, respectively.  The exp-generalized method is the only method in the list which is verified to be quadratically convergent at Solution~5.  On the other hand, the cube-generalized method has a singularity at $x = (0, 0)$, and it fails to converge to the solution, no matter how close to the solution the initial guess is chosen.

A visualization of the global performances of the methods listed in Table~\ref{tbl1:sigproc} is provided in Figure~\ref{fig:sigproc}.  By looking at the graphs, the sinh-generalized method appears to be the best to use in the search domain $[-3,3]^2$ since it has smaller yellow regions and the domain is dominated more by shades of blue, a sign of quicker convergence. In the slightly bigger search domain of $[-10,10]^2$, however, while the cube-generalized method seems to have the largest regions of shades of blue, the classical method looks to have the smallest regions of yellow.  In view of the difficulty of judging from Figure~\ref{fig:sigproc} as to which method is preferable, we will resort to the statistical data presented in Tables~\ref{tbl2:sigproc} and \ref{tbl3:sigproc} for a more conclusive decision. 

\begin{landscape}
\begin{figure}
\centering
\begin{subfigure}{.36\textwidth}
\centering
\includegraphics[width=\textwidth]{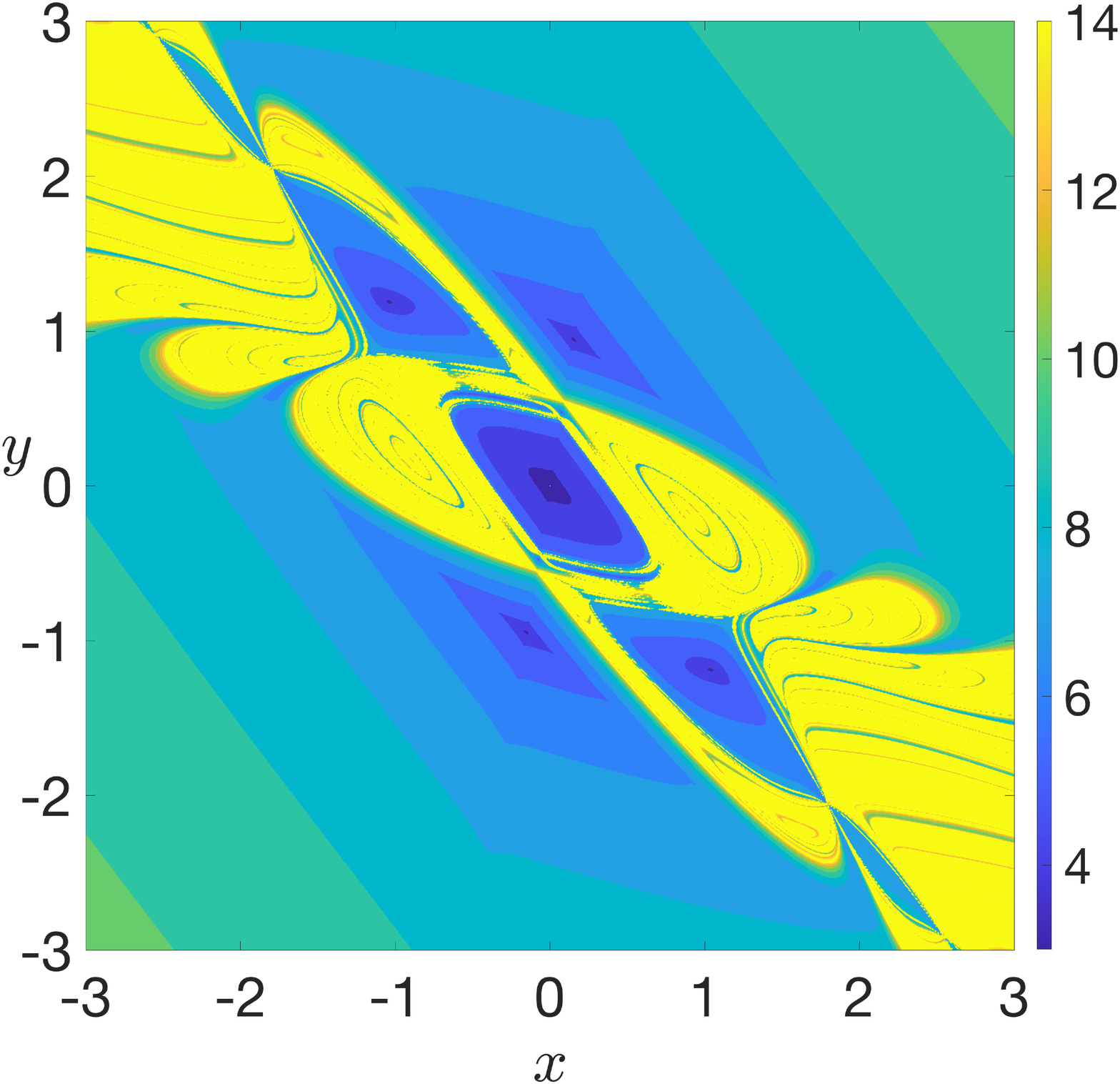}
\end{subfigure}
\begin{subfigure}{.36\textwidth}
\centering
\includegraphics[width=\textwidth]{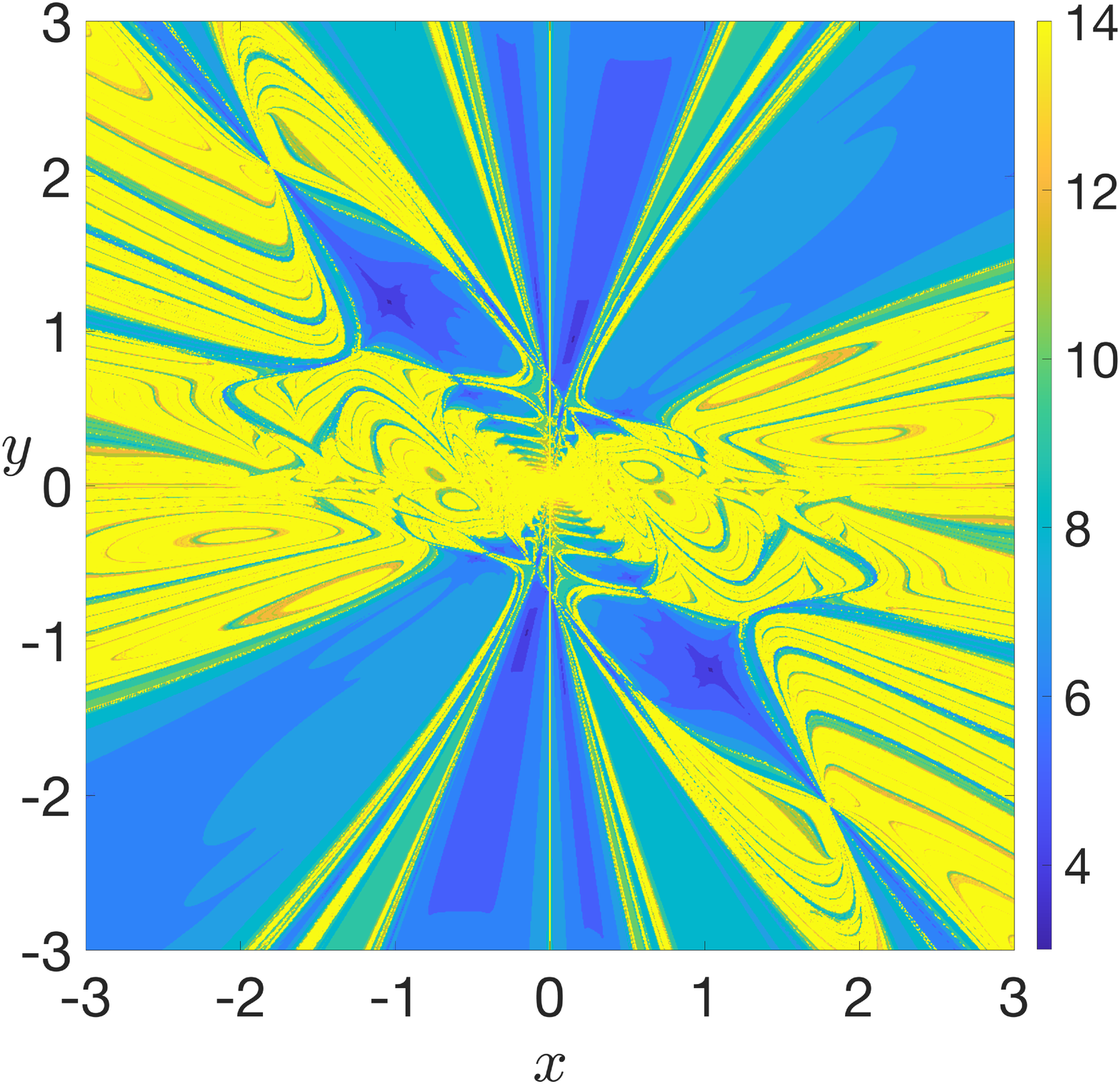}
\end{subfigure}
\begin{subfigure}{.36\textwidth}
\centering
\includegraphics[width=\textwidth]{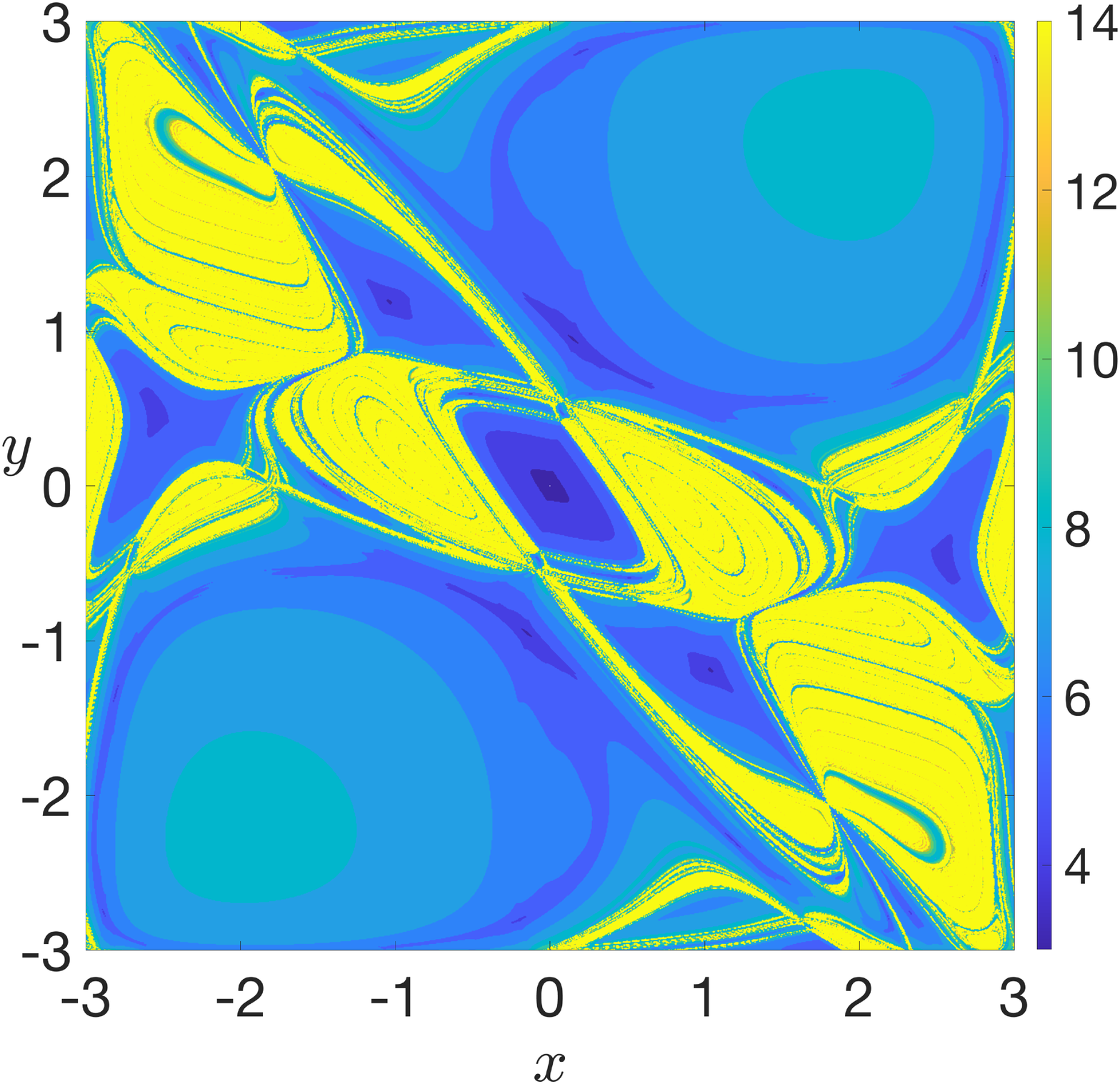}
\end{subfigure}
\begin{subfigure}{.36\textwidth}
\centering
\includegraphics[width=\textwidth]{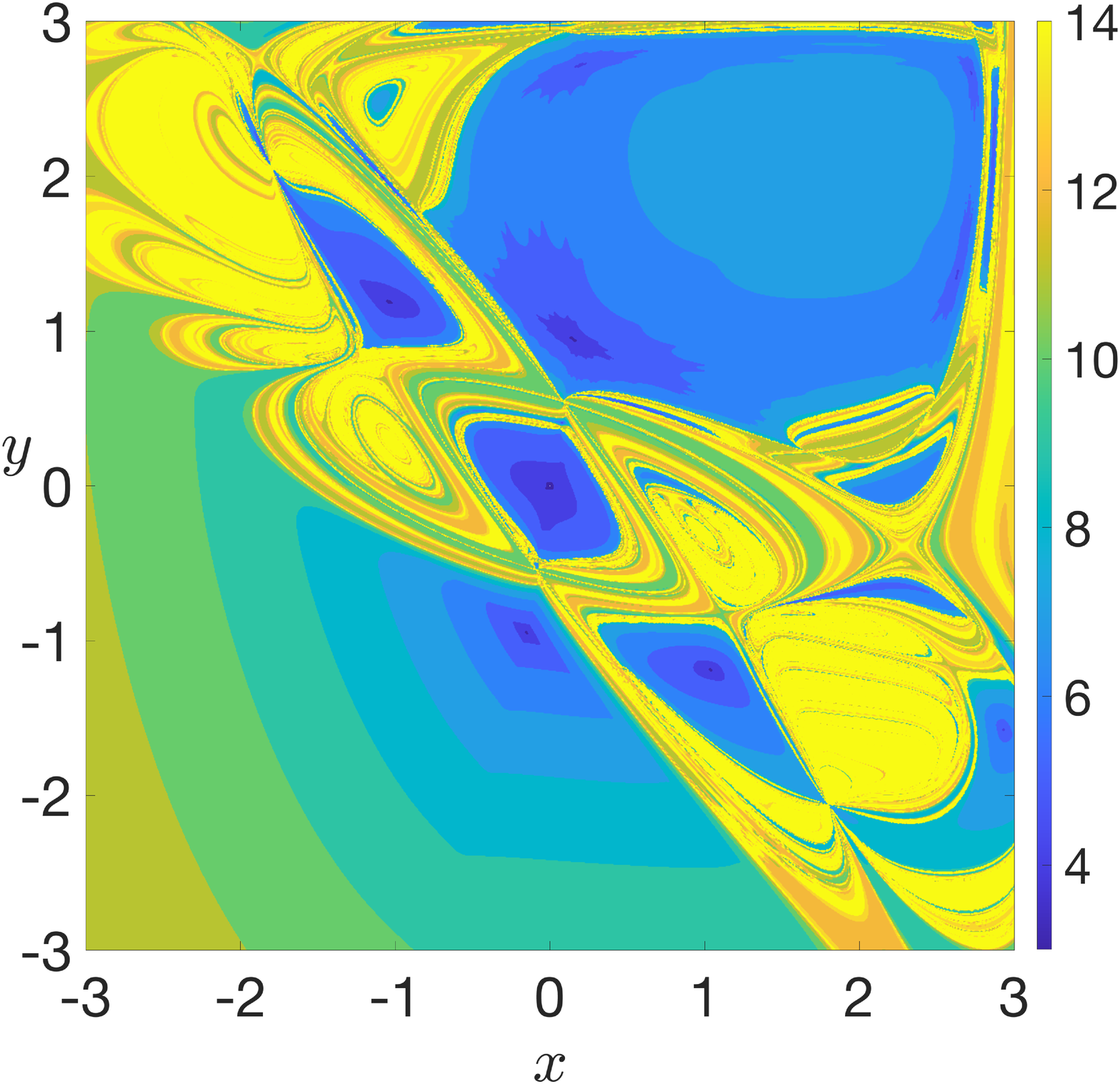}
\end{subfigure}
\begin{subfigure}{.36\textwidth}
\centering
\includegraphics[width=\textwidth]{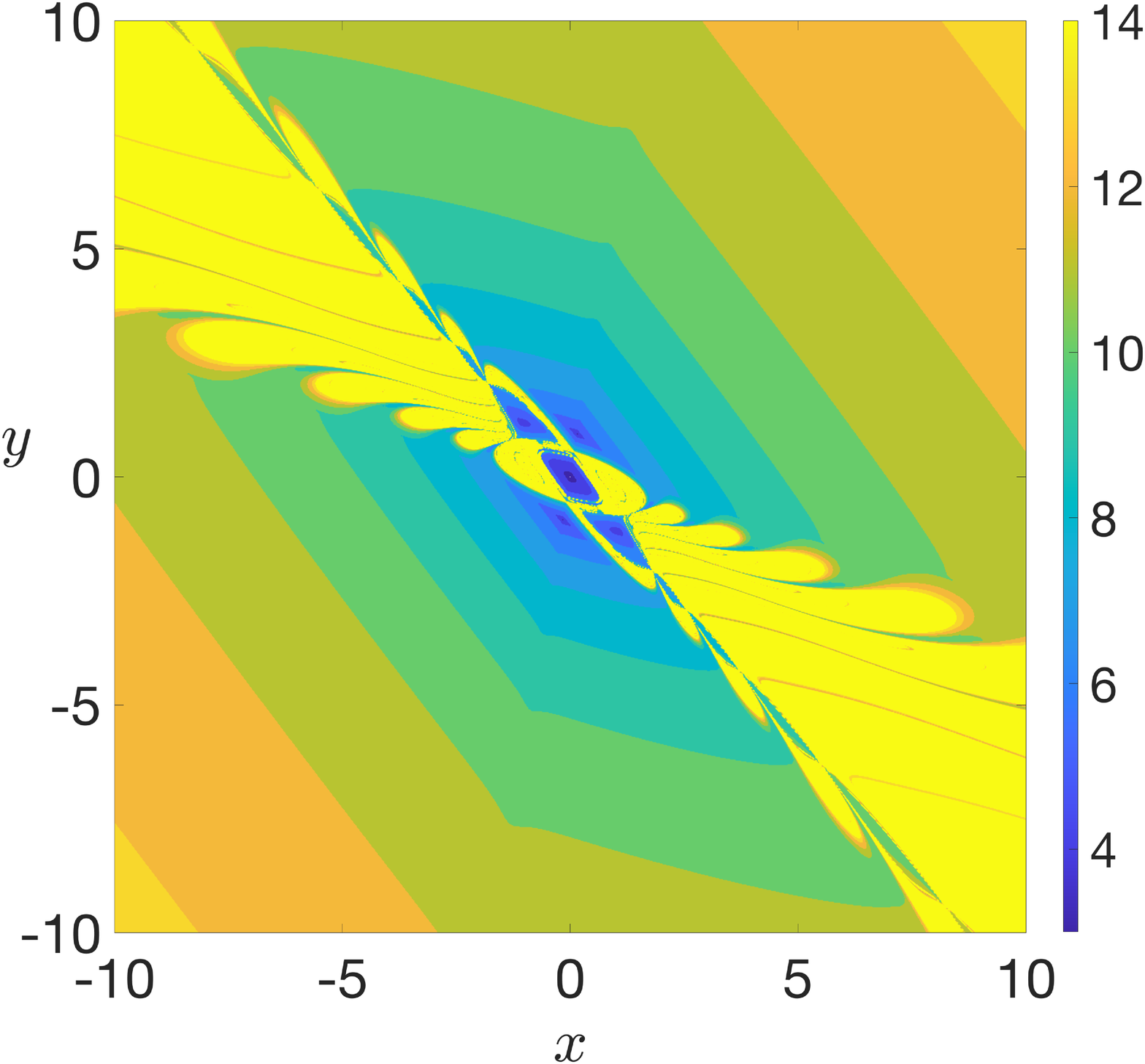}
\caption{$s(x) = (x_1, x_2)$}
\label{fig:sigproc_classic}
\end{subfigure}
\begin{subfigure}{.36\textwidth}
\centering
\includegraphics[width=\textwidth]{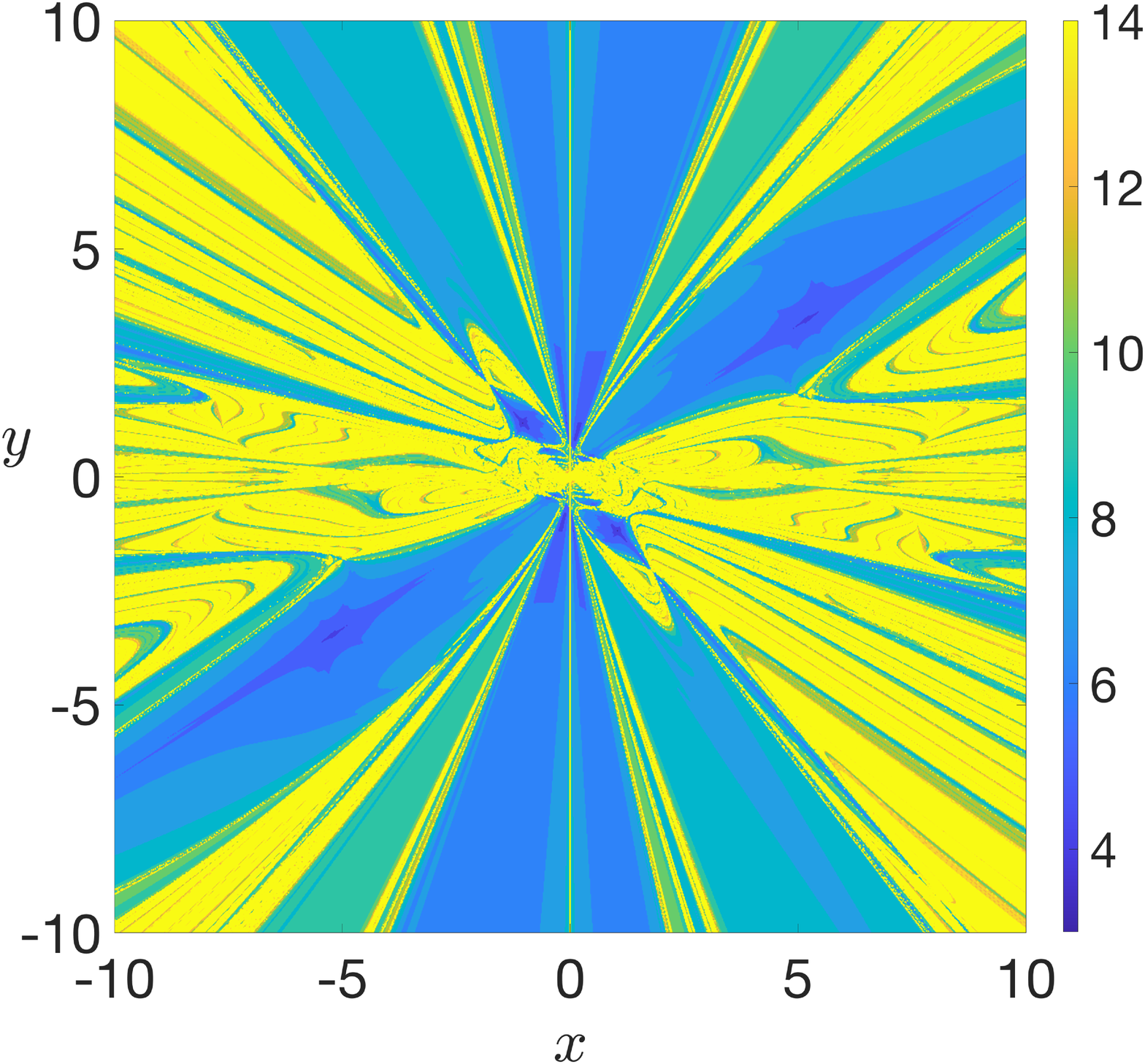}
\caption{$s(x) = (x_1^3, x_2^3)$}
\label{fig:sigproc_general1}
\end{subfigure}
\begin{subfigure}{.36\textwidth}
\centering
\includegraphics[width=\textwidth]{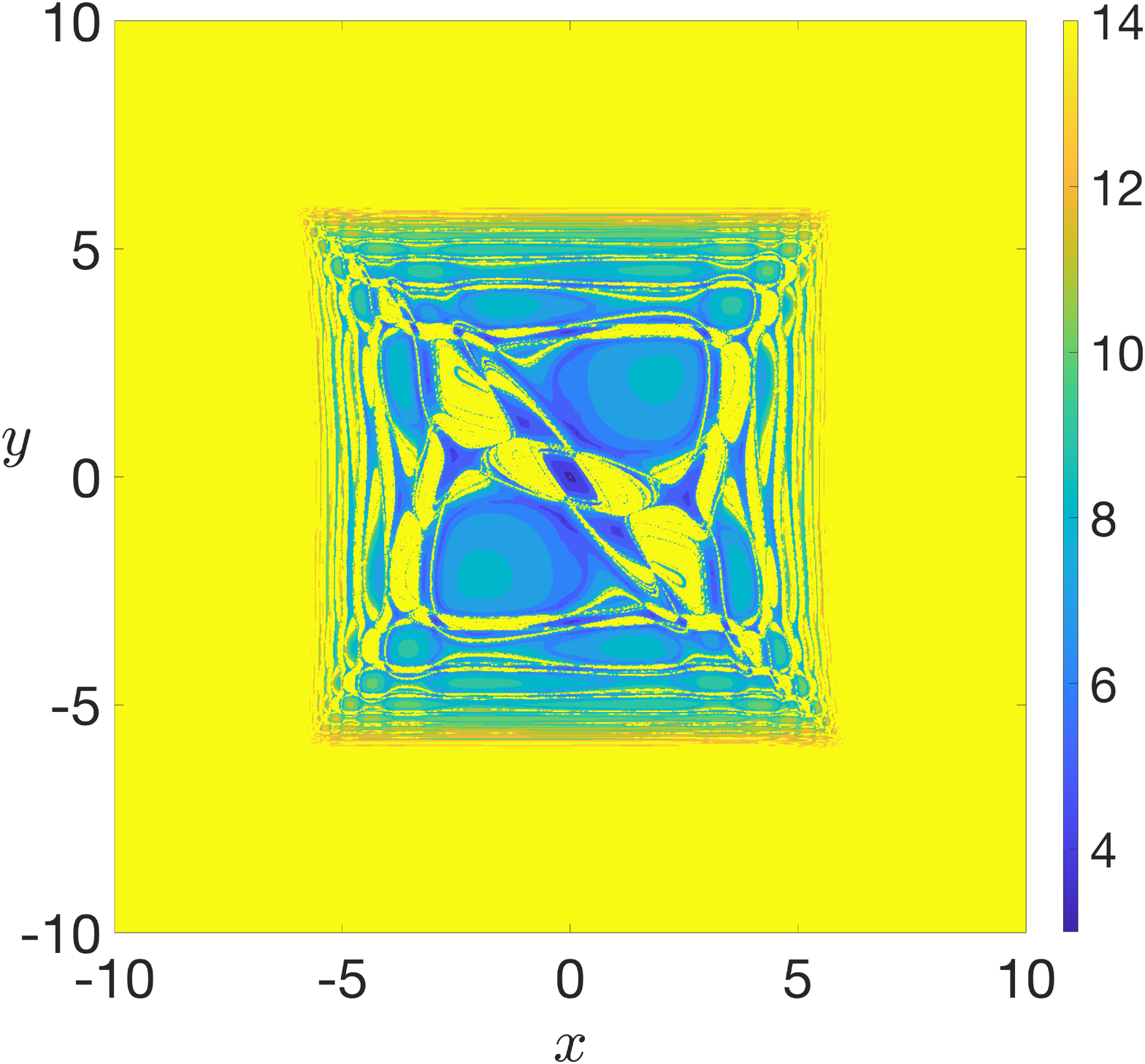}
\caption{$s(x) = (\sinh x_1, \sinh x_2)$}
\label{fig:sigproc_general2}
\end{subfigure}
\begin{subfigure}{.36\textwidth}
\centering
\includegraphics[width=\textwidth]{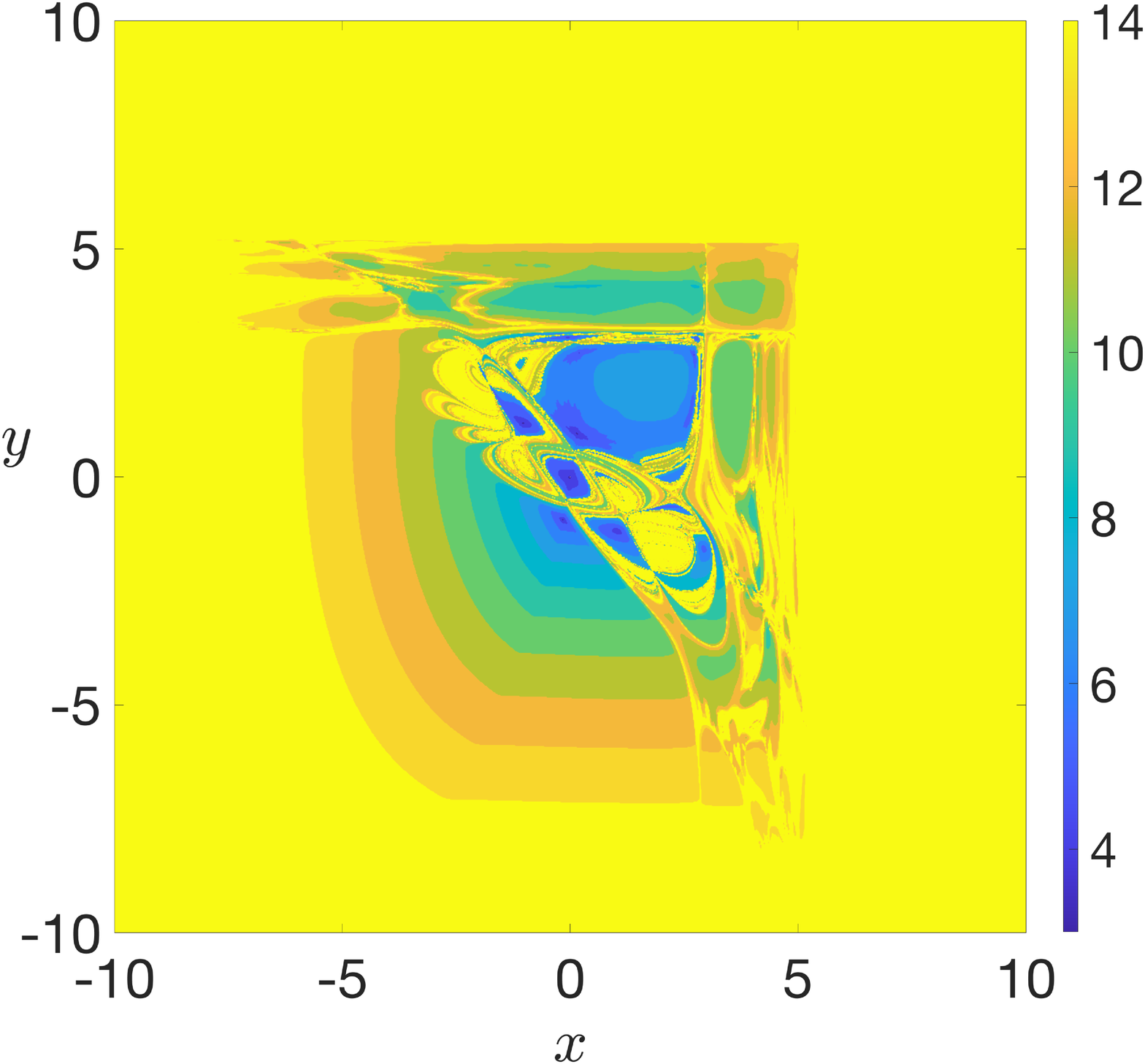}
\caption{$s(x) = (e^{x_1}, e^{x_2})$}
\label{fig:sigproc_general3}
\end{subfigure}
\caption{\small\sf System~\eqref{sys:sigprog}: Portraits of colour-coded number of iterations required for convergence.}
\label{fig:sigproc}
\end{figure}

\end{landscape}

As in the previous examples, we have run the classical and various generalized methods, with randomized initial points, to obtain the statistics in Tables~\ref{tbl2:sigproc}.  The only clear case for the choice of a method is when the search domain is $[-100,100]^2$, for which the cube-generalized method should certainly be the method of preference, given the relatively very high success rate, smaller number of iterations and not so much longer CPU time per iteration, all on the average.  To determine, ultimately, which of the methods will be preferable in the other search regions, we need to refer to Table~\ref{tbl3:sigproc}, as in the previous examples.

\begin{table}[H]
    \centering
{\small
\begin{tabular}{c|rr|rr|rr|c}
 & \multicolumn{2}{|c|}{$[-3,3]^2$} & \multicolumn{2}{c|}{$[-10,10]^2$} & \multicolumn{2}{c|}{$[-100,100]^2$} & CPU time/ \\[0.5mm]
\cline{2-3} \cline{4-5} \cline{6-7}
 & Ave & Success & Ave & Success & Ave & Success & successful \\
$s_i(x)$ & iter & rate [\%] & iter & rate [\%] & iter & rate [\%] & iter [sec] \\[1mm] \hline\\[-4mm]
$x_i$ & 7.8 & 80.1 & 10.5 & 81.1 & 12.2 & 4.2 & $3.7\times10^{-6}$ \\
$x_i^3$ & 7.8 & 68.6 & 8.1 & 69.7 & 8.7 & 67.3 & $6.2\times10^{-6}$ \\ 
$\sinh(x_i)$ & 6.9 & 78.5 & 8.4 & 25.0 & 8.3 & 0.2 & $3.9\times10^{-6}$ \\
$e^{x_i}$ & 8.6 & 81.4 & 10.9 & 27.6 & 10.9 & 0.3 & $5.1\times10^{-6}$ \\
$\tan{x_i}$ & 6.7 & 34.9 & 7.3 & 24.4 & 7.9 & 0.4 & $3.9\times10^{-6}$ \\[1mm] \hline
\end{tabular}
\caption{\small\sf System~\eqref{sys:sigprog}: Performance of the classical and generalized Newton methods with one million randomly generated starting points in domains of various sizes.}
\label{tbl2:sigproc}}
\end{table}

The boxed CPU times in Table~\ref{tbl3:sigproc} dictate that while the sinh-generalized method should be the method of choice in $[-3,3]^2$, the classical Newton method should better be used in $[-10,10]^2$.  We observe that the cube-generalized method is more than 700 times more efficient than its nearest contender, the classical method, in the large search domain $[-100,100]^2$.  This equivalently means that by the time the classical method finds a single solution,  the cube-generalized method will have obtained more than 13 solutions, on the average.

\begin{table}[H]
    \centering
{\small
\begin{tabular}{c|ccc}
 & \multicolumn{3}{c}{Time needed to get a single soln [sec]} \\[1mm]
\cline{2-4} \\[-4mm]
$s_i(x)$ & $[-3,3]^2$ & $[-10,10]^2$ & $[-100,100]^2$ \\[1mm] \hline\\[-4mm]
$x_i$ & $3.6\times10^{-5}$ & \framebox{$4.8\times10^{-5}$} & $1.1\times10^{-3}$ \\
$x_i^3$ & $7.0\times10^{-5}$ & $7.2\times10^{-5}$ & \framebox{$8.0\times10^{-5}$} \\ 
$\sinh(x_i)$ & \framebox{$3.4\times10^{-5}$} & $1.3\times10^{-4}$ & $1.6\times10^{-2}$ \\
$e^{x_i}$ & $5.4\times10^{-5}$ & $2.0\times10^{-4}$ & $1.9\times10^{-2}$ \\
$\tan{x_i}$ & $7.5\times10^{-5}$ & $1.2\times10^{-4}$ & $7.7\times10^{-3}$ \\[1mm] \hline
\end{tabular}
\caption{\small\sf System~\eqref{sys:sigprog}: CPU time needed on the average by the classical and generalized Newton methods to obtain a solution in less than 14 iterations, based on the data in Table~\ref{tbl2:sigproc}.}
\label{tbl3:sigproc}}
\end{table}

\section{Conclusion and Discussion}

We have proposed a family of generalized Newton methods facilitated by an auxiliary, or generalizing, function $s$, for solving systems of nonlinear equations.  The method reduces to the classical Newton method if the generalizing function is the identity map, i.e., $s(x) = x$.  Under mild assumptions, we have proved that the new family of methods are quadratically convergent just like the classical one.  We derived expressions for the bounds on the asymptotic error constants of the family.  These bounds, which can be computed for practical problems easily as illustrated in the numerical experiments, can provide an idea about the relative local speeds of the classical and generalized methods, although they are not tight.

For numerical experimentation, we have considered three types of problems, namely systems of equations involving quartic (in two variables), cubic (in two and six variables), and exponential (in two variables), functions.  We carried out extensive numerical experiments using $s(x) = x$ (the classical method), and $s(x) = x^3$, $\sinh x$, $e^x$, and $\tan x$ (the cube-, sinh-, exp- and tan-generalized methods, respectively), for each of the example problems.

For the problems in two variables, we constructed graphs depicting, for search domains of various sizes, a portrait of the colour-coded number of iterations a method would need to converge to a solution.  These graphs, or portraits, were observed to provide a broad idea as to which method is likely to be more preferable. Using one million randomly generated initial points in each chosen search domain, we presented tables reporting the success rate and average number of iterations of a method as well as the average CPU time one iteration of that method takes.  By using this data, we were able to identify the method with the smallest CPU time required for finding a single solution, as the preferred method in a particular search domain.

For the cubic and quartic problems we have considered, the sinh-generalized method seems to be particularly successful in relatively smaller search domains.  In slightly larger domains where the  sinh-generalized method is not so successful anymore, the classical Newton method looks like the method of preference for the two-variable problems.  For very large domains, the cube-generalized method certainly looks to be the method of choice, if not the only successful method for some problems.  We found that for the exponential problem we studied, the exp-generalized method seems to be the only method one should use in domains of any size.

The kind of numerical exploration we performed could be particularly useful in the case when a system of equations involves parameters and these parameters change only slightly so that the portraits and the statistical data are not altered much.  In other words, given a system of equations
\begin{equation}\label{sys_param}
    f(x,p)={\bf 0}\,,
\end{equation}
with a vector of $n$ unknowns, $x\in \dR^n$, and a fixed vector of $m$ parameters, $p\in\dR^m$, the task would be to find a solution of \eqref{sys_param} as efficiently as possible. Suppose that we have identified the preferred generalized method through the numerical exploration we have devised in this paper.  When $p$ has changed slightly, the preferred method might then be employed to find a new solution of \eqref{sys_param}.

We have demonstrated that a suitable choice of $s$ is possible for some specific forms of systems of equations. Making informed choices of $s$ for more general problems still stands as a challenging research problem.  Like most available modifications on Newton method, our generalized version may switch to the classical one (i.e., with $s(x)=x$) or to another generalized method when the current choice of $s$ is either inconvenient or provides no detectable advantage. One may refer to this version as a {\em hybrid} generalized Newton method.  For example, when solving the cubic problems, based on the results in the tables with CPU times to get a single solution, it might be interesting to devise a hybrid method which switches from the classical or the cube-generalized method to the sinh-generalized method as iterations fall into the search domain $[-3,3]^2$.

In the future, it would be valuable to study quadratic convergence regions as it was done in \cite{Kantorovich1948, Smale1986} by Kantorovich and Smale, for the new generalized methods.  It would also be interesting to consider extending the work presented in this paper to situations where global convergence to a solution is guaranteed, such as the Levenberg--Marquardt approach~\cite{Levenberg, Marquardt}.


\end{document}